\documentclass[a4paper,12pt,reqno]{amsart}
\usepackage{latexsym, amsfonts, amsmath, amsthm, amssymb, enumerate, verbatim}
\usepackage[top=3cm, bottom=3cm, left=2.54cm, right=2.54cm]{geometry}

\newtheorem{thm}{Theorem}[section]
\newtheorem*{thm*}{Theorem}
\newtheorem{cor}[thm]{Corollary}
\newtheorem*{cor*}{Corollary}
\newtheorem{lem}[thm]{Lemma}
\newtheorem{prop}[thm]{Proposition}

\newtheorem*{con*}{Conjecture}
\newtheorem*{prob*}{Problem}

\theoremstyle{definition}
\newtheorem{defn}[thm]{Definition}
\theoremstyle{remark}
\newtheorem{rem}[thm]{Remark}

\numberwithin{equation}{section}

\begin{document}
\title{Convexity and transport for isentropic Euler equations on a sphere}

\author{Gordon Blower}

\address{Department of Mathematics and Statistics, Lancaster University, Lancaster, LA14YF, United Kingdom}
 \email{g.blower@lancaster.ac.uk}


\date{21st November 2020}
\begin{abstract}
The paper considers the Euler system of PDE on a smooth compact Riemannian manifold of positive curvature without boundary, and the sphere ${\mathbb{S}}^2$ in particular. The paper interprets the Euler equations as a transport problem for the fluid density under dynamics governed by the gradient of the internal energy of the fluid. The paper develops the notion of transport cost in the tangent bundle, and compares its properties with the Wasserstein transportation cost on the manifold. There are applications to the discrete approximation to the Euler equations in the style of Gangbo and Wesdickenberg ({\sl Comm. Partial Diff. Equations} {\bf 34} (2009), 1041-1073), except that the analysis is heavily dependent upon the curvature of the underlying manifold. The internal energy is assumed to satisfy convexity conditions that allow analysis via $\Phi$-entropy entropy-production inequalities, and the results apply to the power law $\rho^\gamma$ where $1<\gamma<3/2$, which includes the case of a diatomic gas. The paper proves existence of weak solutions of the continuity equation, and gives a sufficient condition for existence of weak solutions to the acceleration equation.
\end{abstract}

\maketitle

\section{Introduction}

\vskip.05in
\noindent Let $M$ be a connected, compact and smooth Riemannian manifold without boundary 
with measure $m$ and distance $d$. Suppose that $\Theta: (0, \infty )\rightarrow (0, \infty )$ has $\Theta (e^{r})$ strictly increasing and convex for $r\in {\mathbb{R}}$; let $\Theta_1(r)=r\Theta'(r)+\Theta (r)$.
Consider the energy
\begin{align}{\mathcal H}(\rho ,q)=\int_M\Bigl( {\frac{1}{2}}\bigl\Vert\nabla q (x)\bigr\Vert^2 +\Theta (\rho (x))\Bigr)\rho (x)\, m(dx)\end{align}
in the variables $(\rho ,q)$, where $\rho$ is a probability density function on ${\bf M}$, and $q:{\mathbb{R}}\rightarrow {\mathbb{R}}$ is a continuously differentiable function. Here $\nabla$ denotes the distributional gradient, always in the space variable $x$. There is a natural scalar product on each fibre of the tangent bundle, and we abbreviate $\Vert \nabla q(x)\Vert_{T_xM}$ by $\Vert \nabla q(x)\Vert$; let $H^1(M)=\{ q\in L^2_{loc}(M):\Vert \nabla q(x)\Vert\in L^2(M;m; {\mathbb{R}})\}$.\par 
\indent The canonical equations of motion are
\begin{align}\label{canonical}{\frac{\partial \rho}{\partial t}}&={\frac{\partial {\mathcal H}}{\partial q}}=-\nabla\cdot \bigl( (\nabla q)\rho\bigr)\nonumber\\
{\frac{\partial q}{\partial t}}&=-{\frac{\partial {\mathcal H}}{\partial \rho }}=-{\frac{1}{2}}\Vert\nabla q\Vert^2-\Theta_1(\rho )\end{align}
but we seek solutions that are not necessarily classical, and can have shocks. We seek solutions in which the mass is conserved, so $\int_M \rho (x,t)m(dx)=1$, hence the continuity equation 
\begin{align}\label{continuity}{\frac{\partial \rho }{\partial t}}+\nabla\cdot \bigl( \rho \vec v\bigr)=0,\end{align}
holds weakly, where we take $\vec v=\nabla q,$ where we regard $\vec v$ as the velocity field. Also we require dissipation of energy
\begin{align}\label{dissipation}{\frac{d}{dt}}{\mathcal H}(\rho ,q)\leq 0,\end{align}
which in the presence of (\ref{continuity}) amounts to the differential inequality
\begin{align}\label{dissipationdiff}\rho (x,t)\vec v(x,t)\cdot \Bigl( {\frac{\partial \vec v}{\partial t}}+\vec v\cdot \nabla \vec v+\nabla \bigl(\Theta_1\circ \rho  \bigr)\Bigr)\leq 0.\end{align}
We write $d/dt={{\partial }/{\partial t}}+\vec v\cdot \nabla$ for the advection operator. Equality holds in this inequality when the velocity satisfies the dynamical equation 
\begin{align}\label{acceleration}{\frac{\partial \vec v}{\partial t}}+\vec v\cdot \nabla \vec v=-\nabla \bigl(\Theta_1\circ \rho \bigr).\end{align} 
\indent The paper \cite{bib39} shows that the Cauchy problem for the isentropic Euler equation with $\Theta (\rho )=\rho$ is ill-posed for space variables in ${\mathbb{R}}^2$. Our results are concerned with existence of solutions for a particular algorithm based upon transportation. Gangbo and Westickenberg \cite{bib17} considered this problem in the context of Euclidean space, and established that a weak solution of (\ref{continuity}) exists, by a time discretization process based upon Otto calculus. The current paper uses many of the ideas from \cite{bib17} and \cite{bib30}, but some aspects of the problem are different due to the change in the geometry. First, if the density is constant, then the system (\ref{continuity}) and (\ref{acceleration}) has only a static solution. Indeed, with constant density $\rho (x)=1/m(M)$ and $\vec v=\nabla q$, we have $\nabla\cdot q=0$, so $q$ is constant by Hopf's lemma and $\vec v=0$.\par 
\indent We consider the continuity equation (\ref{continuity}) as a flow $t\mapsto \rho (\cdot , t)$ of probability density functions in Wasserstein space, introduced as follows; see \cite{bib28}, \cite{bib29}.\par 
\vskip.05in

\begin{defn} (i) The Wasserstein space ${\mathcal W}^p(M)$ consists of the set of Radon  probability measures on $M$ with the Wasserstein distance $W_p(\mu, \nu )$, where for $1\leq p<\infty$ and  $\mu,\nu\in{\mathcal W}^p(M)$ we define
\begin{align}\label{W2}W_p(\mu , \nu )^p=\inf_\pi\Biggl\{ {\frac{1}{p}}\int\!\!\!\int_{M\times M} d(x,y)^p\pi (dxdy): \pi \in {\hbox{Prob}}(M\times M)\Biggr\}\end{align}
where $\pi$ has marginals $\mu$ and $\nu$, so $\pi$ is called a transport plan taking $\mu$ to $\nu$. For $W_2$, the cost function $d^2/2: M\times M\rightarrow [0, \infty )$ is $d(x,y)^2/2$.  If $\mu, \nu$ are absolutely continuous with respect to $m$, with $\mu (dx)=p_0(x)m(dx)$ and $\nu(dx)=p_1(x)m(dx)$ for probability density functions $p_1,p_0\in L^1(M)$, then we write $W_2(p_0,p_1)=W_2(\mu, \nu)$.\par
\indent (ii) Given a continuous $\varphi :M\rightarrow M$, and $\mu\in {\hbox{Prob}}(M)$, there exists a unique $\nu\in {\hbox{Prob}}(M)$ such that $\int_M g(x)\nu (dx)=\int_M g\circ \varphi (x)\mu (dx)$ for all continuous $g:M\rightarrow {\mathbb{R}}$; we say that $\varphi$ induces $\nu $ from $\mu$, and write $\nu=\varphi \sharp\mu$.  If $\mu, \nu$ are absolutely continuous with respect to $m$, with $\mu (dx)=p_0(x)m(dx)$ and $\nu(dx)=p_1(x)m(dx)$ for probability density functions $p_1,p_0\in L^1(M)$, then we write $p_1=\varphi\sharp p_0$.\par
\indent (iii)  A function $[0, \tau] \rightarrow {\mathcal W}^p(M)$ is $2$-absolutely continuous if there exists $M_2$ such that 
$$\sum_{j=1}^\infty {\frac{W_p^2(\rho_{b_j}, \rho_{a_j})}{b_j-a_j}}\leq M_2$$
for all sequences $((a_j,b_j))_{j=1}^\infty$ of pairwise disjoint subintervals $(a_j,b_j)$ of $[0,\tau]$.\end{defn}

Then 
${\mathcal W}^2(M)$ is a complete and separable metric space for $W_2$. When constructing solutions to (\ref{continuity}), we aim to have $t\mapsto \rho_t:$  $2$-absolutely continuous Lipschitz continuous for the $W_1$ metric. This rules out the possibility of a classical solution $(\rho_t, q_t)$ suddenly changing to the static solution $(1/m(M), 0).$ 

\indent In this paper, we consider convexity on the space of probability density functions in three different senses:\par
\indent (i) $L^1(M)$ is a real linear space, and the probability density functions form a convex cone under the pointwise operation $(1-s)\rho_0(x)+s\rho_1(x)$ for $s\in [0,1]$ and $x\in M$, as in \cite{bib26}.\par
\indent (ii) ${\mathcal W}^2(M)$ is a length space for the metric $W_2$, so we can consider a geodesic $\rho_tdm$ in  ${\mathcal W}^2(M)$ joining $\rho_0dm$ to $\rho_1dm$. See \cite{bib29}.\par
\indent (iii) There are generalized geodesics, namely continuous curves $\rho_tdm$ in  ${\mathcal W}^2(M)$ joining $\rho_0dm$ to $\rho_1dm$; in particular, we introduce these via Jacobi fields on $M$. 
 Generalized geodesics on Hilbert space are considered in \cite{bib2}.\par
\indent  In the context of manifolds of positive curvature, one realizes the full significance of these different notions of convexity. Sections \ref{Moving} and \ref{Convexity} discuss this in more detail, using some fundamental results of McCann \cite{bib25}. \par
  
\begin{defn} A function $\Phi :M\rightarrow {\mathbb{R}}$ has $d^2/2$-transform    
\begin{align}\Phi^c(x)=\inf\bigl\{ d^2(x,y)/2-\Phi (y): y\in M\bigr\}.\end{align}
If $\Psi :M\rightarrow {\mathbb{R}}\cup \{ -\infty\}$ is any function that is not identically $-\infty$ and arises as $\Psi (x)=\Phi^c(x)$ for all $x\in X$ and some such $\Phi$, we say that $\Psi$ is $d^2/2$-concave. Such a $\Psi$ is upper semicontinuous. \end{defn}
\indent If $\Phi : M\rightarrow {\mathbb{R}}$ is twice continuously differentiable in the classical sense, then $\Phi$ has gradient $\nabla \Phi$ and Hessian $D^2\Phi$. We recall that if $\gamma$ is a geodesic emanating from $x\in M$ with $\gamma'(0)=\xi$, then the Hessian is the symmetric endomorphism of $T_xM$ such that $\langle D^2\Phi (x)\xi,\xi\rangle =(d^2/dt^2)_{t=0}\Phi (\gamma (t)).$
Now let $\Phi=\Phi^{cc}: M\rightarrow {\mathbb{R}}\cup\{ \infty \}$ be an infimal convolution that is not identically infinite. McCann \cite{bib25} showed that $\Phi$ is Lipschitz continuous throughout $M$. Furthermore, there exists a subset $Z$ on $M$ that has zero volume such that $\Phi$ is differentiable on $M\setminus Z$, and $\nabla \Phi: M\setminus Z\rightarrow TM$ gives the gradient of $\Phi$. By Theorem 14.25 of \cite{bib29}, $\Phi$ also has a Hessian second derivative $D^2\Phi$, in the sense of Alexandrov, which is essentially a matrix of measures. The differentiability properties of $d^2/2$ concave functions were also considered by Cabre \cite{bib6}, who computed the Hessian of $d^2(x,y)$ and found the Jacobians of various changes of variables which we will use later in this paper. For the moment, given $F:M\rightarrow M$, we take $T_xM={\mathbb{R}}^n$ and consider
$${\mathcal C}_{d^2/2}=\bigl\{ A: M\rightarrow M_n({\mathbb{R}}): D_x^2 (1/2) d^2(x,y)\vert_{y=F(x)}+ A(x)\geq 0\bigr\}$$
where $\geq 0$ means positive semidefinite. Note that $0\in {\mathcal C}_{d^2/2}$ and ${\mathcal C}_{d^2/2}$ is a cone in the sense that $(1-s)A_0+sA_1\in {\mathcal C}_{d^2/2}$
whenever $A_0,A_1\in {\mathcal C}_{d^2/2}$ and $s\in [0,1]$. The condition that $-\varphi$ is $d^2/2$-concave amounts to $D^2\varphi\in {\mathcal C}_{d^2/2}$ for 
$F(x)=\exp_x\nabla \varphi (x)$.\par

\begin{defn} The internal energy for a probability density function $\rho$ is
\begin{align}{\mathcal U}(\rho )=\int_M \Theta (\rho (x))\rho(x)m(dx),\end{align}
and with  $\Theta_1 (\rho )=\rho \Theta'(\rho )+\Theta (\rho ), $ the pressure is $p(\rho ) =\rho \Theta_1 (\rho)-\rho\Theta (\rho )=\rho^2\Theta' (\rho ).$\end{defn}

\indent We introduce the Lagrangian version of the differential equations and  consider the evolution of data consisting of a  probability density function $\rho_0$ on $M$ and a function $q_0:M\rightarrow {\mathbb{R}}$ such that $-q_0$ is $d^2/2$-concave. Initially, we suppose ${\mathcal H}(\rho_0, q_0)<\infty$, and we introduce $V(x,0)=\nabla q_0(x)$, to be regarded as a velocity $V(\cdot , 0): M\rightarrow TM$ so $V(x,0)\in T_xM$. 
Let $\tau >0$ and suppose that $M$ has dimension $n$ and that $TM$ is a complete submanifold of ${\mathbb{R}}^{2n}$; here $[x;v]\in {\mathbb{R}}^{2n}$ denotes a column vector with entries $x\in {\mathbb{R}}^{n}$ above $v\in
{\mathbb{R}}^{n}$. Suppose further that $[X;V]\rightarrow [V;g(X,t)]$ is Lipschitz continuous $TM\rightarrow{\mathbb{R}}^{2n}$ with Lipschitz constant $L$ for all $t\in [0, \tau ]$, and $V,g(X,t)\in T_XM$. Then the ordinary differential equation 
\begin{align}\label{ODE}{\frac{d}{dt}}\begin{bmatrix}X\cr V\end{bmatrix}  =\begin{bmatrix}V\cr g(X,t)\end{bmatrix}\end{align}
with initial condition
\begin{align}\begin{bmatrix}X(x,0)\cr V(x,0)\end{bmatrix}=\begin{bmatrix}x\cr \nabla q_0(x)\end{bmatrix}\end{align}
has a unique solution $[X(x,t);V(x,t)]$ with $V(X,t)\in T_{X(x,t)}$, such that $[x;v]\mapsto [X(x,v,t); V(x,v,t)]$ is Lipschitz continuous $TM\rightarrow TM\subset {\mathbb{R}}^{2n}$.  
 A solution is to be interpreted as follows. We choose and fix $v=\nabla q_0(x)\in T_xM$, write $X(x,t)=X(x,v(x),t)$ and consider this  function $X: M\times [0, \tau ]\rightarrow M$ such that $X(\cdot , t)$ satisfies $X(x,0)=x$, so that $X(x,t)$ follows the trajectory of the particle that starts at $x$, and $V(x,t)\in T_{X(x,t)}M$ is the velocity of the particle at time $t$. We do not assert that $x\mapsto X(x,t)$ is injective for $t>0$, so it is possible that trajectories cross, and that there are conflicting values for $V(X(x,t),t)$ arising at the point at which the trajectories cross.\par  
\indent Nevertheless,  $x\mapsto X(x,t)\in M$ gives a Borel measurable function, hence  for every $\mu\in {\hbox{Prob}}(M)$, there exists $\nu_t\in {\hbox{Prob}}(TM)$ such that  
$$\int_{TM} f(X,V)\nu_t(dXdV)=\int_M f(X(x,t), V(x,t))\mu (dx)\qquad (f\in C_b(TM; {\mathbb{R}}))$$
so $\nu_t$ is induced from $\mu$. Let $\mu_t$ be the marginal of $\nu_t$ on $M$, induced by $TM\rightarrow M: (X,V)\mapsto X$. 

In particular,  if $d\mu =\rho_0dm$ and $\mu_t<<m$, then $d\mu_t=\rho_td\mu$,  and $X(\cdot , t)$ induces $\rho_t $ from $\rho_0$ (usually not optimally). We can also write 
$\nu_t(dx)=\vec v(x,t)\mu_t(dx)$, where $\vec v(x,t)$ is to be interpreted as the Eulerian velocity field. \par
\indent In order to solve the Euler equations on ${\mathbb{S}}^2$, we wish to solve (\ref{ODE}) in the case in which 
$$g(X(x,t),t)=-X-\nabla (\Theta_1\circ \rho_t)(X(x,t)).$$ 
\noindent For $M={\mathbb{S}}^2$, we observe that $[X;V]\in T{\mathbb{S}}^2$ if and only if $F=[X;V; X\times V]$ has $FF^T$ a diagonal matrix, so in section \ref{Moving} we are able to interpret the differential equation in terms of moving frames.  Using Loeper's regularity theory for optimal transport on ${\mathbb{S}}^2$, we are able to show that $x\mapsto \rho_t (x)$ is Lipschitz, so $g(X,t)$ is locally $L^\infty$; however,
we have not established Lipschitz continuity of $X\mapsto g(X,t)$, so we need to incorporate smoothing in the space variable into the approximation process. See \cite{bib30}, page 14.\par

\indent 
 Generally, (\ref{ODE}) has  a discrete approximation
$$\begin{bmatrix} X_{(n+1)h}\cr V_{(n+1)h}\end{bmatrix}=\begin{bmatrix} \exp_{X_{nh}}(h (V_{nh}+V_{(n+1)h})/2)\cr V_{nh}+hg(X_{(n+1)h})\end{bmatrix}.$$
\noindent To solve the discrete version, we consider a three-step algorithm which has initial data a probability density function $\rho_0$ on $M$ and a function $q_0$ such that $-q_0$ is $d^2/2$-concave. \par
\indent {\bf Stage 1.} The first step starts at $[X_0; V_0]$, so the pair $[X_0;V_0]=[x;\vec v]$ gives a section of the tangent bundle $TM\rightarrow M$ such that $\vec v\in T_xM$.
and one proceeds along the geodesic $\gamma_{X_0}(t)=\exp_{X_0} (t V_0)$ at constant speed $V_0=\nabla q_0 (X_0))$ to $\hat X_0=\exp_{X_0}(h V_0)$. The idea is that $X_0\mapsto \hat X_0$ takes $\rho_0$ to $f_h$. \par
\indent  In section \ref{Moving}, we consider cost functions on $TM$ to measure the cost of taking $(X_0,V_0)$ to $(X_h,V_h)$. Acceleration costs for ${\mathbb{R}}^n$ have been considered previously by Gangbo, Westdickenberg and Wilkening in \cite{bib17}, \cite{bib30}. For manifolds such as the sphere ${\mathbb{S}}^2$ in ${\mathbb{R}}^3$, curvature plays an important role in the cost on $TM$, and in sections \ref{Moving}, \ref{Convexity} we provide explicit computations for this example.\par

\indent {\bf Stage 2.} The second step refines the initial choice of density $f_h$. Consider the energy functional
\begin{align}{\mathcal E}(\rho ;f_h )= W_2^2(f_h, \rho )+h^2 \int_M \rho (x)\Theta (\rho (x))m(dx),\end{align}
Suppose that $\rho_0$ is a density such that $x\mapsto \exp_x (\nabla q_0(x))$ induces $f_h$ from $\rho_0$. The speed along the  geodesic $t\mapsto \exp (t\nabla q_0(x))$ is constant, so 
${\mathcal H}(\rho_0, q_0)$ 
satisfies ${\mathcal E}(\rho_0;f_h)\leq {\mathcal H}(\rho_0, q_0) $.  Consider $K={\mathcal H}(\rho_0, q_0)$, and 
\begin{align}\label{Omega}\Omega_K=\bigl\{ \rho\in{\hbox{Prob}}(M): \rho<<m, {\mathcal E}(\rho;f_h )\leq K\bigr\}\end{align}  
so that $\rho_0\in \Omega_K$.  In Proposition \ref{proposition5}, we show that $\Omega_K$ is a convex and weakly sequentially compact subset of $L^1(M)$, and  ${\mathcal E}(\,\cdot\, ;f_h)$ is a lower semi continuous functional on $\Omega_K$, so there exists $\rho_h\in \Omega_K$ such that
\begin{align}\label{Evariation}{\mathcal E}(\rho_h;f_h )=\inf_\rho \{ {\mathcal E}(\rho ;f_h):\rho\in \Omega_K\}.\end{align}
Given $\rho_h$, we can choose map $\hat X_0 \mapsto X_{h}$ that induces $\rho_h$ from $f_h$; 
thus we can solve the implicit equation for $X_h$.\par 

\indent For $T:M\rightarrow M$ and a one-parameter family of diffeomorphism $\varphi_t:M\rightarrow M$, we consider
$${\frac{1}{2}}\int_M d(\varphi_t\circ T(x), x)^2f(x)m(dx)+\int_M \Theta\bigl( \varphi_t\circ T\sharp f (x)\bigr) \varphi_t\circ T\sharp f(x)m(dx).$$
and establish conditions under which this is a convex function of $t$. We compute the first derivative with respect to $t$ of this expression, and use this to locate the minimizer $\rho_h$ of the energy ${\mathcal E}(\rho; f)$. For ${\mathbb{ S}}^2$, we compute the second derivative with respect to $t$ explicitly. To control the derivatives of the internal energy along generalized geodesics,  we use a generalized Fisher information functional from \cite{bib34}, and incorporate it into functional inequalities such as Proposition \ref{proposition8}. This extends the analysis of \cite{bib17} and \cite{bib30}, which involved power laws.\par

\begin{defn} \cite{bib34} Let $\Phi: [0, \infty )\rightarrow [0, \infty )$ be a $C^4$  convex function, and $\mu$ a probability measure on $M$.\par
\indent (i)  Then for a probability density function $f\in L^1(M;\mu )$, the $\Phi$-relative entropy of $fd\mu$ with respect to $d\mu $ is
\begin{align}\label{ent}{\hbox{Ent}}_\Phi (f\mid\mu )=\int_M \Phi (f(x))\mu (dx) -\Phi \Bigl( \int_M f(x)\mu (dx)\Bigr).\end{align}
\indent (ii) For $f\in H^1(\mu )$, the relative $\Phi$ information of  $fd\mu$ with respect to to $d\mu $ is
\begin{align}{\mathcal{I}}_\Phi (f\mid \mu )=\int_M\Phi''(f(x))\Vert \nabla f(x)\Vert^2 \mu (dx).\end{align}
\indent (iii) Say that $\Phi$ is admissible if $-1/\Phi'' (x)$ is convex. For an admissible $\Phi$, we say that $\mu$ satisfies a $\Phi$  entropy-entropy production inequality with constant $\kappa_0>0$ if
\begin{align}\label{ententprod}\int_M \Phi (f(x))\mu (dx) -  \Phi \Bigl( \int_M f(x)\mu (dx)\Bigr)\leq {\frac{1}{2\kappa_0}}\int_M \Phi''(f(x))\Vert\nabla f(x)\Vert^2\mu (dx).\end{align}
\end{defn}

 In Theorem \ref{Theorem1} we show how our solution of the minimization problem  (\ref{Evariation}) is controlled in terms of  
\begin{align}\label{specialFisher}{\mathcal{I}}_\Phi (\rho \mid m )=\int_M \rho (x)\Vert\nabla ( \Theta_1\circ\rho)(x)\Vert^2 m(dx)\end{align}
for  $\Phi (r)=\int_0^r(r-u)u\Theta'_1(u)^2du$
 and the corresponding $\Phi$-entropy is related to ${\mathcal{U}}(\rho)$ via convexity inequalities in Proposition \ref{proposition8}, under the following hypotheses. \par
\begin{defn}\label{Geometrical Hypotheses} {\bf (Geometrical Hypotheses)}  Suppose that $M$ has bounded geometry and has dimension $n$, so that there exist $r_0>0$ and $\kappa_0>0$ such that the injectivity radius is bounded below by $r_0$ and the Ricci curvature is bounded below by $\kappa_0$. Let $\iota_n(M)$ be the isoperimetric constant 
\begin{align}\label{Geom}\iota_n(M)=\inf_\Omega {\frac{A(\partial \Omega)}{\min \{ m(\Omega ), m(\Omega^c)\}^{(n-1)/n}}}\end{align}
\noindent where the infimum is taken over all open submanifolds $\Omega, \Omega^c$ such that $\partial \Omega$ is a $C^\infty$ boundary of $\Omega$ that partitions $M$ into $M=\partial \Omega\cup \Omega\cup\Omega^c$, and $A(\partial \Omega )$ is the area of $\partial \Omega$. 
Let $\iota_\infty (M)$ be Cheeger's constant 
\begin{align}\label{Cheeger}\iota_\infty(M)=\inf_\Omega {\frac{A(\partial \Omega)}{\min \{ m(\Omega ), m(\Omega^c)\}}}.\end{align} 
Suppose that $\iota_n(M)>0$.\end{defn}

\indent When $\Phi (r)=r^2/2$ the  $\Phi$  entropy-entropy production inequality reduces to the spectral gap inequality (\ref{Spec}) in $L^2(\mu )$ , and when $\Phi (r)=r\log r$, (\ref{ententprod}) is the logarithmic Sobolev inequality (\ref{logsob}). In particular, if $M$ has positive Ricci curvature, then the spectral gap inequality holds for the constant density $\rho =1/m(M)$ as in \cite{bib20} this implies a version of the Hodge-Helmholtz decomposition. This is required for stage 3 of the algorithm.\par

\indent  {\bf Stage 3.} The third stage of the algorithm involves finding $q_h$ such that $\vec v_h=\nabla q_h$. At each stage of the algorithm, we have a pair $(\rho ,q)$ where $q$ is constructed such that $v=\nabla q$, and we require the energy to be conserved, or dissipated, so ${\mathcal H}(\rho_h, q_h)\leq {\mathcal H}(\rho_0, q_0)$.\par
\indent  In section \ref{Minimizer} we deduce the implications of the spectral gap condition and  show how the various densities satisfy this condition, using results from \cite{bib10}, \cite{bib20}. In particular, we formulate Theorem \ref{Theorem1} to that ensure that $\rho_h$ is uniformly positive and uniformly bounded on $M$, so that a spectral gap or Poincar\'e inequality holds for the  quadratic form $q\mapsto \int_M \Vert \nabla q(x)\Vert^2\rho_h(x)m(dx)$.  The Corollary \ref{Corollary2} for ${\mathbb{S}}^2$ in section \ref{Minimizer} has obvious terrestrial applications and the proof uses results that are currently known only for manifolds that closely resemble the spheres. The optimal transport map between uniformly positive and smooth densities on ${\mathbb{S}}^2$ is also smooth by results of Loeper \cite{bib22}, \cite{bib23}. The physical interpretation is that the gas does not form a vacuum and its density is bounded, and we avoid the problematic issue of having infinite velocities on subsets where the gas has zero density.\par
\indent We consider $x,y\in M$ and note that the Hessian gives rise to the quadratic form $\langle D_x^2 d(x,y)^2\xi, \xi\rangle =(d^2/dt)^2_{t=0}d^2(\exp_x(t\xi ),y)$ for $\xi\in T_xM$. The condition $(Aw)$ requires that for all 
$\xi,\eta,\zeta\in T_xM$ such that $\langle \xi, \zeta\rangle =0$, the function
\begin{align}s\mapsto {\frac{1}{2}}\Bigl({\frac{d^2}{dt^2}}\Bigr)_{t=0}d^2\bigl(\exp_x(t\xi ), \exp_x(\eta +s\zeta )\bigr)\end{align}
is concave. Originally $(Aw)$ was introduced to ensure regularity of solutions of the Monge-Ampere equations as in \cite{bib29}, before Leoper established  \cite{bib22} the relationship between $(Aw)$ and positive sectional curvature of ${\hbox{span}}\{ \zeta, \xi \}$ in $T_xM$, and verified $(Aw)$ for the spheres ${\mathbb{S}}^n$. In section \ref{Energy} of the current paper, we use a uniform version of this condition $(A3)$ to establish log concavity of certain Jacobian determinants under geodesic interpolation over the sphere.\par
\indent In section \ref{weakeuler}, we construct weak solutions of the continuity equation (\ref{continuity}), and give a sufficient condition for existence of weak solutions to the acceleration equation (\ref{acceleration}).\par
\section{\label{Moving}Moving frames and transport for the Euler equations on the sphere}
\noindent 
 Let $TM$ be the tangent bundle of $M$. Let $\gamma :[0, h]\rightarrow M\subset {\mathbb{R}}^3$ be a $C^2$ curve, which is regular in the sense that the velocity $\dot \gamma(t)\neq 0$ for all $t\in [0, h]$. Given $[x_0; v_0], [x_h; v_h]\in TM$, we introduce the cost $c:TM\times TM\rightarrow [0, \infty )$ on the tangent bundle by
\begin{align}\label{defc}c&\bigl( [x_0; v_0], [x_h; v_h]\bigr)\nonumber\\
&=\inf_\gamma\Biggl\{ \int_0^h \bigl(\Vert \dot \gamma (t)\Vert^2+\Vert \ddot \gamma (t)\Vert^2\bigr)dt: \begin{bmatrix} \gamma (0)\\ \dot \gamma(0)\end{bmatrix} =\begin{bmatrix}x_0\\v_0\end{bmatrix} ,
\begin{bmatrix} \gamma (h)\\  \dot \gamma (h)\end{bmatrix} =\begin{bmatrix} x_h\\ v_h\end{bmatrix} \Biggr\}\end{align}
\noindent where the infimum is taken over all the $C^4$ paths $\gamma :[0,h]\rightarrow M$ with the specified initial and final points and velocities. Note that 
$$c\bigl( [x_0, v_0], [x_h; v_h]\bigr) =c\bigl( [x_h; -v_h], [x_0; -v_0]\bigr)$$
since reversing the journey makes the velocity go negative. \par

\begin{prop}\label{proposition1} Let $\tilde \rho_0$ and $\tilde\rho_h$ be probability measures on $TM$ such that $[X_0;V_0]\mapsto [X_h;V_h]$ induces $\tilde \rho_h$ from $\tilde\rho_0$.  Suppose that $\tilde \rho_0$ has marginal $\rho_0$ and $\tilde \rho_h$ has marginal $\rho_h$ under the canonical projection $TM\rightarrow M:$ $[x;v]\mapsto x$. Then the transport cost for (\ref{defc}) of moving $\tilde \rho_0$ to $\tilde\rho_h$ along a curve $X(t)$ of curvature $\kappa$ is
\begin{align}\label{TC}TC_c(\tilde\rho_0, \tilde\rho_h)\geq W_2^2(\rho_0, \rho_h)+\int_M\int_0^h \dot s^4(t)\kappa (X(t,X_0))^2dt\rho_0(dX_0).\end{align}
 \end{prop}
\begin{proof} Given a $C^2$ curve $\gamma :[0,1]\rightarrow M$ starting at $x$, we introduce moving frames that incorporate acceleration and which reveal the underlying geometry of $M$. Then a curve $\gamma (s,x)$ of unit speed has a Serret-Frenet frame $S=[\gamma'; \gamma''/\kappa; \gamma'\times \gamma''/\kappa ]$ where ${}'=d/ds$ so that 
\begin{align}\label{Sframe}\begin{bmatrix}\gamma'\cr \gamma''/\kappa\cr \gamma'\times\gamma''/\kappa\end{bmatrix} =\begin{bmatrix}1&0&0\cr 0&\kappa_n/\kappa& \kappa_g/\kappa\cr 0&\kappa_g/\kappa&-\kappa_n/\kappa\end{bmatrix} \begin{bmatrix}\gamma'\cr N\cr N\times \gamma'\end{bmatrix},\end{align}
where the orthogonal matrix $R$ in the middle of (\ref{Sframe}) gives a change of orthonormal basis of ${\mathbb{R}}^3$; we abbreviate this by $S=RF,$ where $F=[\gamma';N;N\times \gamma']$ in which $N\times \gamma'$ is a unit vector in $TM$ perpendicular to $\gamma'$. Using this moving frame, we obtain a lower bound on the transportation cost for (\ref{defc}).\par 
\indent We can transform a regular curve $[X;V]$ in $TM$ to a unit speed curve via the change of variables $s=\int_0^t \Vert V(u)\Vert du$, so we have
\begin{align}\label{TC1}\Bigl\Vert{\frac{d}{dt}}\begin{bmatrix} X\\ V\end{bmatrix}\Bigr\Vert^2&= \Bigl\Vert {\frac{dX}{ds}}{\frac{ds}{dt}}\Bigr\Vert^2 + \Bigl\Vert {\frac{d}{dt}}\Bigl( {\frac{dX}{ds}}{\frac{ds}{dt}}\Bigr)\Bigr\Vert^2 \nonumber\\  
&=\dot s^2+\Bigl\Vert -\kappa_n N\dot s^2+\kappa_g N\times {\frac{dX}{ds}}\dot s^2+\ddot s{\frac{dX}{ds}}\Bigr\Vert^2\nonumber\\
&= \dot s^2+(\kappa_n^2+\kappa_g^2)\dot s^4+\ddot s^2\end{align}
so we have
\begin{align} \label{TC2}\int_0^h \int_{TM}\Bigl\Vert{\frac{d}{dt}}\begin{bmatrix} X\\ V\end{bmatrix}\Bigr\Vert^2\tilde\rho_0(dX_0dV_0)& \geq \int_M \int_0^h\dot s^2(t)dt\rho_0(dX_0)\nonumber\\
&\quad +\int_M \int_0^h \dot s^4(t)\kappa (X(t, X_0))^2 dt\rho_0(dX_0),\end{align}
where $\kappa$ is the curvature of the solution curve. Hence the transportation cost satisfies (\ref{TC}).\par
\end{proof} 
\vskip.05in
\indent Let $v:M\rightarrow {\mathbb{R}}^3$ be a $C^2$ vector field. Let $N$ be the unit normal vector to $N$, and observe that $v=v_M+(N\cdot v)N$ where $v_M :M\rightarrow TM$ is the tangential component of the vector field $v$. Now let $\nabla_M=\nabla -N (N\cdot\nabla )$ so that $\nabla_M\phi :M\rightarrow TM$ is a tangential vector field to $M$ for all $C^1$ scalar fields $\phi :M\rightarrow {\mathbb{R}}$. There is a Hodge-Helmholtz decomposition $v_M=\nabla_M\phi+ N\times \nabla_M\psi .$ In the remainder of this section, we write $\nabla$ for $\nabla_M$, so that $\nabla \phi$ is the gradient tangential to the manifold $M$.\par 

\indent Let ${\mathbb{S}}^2$ be the unit sphere in ${\mathbb{R}}^3$, which has ${\mathbb{S}}^1\sim\{ \xi \in {\mathbb{R}}^3: \Vert \xi \Vert =1; x\cdot \xi =0\}$ as fibres of its unit tangent bundle for all $x\in {\mathbb{S}}^2$. 
\indent Now let $[X;V]$ satisfy the ODE (\ref{ODE}),
 and consider the frame $W=[V;  X; X\times  V]$ where $\dot{}=d/dt$ and 
\begin{align}\label{frame}\dot W={\frac{d}{dt}}\begin{bmatrix}V\cr  X\cr X\times  V\end{bmatrix}= \begin{bmatrix}g(X)\cr V\cr X\times g(X)\end{bmatrix}.\end{align}
The space curve $\gamma (s)=X(t)$ with arclength $s=\int_0^t \Vert V(u)\Vert du$ determines a unit speed curve on $M$, hence a Serret -Frenet frame $S$ with differential equation $S'=\Omega S$, and a moving frame $F$, where $S=RF$ as in (\ref{Sframe}) such that 
the ODE for $W$ determines the ODE $F'=(R^{-1}\Omega R -R^{-1}R')F$. 
 In the remaining part of this section, we are mainly concerned with transport on ${\mathbb{S}}^2$, although most results in this section extend to more general surfaces; see \cite{bib11}. For
\begin{align}{\mathcal H}(\rho ,q)=\int_{{\mathbb{S}}^2}\Bigl( {\frac{1}{2}}\bigl\Vert\nabla_{{\mathbb{S}}^2}q (X)\bigr\Vert^2 +{\frac{1}{2}}\bigl\Vert X\Vert^2+\Theta (\rho (X))\Bigr)\rho (X)\, m(dX)\end{align}
a canonical equation of motion gives 
$$-{\frac{\partial q}{\partial t}}=  {\frac{1}{2}}\bigl\Vert\nabla_{{\mathbb{S}}^2}q (X)\bigr\Vert^2 +{\frac{1}{2}}\bigl\Vert X\Vert^2+\Theta_1\circ \rho (X)$$
to which we apply $\nabla=\nabla_{{\mathbb{S}}^2}+X{\frac{\partial}{\partial r}}$, which commutes with ${\frac{\partial}{\partial t}}.$ We obtain
$$-{\frac{\partial V}{\partial t}}=\nabla_{{\mathbb{S}}^2}{\frac{1}{2}}\Vert V\bigr\Vert^2+X+\nabla_{{\mathbb{S}}^2}(\Theta_1\circ\rho ).$$
Then in spherical polar coordinates, 
\begin{align}\label{sphericalframe}X= \begin{bmatrix}\sin\theta\cos\phi\cr \sin\theta\sin\phi\cr \cos\theta\end{bmatrix},\quad \vec \theta= \begin{bmatrix}\cos\theta\cos\theta\cr \cos\theta\sin\phi\cr -\sin\theta\end{bmatrix},\quad \vec \phi= \begin{bmatrix}-\sin\phi \cr \cos\phi \cr 0\end{bmatrix},\end{align}
 gives an orthonormal frame in which we have $V=V_\theta\vec\theta+V_\phi\vec\phi$ and $\nabla_{{\mathbb{S}}^2} = \vec\theta {\frac{\partial }{\partial\theta}}+\vec\phi \sin\theta {\frac{\partial }{\partial\phi}}$, so by a short calculation one finds $\nabla_{{\mathbb{S}}^2}\otimes V$ and checks that
$$\bigl(\nabla_{{\mathbb{S}}^2}\otimes V\bigr) V=\nabla_{{\mathbb{S}}^2}{\frac{1}{2}}\Vert V\Vert^2,$$
so ${\frac{d}{dt}} ={\frac{\partial }{\partial t}}+ \nabla_{{\mathbb{S}}^2}\otimes V$ acts as an advection operator on $T{\mathbb{S}}^2$, and 
$${\frac{dV}{dt}}=-X-\nabla_{{\mathbb{S}}^2}(\Theta_1\circ\rho ).$$
\indent   We interpret $V$ as a Lagrangian velocity, and a convenient label is the initial condition $[X(0);V(0)]=[X_0;V_0]$. The following result is a variant of Kelvin's circulation theorem \cite{bib3} p. 34 which shows that the frames $W$ and $F$ have a very natural interpretation for solutions of Euler equations, and that solutions transport probability densities over ${\mathbb{S}}^2$. The solutions do not generally give a geodesic flow, hence do not give optimal transport for the $W^2({\mathcal{S}}^2)$.\par 
\vskip.05in
\begin{prop}\label{proposition2}Let $[X(t);V(t);X(t)\times V(t)]$ be a bounded solution of the ODE (\ref{frame}) on ${\mathbb{S}}^2$ for $t\in [0, \tau ]$ for some $\tau>0$, with initial condition $[X_0;V_0;X_0\times V_0]$, and let $\rho_0\in {\mathcal W}^2({\mathbb{S}}^2)$.\par
\indent (i) The matrix $R$ of (\ref{Sframe}) has $\kappa_n=-1$ and $\kappa_g=-V\cdot (X\times g(X))/\dot s^3$.\par
\indent (ii) Suppose that $g(X)=-X-\nabla_{{\mathbb{S}}^2}\Theta_0 (X)$ for some $\Theta_0 \in C^2 ({\mathbb{S}}^2, {\mathbb{R}})$. Then\par $X\cdot V=0$ if and only if the curve $t\mapsto X(x,t)$ has unit speed; in this case, $\vert \kappa_g\vert=\Vert\nabla_{{\mathbb{S}}^2}\Theta_0 (X)\Vert$.\par
\indent (iii) Let $\tilde \rho_0$ and $\tilde\rho_h$ be probability measures on $T{\mathbb{S}}^2$ such that $[X_0;V_0]\mapsto [X_h;V_h]$ induces $\tilde \rho_h$ from $\tilde\rho_0$. Then for this 
transport plan along a unit speed curve, the transport cost for $c$  from (\ref{defc}) on $T{\mathbb{S}}^2$ satisfies
$$TC_{c}(\tilde\rho_0, \tilde\rho_h)\geq 2W_2^2(\rho_0, \rho_h)+\int_0^h \int_{{\mathbb{S}}^2} \Vert\nabla_{{\mathbb{S}}^2}\Theta_0 (X(t,X_0))\Vert^2\rho_0(dX_0)dt$$
where $\rho_0$ is the marginal of $\tilde\rho_0$ and $\rho_h$ is the marginal of $\tilde\rho_h$ under the canonical projection $T{\mathbb{S}}^2\rightarrow {\mathbb{S}}^2.$\par
\indent (iv) Also, the normal part of the vorticity $X\cdot (\nabla_{{\mathbb{S}}^2} \times V)$, is invariant under the flow. If initially $V_0=\nabla_{{\mathbb{S}}^2}q_0$ for some velocity potential $q_0$, then  $X\cdot (\nabla_{{\mathbb{S}}^2} \times V)=0$ for all $t>0$ and the orthogonal frame $[ X;V; X\times V]$ is given by position, velocity and vorticity.\par
\indent (v) Suppose that $x\mapsto X(x,t)$ induces $\rho_t$ from $\rho_0$. Then $t\mapsto \rho_t$ gives a $2$-absolutely continuous path $[0, \tau ]\rightarrow {\mathcal W}^2({\mathbb{S}}^2)$.\par 
\indent (vi) In particular, if $g(X)=-X$, then the solution curve is a geodesic on ${\mathbb{S}}^2$, and the corresponding curve $t\mapsto \rho_t$ is also a geodesic in ${\mathcal W}^2({\mathbb{S}}^2)$. \par 
\end{prop}
\begin{proof}  (i) In the case of ${\mathbb{S}}^2$, we have $N=X$ and the Lagrangian differential equation (\ref{ODE}) leads to an ODE for the dynamics of frame $F$, as in
$$F=\begin{bmatrix}V/\dot s\cr X\cr X\times V/\dot s\end{bmatrix}, \qquad F'={\frac{1}{\dot s}}\begin{bmatrix} g(X)/\dot s-\ddot sV/\dot s^2\cr V\cr X\times ( g(X)/\dot s-\ddot sV/\dot s^2)\end{bmatrix}.$$
Also, the entries in the differential equation satisfy
$$V\cdot (X\times g(X))=V\cdot (X\times \dot V)=\gamma'\dot s\bigl( X\times(\gamma''\dot s^2+\gamma'\ddot s)\bigr) =\dot s^3\gamma''\cdot( \gamma'\times X)=-\kappa_g\dot s^3,$$
$$g(X)\cdot X=\dot V\cdot X=(\gamma''\dot s^2+\gamma'\ddot s)\cdot X=-\dot s^2,$$
$$g(X)\cdot V=\dot V\cdot V=(\gamma''\dot s^2+\gamma'\ddot s)\cdot \gamma'\dot s=\ddot s\dot s;$$
$$\Vert \dot V\Vert^2=\Vert g(X)\Vert^2\Vert X\Vert^2=\Vert g(X)\times X\Vert^2 +(X\cdot g(X))^2=\dot s^4+\Vert g(X)\times X\Vert^2,$$ 
\noindent which determine $\kappa_g,$ and since $\kappa_n=-1$ we have $\kappa=\sqrt{1+\kappa_g^2}$, hence we have determined $R$.\par 
\indent (ii)  For $\Theta_0\in C^2({\mathbb R}^3; {\mathbb{R}})$, the map $\nabla \Theta_0(X)$ has a decomposition
into a tangential component $\nabla_{{\mathbb{S}}^2}\Theta_0 =\nabla \Theta_0(X)-(X\cdot \nabla \Theta_0(X))X$ and a normal component $(X\cdot \nabla \Theta_0)X$.
The map $X\mapsto -X-\nabla_{{\mathbb{S}}^2}\Theta_0 (X)$ is Lipschitz continuous, so the ODE has a unique solution.
 We have
$${\frac{d}{dt}}( X\cdot V)=V\cdot V+X\cdot (-X-\nabla_{{\mathbb{S}}^2}\Theta_0 (X))=V\cdot V-X\cdot X=\dot s^2-1;$$
this vanishes, if and only if the curve has unit speed. In this case $-X-\nabla_{{\mathbb{S}}^2} \Theta_0(X)=dV/dt =-X+\kappa_g X\times V,$ so $\kappa_gV=X\times \nabla_{{\mathbb{S}}^2}\Theta_0 (X)$, and the geodesic curvature is $\vert \kappa_g\vert=\Vert\nabla_{{\mathbb{S}}^2}\Theta_0 (X)\Vert$.  Here $-X$ may be interpreted as a constraining force normal to the surface. Any unit speed curve on ${\mathbb{S}}^2$ gives rise to a curve in $SO(3)$, since there exists $\hat \Phi (t)\in SO(3)$ that takes $[ X_0;V_0; X_0\times V_0]$ to $[X(t); V(t); X(t)\times V(t)]$. For a unit speed curve, there is a function $X(t)\mapsto V(t)\in T_{X(t)} {\mathbb{S}}^2$ with $\Vert V(t)\Vert=1$. \par
\indent (iii) We consider the unit speed curve $X(t, X_0)$ on ${\mathbb{S}}^2$ with $\kappa_n=1$ and $\vert \kappa_g(X)\vert=\Vert\nabla_{{\mathbb{S}}^2}\Theta_0 (X)\Vert$, so the acceleration has norm squared
$$\Bigl\Vert {\frac{dV}{dt}}\Bigr\Vert^2=1+\Vert\nabla_{{\mathbb{S}}^2}\Theta_0 (X(t, X_0))\Vert^2$$
so the result follows as in Proposition \ref{proposition1}.\par 
\indent (iv) While there does not exist a nonzero continuous tangential vector field on ${\mathbb{S}}^2$, we can consider a $C^1$ vector field $V$ that is tangential to ${\mathbb{S}}^2$ on a proper region $B\subset {\mathbb{S}^2}$, and we suppose that $\Gamma$ is a contour of unit speed on ${\mathbb{S}}^2$ that bounds $B$. Then we have
$$\int_\Gamma V\cdot \Gamma' \, ds=\int\!\!\!\int_B (\nabla\times V)\cdot N m(dx) =\int\!\!\!\int_B (\nabla_{{\mathbb{S}}^2}\times V)\cdot N m(dx),$$
so $\nabla_{{\mathbb{S}}^2}\times V$ is the vorticity. Hence
\begin{align}{\frac{d}{dt}}\int\!\!\!\int_B (\nabla_{{\mathbb{S}}^2}\times V)\cdot N m(dx)&=\int_{\Gamma} {\frac{dV}{dt}}\cdot \Gamma'(s)ds\nonumber\\
&=\int_\Gamma (-X-\nabla_{{\mathbb{S}}^2}\Theta_0(X))\cdot \Gamma'(s)\, ds\nonumber\\
&=-\int_{\Gamma} {\frac{d}{ds}}\Theta_0(\Gamma (s)) ds=0.\end{align}
\indent For $x,v\in {\mathbb{S}}^2$ such that $x\cdot v=0$, we have $\exp_x(tv)=x\cos t+v\sin t.$  In terms of the frame (\ref{sphericalframe}), $d/dt$ along $\exp_X t\vec e_\theta $ corresponds to $\partial/\partial\theta$ and $d/ds$ along $\exp_X s\vec e_\phi$ corresponds to $\partial/\sin\theta\partial\phi$. Suppose momentarily that $V=\nabla_{{\mathbb{S}}^2} q$. Then the tangential vector field $V$ to ${\mathbb{S}}^2$ satisfies  
\begin{align}\nabla_{{\mathbb{S}}^2}\times V&=\Bigl(\vec\theta {\frac{\partial}{\partial\theta}}+\vec\phi {\frac{\partial}{\sin\theta\partial\phi}}\Bigr)\times \bigl( V_\theta\vec\theta+V_\phi\vec\phi)\nonumber\\
&= 
X\times \bigl( V_\theta\vec\theta+V_\phi\vec\phi)+{\frac{X}{\sin\theta}}\Bigl( {\frac{\partial}{\partial\theta}} (\sin\theta\, V_\phi )-{\frac{\partial V_\theta}{\partial\phi}}\Bigr),\end{align}  
  hence $V=\nabla_{{\mathbb{S}}^2} q$ has vorticity 
$$\nabla_{{\mathbb{S}}^2}\times V= \nabla_{{\mathbb{S}}^2}\times \nabla_{{\mathbb{S}}^2}q=X\times \nabla_{{\mathbb{S}}^2}q=X\times V.$$
Hence $X\cdot (\nabla_{{\mathbb{S}}^2}\times V)=0$ as in (iii), and the vorticity has zero divergence on ${\mathbb{S}}^2$ since for $\Gamma$ any contour of unit speed on ${\mathbb{S}}^2$ that bounds a region $B$, we have $\vec n=X\times \Gamma'$ normal to $\Gamma'$ in $T{\mathbb{S}}^2$ and 
\begin{align}\int_\Gamma (X\times V)\cdot \vec n\, ds&=\int\!\!\!\int_B \nabla_{{\mathbb{S}}^2}\cdot (X\times V)\, m(dx)\nonumber\\
&=\int\!\!\!\int_B (\nabla_{{\mathbb{S}}^2}\times X )\cdot V\, m(dx)-\int\!\!\!\int_B (\nabla_{{\mathbb{S}}^2}\times V )\cdot X\, m(dx)\nonumber\\
&=-\int\!\!\!\int_B (X\times V )\cdot X\, m(dx)=0.\end{align}
by the divergence theorem.  For ${\mathbb{S}}^2$, the Green's function is $$G(B,C)=(4\pi)^{-1}\log (1-\cos a)=(4\pi)^{-1}\log (\Vert B-C\Vert^2/2)$$ 
where $a$ is the angle between $B,C\in {\mathbb{S}}^2$.  On $\{ f\in H^1({\mathbb{S}}^2:\int_{{\mathbb{S}}^2} f(x)m(dx)=0\}$, the operator $\nabla_{{\mathbb{S}}^2}\cdot\nabla_{{\mathbb{S}}^2}$ defines a closeable quadratic form by Poincar\'e's inequality for ${\mathbb{S}}^2$ and $\nabla_{{\mathbb{S}}^2}\cdot\nabla_{{\mathbb{S}}^2}G=I$.  Suppose that $V$ is a tangential vector field to ${\mathbb{S}}^2$ that has a Hodge-Helmholtz decomposition $V=\nabla_{{\mathbb{S}}^2}q+X\times \nabla_{{\mathbb{S}}^2}\psi$; here $q=G(\nabla_{{\mathbb{S}}^2}\cdot V)$ and $\psi =-G (\nabla_{{\mathbb{S}}^2}\cdot (X\times V))$, so 
$\nabla_{{\mathbb{S}}^2}\cdot (X\times \nabla_{{\mathbb{S}}^2}\psi )=0$ and $X\cdot (\nabla_{{\mathbb{S}}^2}\times \nabla_{{\mathbb{S}}^2}q)=0$.  Then by the divergence theorem
\begin{align}0&=\int\!\!\!\int_{{\mathbb{S}}^2}\nabla_{{\mathbb{S}}^2}\cdot \bigl( q(X\times \nabla_{{\mathbb{S}}^2}\psi )\bigr) m(dx)\cr
&=\int\!\!\!\int_{{\mathbb{S}}^2}\Bigl(\nabla_{{\mathbb{S}}^2}q)\cdot \bigl(X\times \nabla_{{\mathbb{S}}^2}\psi\bigr) + q(\nabla_{{\mathbb{S}}^2} \times X)\cdot \nabla_{{\mathbb{S}}^2}\psi -qX\cdot ( \nabla_{{\mathbb{S}}^2}\times \nabla_{{\mathbb{S}}^2}\psi )\Bigr) m(dx)\cr
&=\int\!\!\!\int_{{\mathbb{S}}^2}\bigl(\nabla_{{\mathbb{S}}^2}q)\cdot \bigl(X\times \nabla_{{\mathbb{S}}^2}\psi\bigr)m(dx).\end{align}
Hence by orthogonality we have
$$\int_{{\mathbb{S}}^2} \Vert V(x)\Vert^2m(dx)= \int_{{\mathbb{S}}^2} \Vert \nabla_{{\mathbb{S}}^2}q(x)\Vert^2m(dx)+\int_{{\mathbb{S}}^2} \Vert \nabla_{{\mathbb{S}}^2}\psi\Vert^2m(dx).$$
\indent If $V_0=\nabla_{{\mathbb{S}}^2}q_0$, then $X\cdot (\nabla_{{\mathbb{S}}^2}\times V)=0$ initially and for all subsequent times by (ii), so there exists $q(x,t)$ such that  $V=\nabla_{{\mathbb{S}}^2} q$.\par
\indent (v) The proof is similar to Theorem 8.3.1 of \cite{bib2}. The solution gives a curve that passes through $X(x,t_1)$ and $X(x,t_2)$, and $V(x,t)\in T_{X(t,x)}{\mathbb{S}}^2$, so by Cauchy-Schwarz 
\begin{align}\label{CS}{\frac{d(X(x,t_1), X(x,t_2))^2}{t_2-t_1}}\leq\int_{t_1}^{t_2} \Vert V(x,t)\Vert^2 dt\qquad (0<t_1<t_2<\tau ).\end{align}
Note that $x\mapsto (X(x,t_1), X(x,t_2))$ induces from $\rho_0$ a probability measure on ${\mathbb{S}}^2\times {\mathbb{S}}^2$ that has marginals $\rho_{t_1}$ and $\rho_{t_2}$, which is a transport plan for taking $\rho_{t_1}$ to $\rho_{t_2}$. Integrating (\ref{CS}) against $\rho_0(x)m(dx)$ gives
\begin{align}\label{2absfromODE}{\frac{W_2^2(\rho_{t_2}, \rho_{t_1})}{t_2-t_1}}\leq\int_{t_1}^{t_2} \int_{{\mathbb{S}}^2} \Vert V(x,t)\Vert^2\rho_0(x)m(dx)dt\qquad (0\leq t_1<t_2\leq \tau ),\end{align}  
so $t\mapsto \rho_t$ is $2$-absolutely continuous, provided $\int_0^\tau \int_{{\mathbb{S}}^2} \Vert V(x,t)\Vert^2\rho_0(x)m(dx)dt<\infty$.
 The intended application has $\Theta_0(x)=\Theta_1(\rho (x,t))$, where the density varies with time and satisfies the continuity equation.\par
\indent (vi)  We can express a typical  unit speed geodesic on ${\mathbb{S}}^2$ as
$$\begin{bmatrix} X\cr V\cr X\times V\end{bmatrix} =\begin{bmatrix} \cos t &\sin t&0\cr -\sin t&\cos t&0\cr 0&0&1\end{bmatrix} \begin{bmatrix} X_0\cr V_0\cr X_0\times V_0\end{bmatrix}$$
for $X_0, V_0\in {\mathbb{S}}^2$ such that $X_0\cdot V_0=0$, so that $X_0\times V_0$ is the unit normal to the great circle through $X_0$ in the direction of $V_0$.
\end{proof} 
\begin{rem}\label{remark1}(i) The Green's function for ${\mathbb{S}}^2$ may be found by taking the conformal stereographic projection of the Riemann sphere onto ${\bf C}\cup \{ \infty \}$. One can obtain the Green's function for some other compact surfaces similarly. The cost function $c(x,y) =-(1/2)\log (2-2x\cdot y)$ is considered in the reflector antenna problem on ${\mathbb{S}}^2$, and shares some properties with $d^2(x,y)/2=(\arccos (x\cdot y))^2/2$ by \cite{bib23}. \par
\indent (ii) The notion of interpolation between measures is considered in section 4 of \cite{bib1} for ${\mathbb{R}}^d$, in which case the Jacobi fields amount to families of straight lines. The situation for ${\mathbb{S}}^2$ is considered in the next section.\end{rem}
\section{\label{Convexity}Convexity and Wasserstein distance}
\noindent In this section, we are concerned with how curvature of $M$ affects curvature of ${\mathcal W}^2(M)$ as a metric space, and begin by considering a functional from (\ref{W2}).   
\vskip.05in
\begin{defn}For a bijection $T\in C({\mathbb{S}}^2, {\mathbb{S}}^2)$, and $\Phi_t(x)=\exp_x(tv(x))$ for $v(x)\in T_x{\mathbb{S}}^2$ let 
\begin{align}E[T]={\frac{1}{2}}\int_{{\mathbb{S}}^2} d(T(x),x)^2f(x)m(dx),\end{align}
\noindent where $f$ is a probability density function on ${\mathbb{S}}^2$. Let $v$ be a smooth tangential vector field on ${\mathbb{S}}^2$ and let $\Phi_t:{\mathbb{S}}^2\rightarrow {\mathbb{S}}^2$ satisfy $(d/dt)\Phi_t(x)=v\circ \Phi_t(x)$ with $\Phi_0(x)=x$ for all $(x,t)\in {\mathbb{S}}^2\times [0,1].$ Then the first outer variation of $E[T]$ with respect to the flow $\Phi_t$ is defined by
\begin{align}\langle \delta_o E[T],v\rangle =\Bigl({\frac{d}{dt}}\Bigr)_{t=0}E[ \Phi_t\circ T].\end{align}
\end{defn}
\noindent Our terminology is adopted so that it does not conflict with the inner variation considered in \cite{bib13}. Our notation emphasizes the vector field $v$ and we do not require $T$ to be differentiable with respect to $x$.  

\begin{prop}\label{proposition3}
(i) Then the first outer variation is 
\begin{align}\label{firstvariation}\langle \delta_o{E}[T], v\rangle=-\int_{{\mathbb{S}}^2} {\frac{d(T(x),x)}{\sin d(T(x),x)}}x \cdot v(T(x))f(x)m(dx),\end{align}
\indent (ii)  and $({d^2}/{dt^2})_{t=0}E[ \Phi_t\circ T]\geq 0$.\end{prop}
\begin{proof}  (i) Let $T^*$ be the inverse of $T$, and $T^*(y)=\exp_y\zeta (y)$ where $\zeta\in T_y{\mathbb{S}}^2$ has $\Vert \zeta (y)\Vert=d(T^*(y),y)$ and
\begin{align}\label{zeta}\zeta (y)={\frac{\Vert \zeta\Vert}{\sin \Vert\zeta\Vert}}\bigl(T^*(y)-(T^*(y)\cdot y) y\bigr).\end{align} 
Then by the cosine rule applied to the spherical triangle $\triangle (T^*(y), \Phi_t(y), y)$, we have
\begin{align}\cos d(\Phi_t(y),T^*(y))&=\cos d(y,T^*(y))\cos d(y,\Phi_t(y))\nonumber\\
&\quad +\sin d(y, T^*(y))\sin d(y, \Phi_t(y)){\frac{v(y)\cdot \zeta (y)}{\Vert v(y)\Vert\Vert \zeta (y)\Vert}},\end{align}
hence 
\begin{align}\label{firstderivative}\Bigl({\frac{d}{dt}}\Bigr)_{t=0} d(\Phi_t(y), T^*(y))=-{\frac{v(y)\cdot \zeta (y)}{\Vert v(y)\Vert\Vert \zeta (y)\Vert}}\Bigl({\frac{d}{dt}}\Bigr)_{t=0} d(\Phi_t(y),y)=-{\frac{v(y)\cdot \zeta (y)} {\Vert \zeta (y)\Vert}}.\end{align}
Hence with $\rho=T\sharp f$, we have
\begin{align}\label{varinnerproduct}\langle \delta_o{E}[T], v\rangle&=\Bigl({\frac{d}{dt}}\Bigr)_{t=0}{\frac{1}{2}}\int_{{\mathbb{S}}^2} d(\Phi_t\circ T(x),x)^2f(x)m(dx)\nonumber\\
&=\Bigl({\frac{d}{dt}}\Bigr)_{t=0}{\frac{1}{2}}\int_{{\mathbb{S}}^2} d(\Phi_t(y),T^*(y)))^2\rho (y)m(dy)\nonumber\\
&=-\int_{{\mathbb{S}}^2} d(y,T^*(y)) {\frac{v(y)\cdot \zeta (y)}{\Vert \zeta (y)\Vert}} \rho (y)m(dy)\nonumber\\
&=-\int_{{\mathbb{S}}^2} v(T(x))\cdot \zeta (T(x)) f(x)m(dx),\end{align}
so from (\ref{zeta}), 
\begin{align}\langle \delta_o{E}[T], v\rangle=-\int_{{\mathbb{S}}^2} v(T(x))\cdot {\frac{d(T(x),x)}{\sin d(T(x),x)}}\bigl( x-(x\cdot T(x))T(x)\bigr) f(x)m(dx)\end{align}
and $T(x)\cdot v(T(x))=0$ since $v$ is a tangential vector field on ${\mathbb{S}}^2$, so we reduce to (\ref{firstvariation}).\par
\indent (ii)  Likewise, the second derivative is 
\begin{align}\Bigl(&{\frac{d^2}{dt^2}}\Bigr)_{t=0}{\frac{1}{2}}\int_{{\mathbb{S}}^2} d(\Phi_t\circ T(x),x)^2f(x)m(dx)\nonumber\\
&=\int_{{\mathbb{S}}^2} \Bigl({\frac{\Vert\zeta (T(x))\Vert}{\tan \Vert \zeta (T(x))\Vert}}\Bigl( \Vert v(T(x))\Vert^2-{\frac{(v(T(x))\cdot \zeta (T(x)))^2}{\Vert \zeta (T(x))\Vert^2}}\Bigr)\nonumber\\
&\quad  +{\frac{(v(T(x))\cdot \zeta (T(x)))^2}{\Vert \zeta (T(x))\Vert^2}}\Bigr) f(x)m(dx),\end{align}
\noindent and is evidently nonnegative. In the proof of Theorem \ref{Theorem1}, we use a more general version of this computation for manifolds.
\end{proof} 

\indent Now we consider triples $\{ f,\rho_0, \rho_1\}$ of probability density functions, where $f$ is regarded as a base point, and various costs of transporting one to another. In \cite{bib2}, the authors discuss generalized geodesics in ${\mathcal{W}}^2({\mathbb{R}}^d)$ with base point, and we have the added complication of the curvature of $M$. Given $y,x_0,x_1\in M$, we introduce a geodesic $\gamma^{(0)}(t)=\exp_y (t\xi_0)$ from $y=\gamma^{(0)}(0)$ to $x_0=\gamma^{(0)}(1),$ and likewise a geodesic $\gamma^{(1)}(t)=\exp_y (t\xi_1)$ from $y=\gamma^{(1)}(0)$ to $x_1=\gamma^{(1)}(1).$ There are two natural routes from $x_0$ to $x_1$:\par
\indent (1) let $\gamma^{(2)}$ be a geodesic from $x_0$ to $x_1$, so $\Delta yx_0x_1$ is a geodesic triangle; or\par
\indent (2) let $F(s,t) =\exp_y ( t(1-s)\xi_0+ts\xi_1)$, so $F(s,0)=y$, $t\mapsto F(s,t)$ is a geodesic, and $s\mapsto F(s,1)$ is a curve from $x_0$ to $x_1$. \par

\indent Let $T_0: M\rightarrow M$ and $T_1: M\rightarrow M$ be optimal transport maps such that $\rho_0=T_0\sharp f$ and $\rho_1=T_1\sharp f$, so each $y\in M$ gives a triple $\{ y, T_0(y), T_1(y)\}$. There are correspondingly two natural ways of interpolating between $\rho_0$ and $\rho_1$ according to whether we use (1) optimal transport from $\rho_0$ to $\rho_1$, or (2) Jacobi fields. When $T_0(y)=\exp_y (t\nabla \phi_0(y))$ and $T_1(y)=\exp_y (t\phi_1(y))$, the relevant Jacobi field is 
$$F(s,t) =\exp_x (t(1-s)\nabla \phi_0(y)+st\nabla\phi_1(y)).$$
For an interpolation as in (2), we have the following result.\par
\begin{prop}\label{proposition4} (i) For $M$ satisfying the geometrical hypotheses \ref{Geometrical Hypotheses}, there exists $\kappa_M>0$ such that given probability density functions $\rho_0, \rho_1,f$ on $M$, there exists a continuous path $(\rho_s)_{s\in [0,1]}$ in ${\mathcal{W}}^2(M)$ from $\rho_0$ to $\rho_1$ such that
\begin{align}(1-s)W_2^2(\rho_0, f)+sW_2^2(\rho_1,f)\geq W_2^2(\rho_s,f)+\kappa_M s(1-s)W_2^2(\rho_0,\rho_1)\qquad (s\in [0,1]).\end{align}
\indent (ii) In particular, $\kappa_{{\mathbb{S}}^2}=4/\pi^2$.\end{prop}
\begin{proof} In the case of ${\mathbb{S}}^2$, let $y=\exp_x \xi$ and $z=\exp_x\eta$ where $\xi, \eta\in T_s{\mathbb{S}}^2$ have angle $\theta$ between them; applying the cosine rule to the spherical triangle $\triangle xyz$, we
have
$$\cos d(y,z)=\cos \Vert \xi\Vert \cos\Vert\eta\Vert+\sin\Vert \xi\Vert \sin\Vert\eta\Vert\cos\theta$$
and we deduce by trigonometric identities that
\begin{align}\label{sphericaltrig}2\sin^2{\frac{d(y,z)}{2}}&=8\sin^2{\frac{\Vert\xi\Vert-\Vert\eta\Vert}{4}} \cos^2{\frac{\Vert\xi\Vert-\Vert\eta\Vert}{4}}+2\sin^2{\frac{\theta}{2}}\sin \Vert\xi\Vert \sin\Vert\eta\Vert,\nonumber\\
&\leq {\frac{1}{2}}\bigl( \Vert \xi\Vert-\Vert\eta\Vert\bigr)^2 +2\Vert\xi\Vert\Vert\eta\Vert \sin^2{\frac{\theta}{2}}\end{align}
in which $d(y,z)/2\leq \pi/2$. By simple estimates, we deduce that
\begin{align}\label{Csphere}d(y,z)^2\leq {\frac{\pi^2}{4}}\Vert\xi-\eta\Vert^2.\end{align}  
\indent The general case is based upon a similar idea. Let $-\phi_0$ and $-\phi_1$ be $d^2/2$ concave functions on $M$, and suppose that $T_0(x)=\exp_x\nabla\phi_0$ and $T_1(x)=\exp_x\nabla\phi_1(x)$ are optimal transport maps such that $\rho_0=T_0\sharp f$ and $\rho_1=T_1\sharp f$. Then it is natural to introduce $T_s(x) =\exp_x ((1-s)\nabla\phi_0(x)+s\nabla\phi_1(x))$ and $\rho_s=T_s\sharp f$ to interpolate between these; we do not assert that $T_s$ is an optimal transport map. One can construct a family of geodesics $t\mapsto F(s,t)$ all starting at $x$, given by  
$$F(s,t)=\exp_x (t(1-s)\nabla\phi_0(x)+st\nabla\phi_1(x)),$$
and determine $T_s(x)$ by Jacobi's equation on page 366 of \cite{bib29}; the curve $s\mapsto F(s,1)$ from $T_0(x)$ to $T_1(x)$ is not necessarily a geodesic.  Alternatively, one can find the minimizer of
$$\inf_x\bigl\{d(x,y)^2/2+(1-s)\phi_0(x)+s\phi_1(x)\bigr\}$$
\noindent since the infimum is attained for $x$ such that $y=T_s(x)$, and
 $$d^2(T_s(x),x)/2=\Vert (1-s)\nabla\phi_0(x)+s\nabla\phi_1(x)\Vert^2/2.$$ 
We have 
$$(1-s)\Vert\nabla\phi_0\Vert^2+s\Vert\nabla\phi_1\Vert^2-\Vert (1-s)\nabla\phi_0(x)+s\nabla\phi_1(x)\Vert^2=s(1-s)\Vert \nabla\phi_0-\nabla\phi_1\Vert^2.$$
then 
\begin{align}(&1-s)W_2^2(\rho_0, f)+sW_2^2(\rho_1,f)\nonumber\\
&=(1-s)\int_M\Vert \nabla \phi_0(x)\Vert^2f(x)m(dx)+s\int_M\Vert\nabla\phi_1(x)\Vert^2f(x)m(dx)\nonumber\\
&\geq\int_M    d^2(T_s(x),x) f(x)m(dx)/2+s(1-s)\int_M \Vert \nabla\phi_0(x)-\nabla\phi_1(x)\Vert^2 f(x)m(dx)\nonumber\\
&\geq W_2^2(\rho_s, f)+s(1-s)\int_M \Vert \nabla\phi_0(x)-\nabla\phi_1(x)\Vert^2 f(x)m(dx).\end{align}
With $\theta$ denoting the angle between $\nabla\phi_0$ and $\nabla\phi_1$, as measured with respect to the Riemannian metric on $M$,  we have by \cite{bib29} 14.1 
$$\Vert \nabla\phi_0-\nabla\phi_1\Vert^2=(\Vert \nabla\phi_0\Vert-\Vert\nabla\phi_1\Vert )^2+4\Vert\nabla\phi_0\Vert\Vert\nabla\phi_1\Vert\sin^2(\theta/2),$$
where with $\kappa_x$ the Gaussian curvature at $x$, 
$$d\Bigl( \exp_x{\frac{t\nabla\phi_0(x)}{\Vert \nabla\phi_0(x)\Vert}}, \exp_x{\frac{t\nabla\phi_1(x)}{\Vert \nabla\phi_1(x)\Vert}}\Bigr)^2=4\Bigl(\sin^2{\frac{\theta}{2}}\Bigr)\Bigl(t^2-{\frac{\kappa_xt^4 \cos^2(\theta/2)}{3}} +O(t^6)\Bigr)\quad (t\rightarrow 0+),$$
so with $t=\Vert \nabla \phi_0(x)\Vert$ and $s=\Vert\nabla\phi_1(x)\Vert$, we assume without loss of generality that $t\leq s$, so have
\begin{align}d\Bigl(& \exp_x\nabla\phi_0(x),\exp_x\nabla\phi_1(x)\Bigr)^2 \nonumber\\
&\leq 2d\Bigl( \exp_x{\frac{t\nabla\phi_0(x)}{\Vert \nabla\phi_0(x)\Vert}}, \exp_x{\frac{t\nabla\phi_1(x)}{\Vert \nabla\phi_1(x)\Vert}}\Bigr)^2+2d\Bigl(\exp_x{\frac{t\nabla\phi_1(x)}{\Vert \nabla\phi_1(x)\Vert}}, \exp_x{\frac{s\nabla\phi_1(x)}{\Vert \nabla\phi_1(x)\Vert}}\Bigr)^2\nonumber\\
&\leq  8\Bigl(\sin{\frac{\theta}{2}}\Bigr)^2\Bigl(t^2-{\frac{\kappa_xt^4 \cos^2(\theta/2)}{3}} +O(t^6)\Bigr)+2\Bigl( \Vert\nabla \phi_0(x)\Vert-\Vert\nabla \phi_1(x)\Vert\Bigr)^2\nonumber\\
&\leq (2+O(t^2)) \Vert \nabla\phi_0(x)-\nabla\phi_1(x)\Vert^2\nonumber\\
&\leq C\Vert \nabla\phi_0(x)-\nabla\phi_1(x)\Vert^2,\end{align}
\noindent for come uniform $C>0$ on $M$.  Finally, we have
\begin{align}W_2^2(\rho_0, \rho_1)&\leq \int_Md\Bigl( \exp_x\nabla\phi_0(x),\exp_x\nabla\phi_1(x)\Bigr)^2f(x)m(dx)\nonumber\\ 
 &\leq C\int_M \Vert \nabla\phi_0(x)-\nabla\phi_1(x)\Vert^2f(x)m(dx).\end{align}
 For $M={\mathbb{S}}^2$, the choice of $C$ is given by (\ref{Csphere}).
\end{proof} 
\indent In Proposition \ref{proposition4}, we have shown that the graph of $s\mapsto W_2^2(\rho_s, f)$ for $s\in [0,1]$ lies below an upward pointing parabola that intersects at 
$s=0$ and $s=1$; we have not quite shown that $s\mapsto W_2^2(\rho_s,f)$ is convex; compare with \cite{bib2}, section 9.1. Having considered convexity of $W_2$, we proceed in the next section to establish convexity of the internal energy.\par

\section{\label{Internal}Convexity of internal energy}
\noindent The internal energy for density $\rho$ is ${\mathcal U}(\rho )=\int_M \Theta (\rho (x))\rho(x)m(dx)$, and in this section we consider the convexity of ${\mathcal U}$ under transportation when $\Theta$ satisfies the following conditions.\par
\vskip.05in
\begin{defn}\label{Convexity Hypotheses} {\bf (Convexity Hypotheses)} We assume that $\Theta :(0, \infty )\rightarrow {\mathbb{R}}$ is four times continuously differentiable, and that\par
\indent (i)  $C_\infty:$ $r\mapsto \Theta (e^{r})$ is strictly increasing and convex;\par
\indent (ii) $\Theta (x)\rightarrow 0$ as $x\rightarrow 0+$ and $\Theta (x)\rightarrow \infty$ as $x\rightarrow\infty$; \par
\indent (iii)  $\Theta$ satisfies $(\Delta_2)$, in the sense of Orlicz's theory, so there exists $K_2>0$ such that $\Theta (2x)\leq K_2 \Theta (x)$ for all $x>0$. \par
The conditions (i), (ii) and (iii) imply that $\Theta_1(x)=x\Theta'(x)+\Theta (x)$ is positive and increasing. For application to (\ref{UPhi}), we later impose the additional conditions:\par
\indent (iv) $x\Theta_1'(x)$ is increasing;\par
\indent (v) $-1/(x\Theta'_1(x)^2)$ is convex.\end{defn}
\vskip.05in
\noindent{\bf Examples 4.1}  (1) For $1<\gamma <3/2$, $\Theta (x)=x^{\gamma -1}$ satisfies all conditions (i)-(v). This choice is physically relevant since for $\rho$ the density of a diatomic gas in ${\mathbb{R}}^3$, the specific heat capacity of air at constant pressure divided by the  specific heat capacity at constant volume is approximately $\gamma =\alpha +1=7/5$.  For a diatomic gas in ${\mathbb{R}}^3$, the assumption $\rho^{1/2}\in H^1({\mathbb{R}}^3)$ ensures that $\rho_0\in L^\gamma ({\mathbb{R}}^3).$ \par
\indent (2) For $\gamma >0$, the powers $\Theta (r)=r^{\gamma -1}$ satisfy (i)-(iii). For a monatomic gas, the ratio of specific heats is $\gamma =5/3$.\par
\indent (3) Also, $\Theta (r)=\log r$ satisfies (i) but not (ii). The expression $\rho \Theta (\rho )=\rho \log \rho$ is the integrand that is involved in the thermodynamic entropy of $\rho$. (The notion of $\Phi$-entropy is related to but different from this.) \par
\indent (4) Proposition \ref{proposition5} applies to $\Theta (x)=\log (1+x)$ and $\Theta (x)=x^\alpha$ for $\alpha >0$. However, it does not apply to $\Theta (x)=\sqrt{\log (1+x)}$, as discussed in \cite{bib4}\par
\indent (5) Under conditions (i) and (ii), the function $U:[0,\infty )\rightarrow {\mathbb{R}}$ given by  $U(x)=x\Theta (x)$ satisfies $U(0)=0$, $U(x)$ is convex and $r\mapsto e^r U(e^{-r })$ is convex, so $U$ belongs to Villani's class ${\mathcal{DC}}_\infty$ of \cite{bib29}; also $U$ is increasing and $U(x)\geq 0$, $U(x)\rightarrow \infty$ as $x\rightarrow\infty$, so $U$ is an Orlicz function, and gives rise to an Orlicz function space. The duality is most conveniently expressed in terms of $\Theta_1$ and its inverse function $\chi$. If $\Theta$ satisfies (iii) then $U$ also satisfies (iii). We provide details in Proposition \ref{proposition5} and Remark \ref{remark2}(iv).\par

\begin{prop}\label{proposition5}Suppose that $\Theta (e^{r})$ is strictly increasing and convex, with $\Theta (e^{r})\rightarrow \infty$ as $r\rightarrow \infty$ and $\Theta (e^{r})\rightarrow 0$ as $r\rightarrow-\infty$.  \par
\indent (i) Then the internal energy satisfies 
\begin{align}\int_M \rho (x)\Theta (\rho (x))m(dx)\leq 2\Theta (4)+2\int_{0}^\infty \Theta \Bigl({\frac{2}{ m\{ x: \rho (x)>\lambda\} }}\Bigr) m \{x: \rho (x)>\lambda \}\, d\lambda.\end{align}
\indent (ii) Suppose that   $f\in {\mathcal W}^2(M)$ satisfies $\int_Mf (x)\Theta (f (x))m(dx)<K$ for some $K>1$.  Then the set of probability density functions
\begin{align}\label{defOmega}\Omega_K=\Bigl\{ \rho \in L^1(M): \int_M \rho dm=1; W_2^2(\rho, f)+ \int_M\rho (x)\Theta (\rho (x))m(dx)\leq K\Bigr\}\end{align}
 is closed, convex and weakly sequentially compact in $L^1(M)$.\par
\indent (iii) The functional $E: \Omega_K\rightarrow {\mathbb{R}}$
\begin{align}{\mathcal E}(\rho ;f)= W_2^2(\rho, f)+ \int_M\rho (x)\Theta (\rho (x))m(dx)\end{align}
attains its infimum at a unique point $\rho_1\in \Omega_K$.\par
\indent (iv) If $x\Theta (x)$ and its Legendre transform satisfy the $\Delta_2$ condition, then $x\Theta (x)$ determines an Orlicz norm such that $\Omega_K$ is compact for the weak topology.\end{prop}
\begin{proof} (i) Let $\rho$ be a probability density function with respect to a positive Radon measure $m$ on $M$. Then $\rho$ has median $\mu_\rho$ such that $\mu_\rho m \{ x: \rho (x)> \mu_\rho \}=\mu_\rho/2\leq 1$ by Chebyshev's inequality. Also
\begin{align}\int_M \rho (x)\Theta (\rho (x))m (dx)&=\int_0^\infty (\Theta (\lambda )+\lambda\Theta'(\lambda ))m \{x: \rho (x)>\lambda \}\, d\lambda\nonumber\\
&\leq 2\int_0^\infty \Theta (2\lambda )m \{x: \rho (x)>\lambda \}\, d\lambda;\end{align}
and we split this integral as the sum of
\begin{align}\label{distintegral}2\int_{\mu_\rho}^\infty \Theta (2\lambda )m \{x: \rho (x)>\lambda \}\, d\lambda \leq  2\int_{\mu_\rho}^\infty \Theta \Bigl({\frac{2}{ m\{ x: \rho (x)>\lambda\} }}\Bigr) m \{x: \rho (x)>\lambda \}\, d\lambda,\end{align}
by Chebyshev's inequality, and 
\begin{align}\int_0^{\mu_\rho} \Theta (2\lambda )m \{x: \rho (x)>\lambda \}\, d\lambda\leq \Theta (2\mu_\rho )\int_{\mu_\rho}^\infty m\{ x: \rho (x)>\lambda \} d\lambda\leq \Theta (4).\end{align}
(ii) For $\rho\in \Omega_K$, we have 
$$\int_M d(x,x_0)^2 \rho (x)m(dx)\leq \int_M d(x, x_0)^2 f(x)m(x)+K=K_0+K,$$
and $\int_M \rho (x)\Theta (\rho (x))m(dx)\leq K$ for all $\rho\in \Omega_K$. Also, for $\rho_0, \rho_1\in \Omega_K$ let $\pi_0$ be an optimal transport plan taking $f$ to $\rho_0$, and $\pi_1$ an optimal transport plan taking $f$ to $\rho_1.$ Then for $0<t<1$, we introduce $\rho_t=(1-t)\rho_0+t\rho_1\in {\hbox{Prob}}(M)$, and we have
$$W_2(\rho_t, f )^2\leq (1-t)W_2(\rho_0, f)+tW_2(\rho_1, f),$$
since $(1-t)\pi_0+t\pi_1$ is a transport plan taking $f$ to $\rho_t$. Also
$$\int_M\rho_t(x)\Theta (\rho_t(x))m(dx)\leq (1-t)\int_M \rho_0(x)\Theta (\rho_0(x))m(dx)+t\int_M \rho_1(x)\Theta
 (\rho_1(x))m(dx)$$
since $r\Theta (r)$ is strictly convex. Hence $\Omega_K$ is convex. \par
\indent Given any sequence $(p_j)_{j=1}^\infty$ in $\Omega_K$, we have
$$\Theta (R)\int_{[p_j(x)>R]}p_j(x)m(dx)\leq \int p_j(x)\Theta (p_j(x))m(dx)\leq K$$
and 
$$R^2\int_{\{ x\in M: d(x, x_0)>R\}}p_j(x)m(dx)\leq \int_M d(x,x_0)^2 p_j(x)m(dx)\leq K+K_0$$
so  by the criterion of \cite{bib15} page 292 there exists a weakly convergent subsequence, with limit $p_\infty \in L^1$ and $p_\infty \in {\hbox{Prob}}(M)$. By Fatou's Lemma, we have 
\begin{align}W_2^2(p_\infty , f )&+\int_M p_\infty (x)\Theta (p_\infty (x))m(dx)\nonumber\\
&\leq \lim\inf_{j\rightarrow\infty} \Bigl( W_2^2(p_j, f)+\int_Mp_j(x)\Theta (p_j(x))m(dx)\Bigr)\leq K,\end{align}
\noindent so $p_\infty\in \Omega_K$. Hence $\Omega_K$ is weakly sequentially compact and closed in the weak topology.\par
\indent (iii) This is a marginal entropy transport problem in the sense of E8 of \cite{bib21}. Note that $f\in \Omega_K$, so the problem is feasible. The functional $E$ is nonnegative and weakly lower semicontinuous on the bounded convex and weakly compact set $\Omega_K$. Let $E_0=\inf\{ {\mathcal E}(\rho ;f): \rho\in \Omega_K\}$, and choose $(p_j)_{j=2}^\infty\in \Omega_K$ such that ${\mathcal E}(p_j;f)\rightarrow E_0$ as $j\rightarrow\infty$. Then there exists a weakly convergent subsequence, converging to $\rho_1\in \Omega_K$. Then by lower semi continuity, we have $E_0\leq {\mathcal E}(\rho_1;f)\leq \lim\inf_{j\rightarrow \infty } {\mathcal E}(p_j;f)=E_0$; hence $E$ attains its infimum at $\rho_1$.\par
\indent Suppose that $E$ attains the infimum at $\rho_0, \rho_1\in \Omega_K$ and $V=\{ x: \rho_0(x)\neq \rho_1(x)\}$ has $m(V)>0$. Then $x\Theta (x)$ is strictly convex, which implies  
\begin{align}\int_M &\rho_t(x)\Theta(\rho_t(x))m(dx)\nonumber\\
 &< (1-t)\int_M \rho_0(x)\Theta (\rho_0(x))m(dx)+t\int_M \rho_1(x)\Theta (\rho_1(x))m(dx)\qquad (0<t<1),\end{align} 
so  ${\mathcal E}(\rho_t;f)< (1-t){\mathcal E}(\rho_0;f )+t{\mathcal E}(\rho_1;f)={\mathcal E}(\rho_1;f)$, which contradicts the definition of $\rho_1$. This proves uniqueness.\par
\indent Let $x_0\in M$ and recalling that $M$ has finite diameter, choose $r_0$ such that $d(x,x_0)\leq r_0$ for all $x\in M$; then let $\kappa_0=1/m(M)$. Suppose that $\Theta$ is as in Proposition 4.2. Then $\kappa_0$ gives a probability density function on $M$ with respect to $m$, so 
$$W_2^2(\delta_{x_0}, \kappa_0dm)+\int_M \Theta (\kappa_0)\kappa_0m(dx)\leq r_0^2/2+\Theta (\kappa_0).$$ 
This provides $K$ such that $\Omega_K$ from (\ref{defOmega}) is nonempty. \par
\indent (iv) We have $e^{r}\Theta' (e^{r})\geq 0$ and $e^{2r}\Theta''(e^{r})+e^{r}\Theta'(e^{r})\geq 0.$ Then $U(x)=x\Theta(x)$ is convex on $[0, \infty )$, with derivative 
$\Theta_1  (x)=x\Theta'(x)+\Theta (x),$ which is nonnegative and strictly increasing. Now by convexity $x\Theta (x)+x\Theta_1(x)\leq 2x\Theta (2x),$
so $0\leq \Theta_1(x)\leq 2\Theta (2x)-\Theta (x)$, so by a sandwich argument we have $\Theta_1 (x)\rightarrow 0$ as $x\rightarrow 0+$. Hence $\Theta_1$ has an inverse function $\chi$, such that $\chi (0)=0$ and $\chi:[0, \infty )\rightarrow [0, \infty )$ is strictly increasing. 
The Legendre transform of the convex function $u\Theta (u)$ is
$$\sup_u \bigl( uv-u\Theta (u): u>0\} =(v-\Theta\circ \chi (v))\chi (v)=\int_0^v \chi (s)ds, \qquad (u,v>0).$$
The Orlicz norm associated with $x\Theta (x)$ on real measurable functions $u:M\rightarrow {\mathbb{R}}$ is 
$$\Vert u\Vert_{L_U}=\inf\Biggl\{ \lambda >0: \int_M {\frac{\vert u(x)\vert }{\lambda}}\Theta \Bigl( {\frac{\vert u(x)\vert }{\lambda}}\Bigr) m(dx)\leq 1\Biggr\},$$
which is equivalent to 
$$\sup \Biggl\{ \int_M u(x)v(x)m(dx): \int_M(v(x)-\Theta\circ\chi (v(x)))\chi (v(x))m(dx)\leq 1\Biggr\}$$
hence $\Vert \rho\Vert_{L_U}\leq K$ for all $\rho\in \Omega_K$.
If $x\Theta (x)$ and its dual function $\int_0^x \chi (t)dt$ both satisfy the $\Delta_2$ condition, then $L_U$ is reflexive as a Banach space, and the ball $\{ u\in L_U :\Vert u\Vert_{L_U}\leq K\}$ is compact for the weak topology.\par 
\end{proof} 
\begin{rem} \label{remark2} (Variational formulas) (i) In view of (\ref{distintegral}), we consider $\varphi (s)=s\Theta (2/s)$
and observe that if $\Theta$ is concave with $\Theta (0)=0$, then $\varphi$ is increasing. Indeed, using the mean value theorem, we find $\varphi'(s)=\Theta (0)+(2/s)(\Theta'(\xi )-\Theta'(2/s))$ for some $0<\xi <2/s$, so $\varphi'(s)\geq 0$.\par
\indent (ii) If $\Theta_1$ and its inverse $\chi$ both satisfy $\Delta_2$, then $x\Theta (x)$ at its Young conjugate function satisfy $\Delta_2$, so $L_U$ is reflexive.\par   
\indent (iii) Since  $\Theta (e^{r})$ and  $-\log r$ are convex, by Jensen's inequality and Csiszar's inequality we have
\begin{align}\int_M \Theta (m(M)\rho (x) )\rho (x)m(dx)&\geq \Theta \Bigl( \exp \int_M \log (m(M)\rho (x))\, \rho (x) m(dx)\Bigr)\nonumber\\
&= \Theta \Bigl[ \exp\Bigl({\hbox{Ent}}\Bigl(\rho dm\mid {\frac{dm}{m(M)}}\Bigr)\Bigr]\nonumber\\
&\geq \Theta\Bigl[ \exp \Bigl(\Bigl(2^{-1} \int_M \Bigl\vert\rho (x)-{\frac{1}{m(M)}}\Bigr\vert m(dx)\Bigr)^2\Bigr) \Bigr].\end{align}
\indent (iv) (Dual variational formula) The variational problem for finding $\rho_1$ in Proposition \ref{proposition5} can be expressed as a transport-relative entropy problem in which we seek a probability measure $\pi$ on $M\times M$ with marginals $\pi_1=\rho_1dm$ and $\pi_2=fdm$, where $f$ is fixed, while $\rho_1$ is an unknown density with respect to the reference measure $dm$. The density $\rho_1$ contributes to a relative entropy functional ${\mathcal U}(\rho_1)$ and there is a transportation cost involved in moving $f$ to $\rho_1$. Following \cite{bib21}, we can make the variational problem appear more symmetrical in $f$ and $\rho_1$ by introducing suitable entropy functions. 
Let $F_1(u)=u\Theta (u)$ have Legendre transform  $F_1^*: {\mathbb R}\rightarrow (-\infty, \infty ]$ by $F_1^*(u)=\sup_{s>0} \{ su-s\Theta (s)\}$; then introduce
$$F_1^\circ (s)=-F_1^*(-s)=\begin{cases}s\chi (-s)+\chi (-s)\Theta\circ \chi (-s),\qquad s<0;\\
0,\qquad s\geq 0.\end{cases}$$
Let $F_2(u)=0$ for $u=1$ and $F_2(u)=\infty$ otherwise; then $F_2^*(u)=u$ for all $u\in {\mathbb R}$ and we define $F_2^\circ (s)=-F_2^*(-s)=s$ for all $s\in {\mathbb R}$.
Then
\begin{align}{\mathcal E}(\rho_1; f)&=\inf_\rho \Bigl\{ {\mathcal U}(\rho )+W_2^2 (\rho, f): \rho dm\in {\hbox{Prob}}(M)\Bigr\}\nonumber\\
&=\inf_\pi\Biggl\{\int_M F_1(\rho (x))m(dx)+\int_M F_2(g(x))f(x)m(dx)\nonumber\\
&\qquad +{\frac{1}{2}}\int\!\!\!\int_{M\times M}d^2(x,y) \pi(dxdy): \pi_1=\rho dm; \pi_2=gfdm\Biggr\}\end{align}
\noindent in which $\int_M F_2(g(x))f(x)m(dx)=0$ in all finite cases.\end{rem} 
\vskip.05in
\begin{prop}\label{proposition6}The energy from this minimization problem may be expressed as the supremum of a dual functional 
\begin{align}\label{dual}{\mathcal E}(\rho_1; f)=\sup_{(\varphi_1, \varphi_2)}\Biggl\{ \int_M F_1^\circ (\varphi_1(x))m(dx)+\int_M F_2^\circ (\varphi_2(y))f(y)m(dy)\Biggr\},\end{align}   
where  $\varphi_1, \varphi_2: M\rightarrow {\mathbb R}$ are continuous and satisfy $\varphi_1(x)+\varphi_2(y)\leq (1/2)d^2(x,y)$ for all $x,y\in M$.\end{prop}
\begin{proof} This follows from Theorem 4.11 of \cite{bib21}, where the constraint on $\varphi_1$ and $\varphi_2$ amounts to $\varphi_1(x)\leq \varphi_2^c(x)$ for all $x\in M$, and by Corollary 4.12, we can assume that $\varphi_1, \varphi_2$ are continuous.  Note that
$F_1^*(u)=u\chi (u)-\chi (u)\Theta\circ\chi (u)\geq 0$ since $\Theta_1(s)\geq \Theta (s)$ so in the first integral $F_1^\circ (s)=0$ for $s>0$, and $F_1^\circ (s)\leq 0$ for all $s\in {\mathbb R}$.
In the  final integral,  $F_2^\circ (\varphi_2(y))f(y)$ increases with $\varphi_2(y)$, and  if $\varphi_2(y)>d^2(x,y)/2$, then $\varphi_1(x)\leq \varphi_2^c(x)<0$, so $F_1^\circ (\varphi_1(x))<0$.\par 
\indent In particular, with $u\Theta (u)=u^\gamma/\gamma$ and $\gamma^*$ such that $1<\gamma, \gamma^*<\infty$ and $1/\gamma+1/\gamma^*=1$, we have
$$F_1^\circ (u)=\begin{cases} -(-u)^{\gamma^*}/\gamma^*, \qquad u<0;\\
0,\qquad u\geq 0.\end{cases}$$
\end{proof} 
\begin{rem}\label{remark3} (Convexity criteria) The following proofs use convexity calculations with some potentially confusing signs, which we resolve here.  Let $\Psi : {\mathbb{R}}^n\rightarrow M_m({\mathbb{R}})$ be a matrix function such that $\Psi (x)$ is positive definite for all $x\in {\mathbb{R}}^n$. Then the Hessian in $(t,x)$ of $e^t\Psi (x)$ satisfies 
$$D^2_{(t,x)} e^t\Psi (x)=e^t\begin{bmatrix}\Psi (x)& \nabla\Psi (x)\cr (\nabla\Psi (x))^\dagger& D_x^2 \Psi (x)\end{bmatrix},$$
so the Schur complement of $\Psi (x)$ in this matrix is $D_x^2\Psi (x)-(\nabla\Psi (x))^\dagger\Psi(x)^{-1}(\nabla\Psi (x)).$ Then by \cite{bib18} page 472:\par
\indent (a) $D_x^2\Psi (x)-(\nabla\Psi (x))^\dagger\Psi(x)^{-1}(\nabla\Psi (x))$ is positive definite, if and only if  $D^2_{(t,x)} e^t\Psi (x)$ is positive definite; whereas\par
\indent (b) if $-D_x^2\Psi (x)+(\nabla\Psi (x))^\dagger\Phi(x)^{-1}(\nabla\Psi (x))$ is positive definite, then  $-\log\det \Psi (x) $ is convex. We use this in Lemma \ref{lemma1}.\par 
\indent (c) Suppose that $\Phi :[0, \infty )\rightarrow [0, \infty )$ is $C^4$ and $\Psi=\Phi''>0$, so $\Phi$ is strictly convex. Then $\Phi$ is admissible in the sense of (\ref{ententprod}), if and only if $-1/\Phi''(x)$ is convex, if and only if $(x,y)\mapsto \Phi''(x)y^2$ is convex, so 
$$\begin{bmatrix} 2\Phi''(x)& 2\Phi'''(x)y\\ 2\Phi'''(x)y& \Phi^{(4)}(x)y^2\end{bmatrix}$$
is positive semidefinite. From this criterion, it follows that if $\Phi_1$ and $\Phi_2$ are admissible, then $\Phi_1+\Phi_2$ and $\lambda \Phi_1$ are also admissible for all $\lambda>0$. This condition is used in $\Phi$-entropy (\ref{ententprod}), and shows that $\rho\mapsto \int \Phi''(\rho )\Vert\nabla\rho\Vert^2m(dx)$ is convex for the pointwise linear structure on ${\hbox{Prob}}(M)$. We use this generalized information in Proposition \ref{proposition8}  and Theorem \ref{Theorem1}, where we compute the Hessian of ${\mathcal{U}}(\rho )$.\end{rem}
\indent The following results discuss the convexity of the internal energy under the flows considered in section 2.\par
\begin{lem}\label{lemma1}Suppose that is $\Theta (e^{r})$ is convex and increasing on $[0, \infty )\rightarrow [0, \infty )$ , and let $t\mapsto \varphi (t,x)$ be a curve on $M$ starting at $x=\varphi (0, x)$, and let $\Delta (t,x)=\det D_x \varphi.$  
Let $f_t$ be the probability measure induced from $\rho_0$ by $\varphi (t, \cdot )$.\par
\indent (i)  If $t\mapsto -\log \Delta (t, x)$ is convex, then $t\mapsto {\mathcal U}(f_t)$ is also convex, where
\begin{align}t\mapsto {\mathcal U}(f_t)= \int_{{\mathbb{S}}^2} \Theta (f_t(x))f_t(x) m(dx))\qquad (t>0).\end{align}
\indent (ii) If $\Theta (r)=r^{\gamma -1}$ with $\gamma >1$ and $t\mapsto 1/\Delta (t,x)^{\gamma -1}$ is convex, then $t\mapsto {\mathcal U}(f_t)$ is also convex.\end{lem}
\begin{proof} (i) We consider $\rho_0$, and a smooth vector field $\xi: M\rightarrow TM$ that generates a flow $\varphi_t$ of continuous maps $\varphi_t:M\rightarrow M$ such that $\dot \varphi_t(x)=\xi (\varphi_t(x))$ on $M$, so $\varphi_t$ induces a probability measure $f_t$ from $\rho_0$. Then 
\begin{align}\int_M \bigl( \Theta (f_t(x))f_t(x)-&\Theta (\rho_0(x))\rho_0(x)\bigr) m(dx)\nonumber\\
&= \int_M \bigl( \Theta (f_t(\varphi_t(x))-\Theta (\rho_0(x))\bigr)\rho_0(x)m(dx)\nonumber\\
&= \int_M \Bigl( \Theta \Bigl({\frac{\rho_0(x))}{\det (D\varphi_t(x))}}\Bigr) -\Theta (\rho_0(x))\Bigr)\rho_0(x)m(dx)\nonumber\\
&=-t\int_M \Bigl({\frac{d}{dt}}\Bigr)_{t=0}\det \bigl(D\varphi_t(x) \bigr)\, \Theta'(\rho_0(x))\rho_0(x)^2 m(dx)+o(t)\end{align} 
\noindent as $t\rightarrow 0+$. For $\dot \varphi_t(x)=\xi (\varphi_t(x))$, we have 
\begin{align}\label{dotdet}\Bigl({\frac{d}{dt}}\Bigr)_{t=0}\det \bigl(D\varphi_t(x)\bigr)={\hbox{div}}\,\xi (x)\end{align}
so the integral becomes
$$ \int_M \bigl( \Theta (f_t(x))f_t(x)-\Theta (\rho_0(x))\rho_0(x)\bigr) m(dx)=-t\int_M {\hbox{div}}\, \xi (x)\Theta' (\rho_0(x))\rho_0(x)^2\, m(dx)+o(t).$$
 We have $f_t(\varphi (t,x))\Delta (t,x)=\rho_0(x)$, hence
$$\int_M\Theta (f_t(x))f_t(x)m(dx)= \int_M \Theta\Bigl( {\frac{\rho_0(x)}{\Delta (t,x)}}\Bigr)\rho_0(x)m(dx)$$
so 
\begin{align}{\frac{d^2}{dt^2}}\int_M &\Theta (f_t(x))f_t(x)m(dx)\nonumber\\
&=\int_M\Bigl[\Theta''\Bigl( {\frac{\rho_0(x)}{\Delta (t,x)}}\Bigr) {\frac{\rho_0(x)}{\Delta (t,x)}}+\Theta'\Bigl( {\frac{\rho_0(x)}{\Delta (t,x)}}\Bigr)\Bigr]  \Bigl( {\frac{\dot\Delta (t,x)^2}{\Delta (t,x)^3}}\Bigr)\rho_0(x)^2m(dx)\nonumber\\
&\quad +\int_M\Theta'\Bigl( {\frac{\rho_0(x)}{\Delta (t,x)}}\Bigr)\Bigl( {\frac{\dot\Delta (t,x)^2}{ \Delta (t,x)^2}}-{\frac{\ddot\Delta (t,x)}{\Delta (t,x)}}\Bigr){\frac{\rho_0(x)^2}{\Delta (t,x) }} m(dx).\end{align}
Under the hypotheses of the Lemma, both these integrals are non negative. \par
\indent (ii) Here we have 
$$\gamma\Bigl({\frac{\dot \Delta}{\Delta}}\Bigr)^2-{\frac{\ddot\Delta}{\Delta}}\geq 0,$$
so we can use a similar proof to (i)
\end{proof} 
\indent  Suppose that $S\in M_n({\mathbb{R}})$ is symmetric, and $S<I_n$. Then $t\mapsto \det  (I_n-tS)^{1/n}$ is concave, by Lemma 5.21 of \cite{bib28}, so $t\mapsto \det  (I_n-tS)^{1-\gamma}$ is convex for all $\gamma -1\geq 1/n.$ This result is decisive when one considers displacement convexity on ${\mathbb{R}}^n$; however, in the context of Riemannian manifolds, the formula for $\Delta (t,x)$ is more complicated, as we see in Proposition \ref{proposition10} below.\par
\indent In Section \ref{Minimizer}, we will use the results of Section \ref{Internal} to determine the minimizer of ${\mathcal{E}}(\rho; f)$. However, we need to establish some further properies of $f$, as we do in the next Section \ref{Specgap}.\par  

\section{\label{Specgap}Spectral gaps and $\Phi$-entropy}
\begin{defn}Say that the density $\rho$ satisfies a spectral gap or Poincar\'e inequality, if there exists $\lambda_1>0$ such that
\begin{align}\label{Spec}\lambda_1\int_M \Bigl( g(x)-\int_M g\rho dm\Bigr)^2\rho(x)m(dx)\leq \int_M \Vert \nabla  g(x)\Vert^2 \rho (x)m(dx)\qquad (g\in C^\infty (M; {\mathbb{R}})).\end{align}\end{defn}
\begin{prop}\label{proposition7}Suppose that $\rho_0$ and $f_t$ are probability density functions on $M$, where $M$ has dimension $n$, with  positive curvature $\kappa\leq Ric_x\leq K$ and that $\phi$ is a $C^2$ function such that $\varphi_t(x)=\exp_x (t\nabla \phi (x) )$ induces $f_t$ from $\rho_0$. Suppose  $\rho_0$ is uniformly positive and uniformly bounded on $M$.\par
\indent (i)  Then there exists $h>0$ such that $f_t$ is also uniformly positive and uniformly bounded for all $0<t\leq h$, so $\delta_1\leq f_t\leq \delta_1^{-1}$. \par
\indent (ii) Also $f_t$ satisfies a spectral gap condition (\ref{Spec}) with $\lambda_1=\kappa n\delta_1/(n-1)$.
\end{prop}  
\begin{proof}  (i) First note that $\inf_x\bigl\{ d(x,y)^2/2+t\phi (x): x\in M\}$ has a unique minimizer $x$ such that $y=\exp_x(t\nabla \phi (x))$, so $\varphi_t:M\rightarrow M$ gives a bijective map for $0\leq t\leq t_0$ and some $t_0>0$. Now we follow the calculations by Villani on pages 365-9 of \cite{bib31}. For the geodesic $t\mapsto \exp_x(t\nabla\phi (x))$, let $A(t,x)$ be the solution of Jacobi's equation $(d^2/dt^2)A(t,x)+R(t,x)A(t,x)=0$ where $A(t,x)$ is an $n\times n$ matrix with $A(0,x)=I_n$ and $R(t,x)$ is a symmetric matrix given by Riemann's curvature tensor. Then $U(t,x)=A'(t,x)A(t,x)^{-1}$ is a symmetric matrix with trace
$$u(t,x)={\hbox{trace}}\, U(t,x)=(d/dt)\log\det A(t,x)$$
and $\kappa I\leq Ric_x\leq KI$, so
\begin{align}\label{udiffineq}{\frac{d}{dt}} u(t,x) +{\frac{1}{n}}u(t,x)^2 +n\kappa\leq 0;\end{align}
so integrating this differential inequality, we have 
$$u(t,x)\leq {\frac{u(0,x)-n\sqrt{\kappa} \tan (\sqrt{\kappa} t)}{ 1+(u(0,x)/(n\sqrt{\kappa}))\tan \sqrt{\kappa}t}},$$
then 
$$\det A(t,x)\leq \exp\Bigl( \int_0^t {\frac{u(0,x)-n\sqrt{\kappa} \tan (\sqrt{\kappa} s)}{ 1+ (u(0,x)/(n\sqrt{\kappa}))\tan\sqrt{\kappa}s}}ds\Bigr).$$
We also have $u(x,t)\geq -n\Vert U(t,x)\Vert$, where by the triangle inequality 
$$(d/dt)\Vert U(t,x)\Vert\leq \Vert (d/dt)U(t,x)\Vert\leq \Vert R\Vert +\Vert U(x,t)^2\Vert,$$
so 
$${\frac{ (d/dt) \Vert U(t,x)\Vert}{ K+\Vert U(t,x)\Vert^2}}\leq 1,$$
hence, integrating the differential inequality, 
$$\Vert U(t,x)\Vert \leq {\frac{\Vert U(0,x)\Vert +\sqrt{K}\tan (\sqrt{K}t)}{ 1-K^{-1/2}\Vert U(0,x)\Vert \tan (\sqrt{K}t)}}.$$
Then we have 
$$\det A(t,x)= \exp \Bigl(\int_0^tu(s,x)\, ds\Bigr)\geq \exp\Bigl( -n \int_0^t {\frac{\Vert U(0,x)\Vert +\sqrt{K}\tan (\sqrt{K}s)}
{ 1-K^{-1/2}\Vert U(0,x)\Vert \tan (\sqrt{K}s)}}ds\Bigr),$$
which is valid for $K^{-1/2}\Vert U(0,x)\Vert \tan (\sqrt{K}t)<1,$
so we choose $h>0$ small enough that $t$ satisfies this for all $0<t\leq h$ and all $x\in M$. \par
\indent Then $\delta_0<\rho_0(x)<1/\delta_0$ for all $x\in M$ and some $\delta_0>0$, and 
\begin{align}\label{Jacobidet}f_t(\varphi_t(x))\det A(t,x)=\rho_0(x)\qquad (x\in M, 0<t<h)\end{align}
so there exists $\delta_1>0$ such that $\delta_1<f_t(y)<1/\delta_1$ for all $y\in M$ and  $0<t<h.$\par
\indent (ii) So by the Ledoux's theorem \cite{bib20}, $m$ satisfies the logarithmic Sobolev inequality with constant $\kappa n/(n-1)$, so that 
\begin{align}\label{logsob}\int_M g(x)^2 \log \Bigl( g(x)^2/\int g^2dm\Bigr) m(dx)\leq {\frac{2(n-1)}{n\kappa }}\int_M \Vert \nabla g(x)\Vert^2m(dx)\end{align} 
\noindent for all $g\in C^\infty (M; {\mathbb{R}})).$ Now by(i),  $\delta_1^{-1}\geq f_t(x) \geq \delta_1>0$ for all $x\in M$. Then by the Holley--Stroock perturbation theorem \cite{bib14}, $f_t$ satisfies a logarithmic Sobolev inequality (\ref{logsob}) with constant $\delta_1\kappa n/(n-1)$, hence $f_t$ satisfies a spectral gap inequality (\ref{Spec}) with $\lambda_1=\delta_1n\kappa /(n-1)$.
\end{proof} 
For completeness, we give the following known result for power laws, before extending to more general choices of the internal energy.\par

\begin{lem}\label{lemma2} (i) Suppose that $\rho_0$ is a probability density function such that $\rho_0^{\beta -1/2}\in H^1(M)$ for some $\beta>(n-1)\gamma/n$ where $M$ satisfies the geometrical hypotheses \ref{Geometrical Hypotheses}. Then $\rho_0\in L^\gamma (M)$.\par
\indent (ii) In particular, suppose that for $1<\gamma <3/2$, and $\Phi (r)=r^{2\gamma -1}$, the Fisher $\Phi$ relative information
\begin{align}{\mathcal{I}}_\Phi (\rho \mid m)=\int_M\Phi''(\rho_0(x))\Vert \nabla \rho_0 (x)\Vert^2 m (dx)\end{align}
is finite. Then with $\Theta (r)=r^{\gamma -1}$, the internal energy ${\mathcal{U}}(\rho_0)$ is finite.\end{lem} 
\begin{proof} Let $H_t$ be the $(n-1)$-dimensional Hausdorff measure on $\{ x\in M: \rho_0(x)^\beta =t\}$ for $t>0$. By the Cauchy--Schwarz inequality we have
\begin{align}\Bigl( \int_M\rho_0(x)^{2\beta -3}\Vert \nabla &\rho_0(x)\Vert^2m(dx)\Bigr)^{1/2}\Bigl( \int_M\rho_0(x)m(dx)\Bigr)^{1/2}\nonumber\\
&\geq \int_M \Vert \nabla \rho_0(x)^\beta\Vert m(dx)\end{align}
which by the co area formula \cite{bib7} page 25, 
\begin{align}
&=\int_0^\infty \int_{\{x: \rho_0(x)^\beta =t\}} H_{t}(dx) \, dt\nonumber\\
&\geq  \iota_n(M)\int_0^\infty m(\{ x: \rho_0(x)^\beta >t\})^{(n-1)/n}dt\nonumber\\
& = \iota_n(M)\beta\int_0^\infty u^{\beta -1}m(\{ x: \rho_0(x) >u\})^{(n-1)/n}du\end{align}
\noindent by the isoperimetric inequality, so by Chebyshev's inequality, we deduce that $m(\{ x: \rho_0(x) >u\})\leq Cu^{-\beta n/(n-1)}$
for some constant $C>0$ and all $u>0$; then
\begin{align}\int_0^\infty \gamma u^{\gamma -1}&m(\{ x: \rho_0(x) >u\})du\nonumber\\
&\leq \gamma \int_0^1 m(\{ x: \rho_0(x) >u\})du +\int_1^\infty u^{\gamma -1} m(\{ x: \rho_0(x) >u\})du\nonumber\\
&\leq \gamma\int_0^\infty m(\{ x: \rho_0(x) >u\})du+C\int_1^\infty u^{\gamma -\beta n/(n-1)-1}du,\end{align}
which converges; hence $\rho_0\in L^\gamma (M)$.\par
\indent (ii) We can take $\beta =\gamma$, and apply (i).
\end{proof} 
\begin{prop}\label{proposition8} Suppose that $\Theta$ satisfies (i)-(v) of the convexity hypotheses \ref{Convexity Hypotheses}, and that $M$ satisfies the geometrical hypotheses  \ref{Geometrical Hypotheses}. Let $\Phi (r)=\int_0^r (r-u)u\Theta_1'(u)^2du$. Then there exists a continuous and strictly increasing function $\varphi :[0, \infty )\rightarrow [0, \infty )$ such that $r\Theta (r)=\varphi (\Phi (r))$, and the internal energy satisfies 
$${\mathcal{U}}(\rho) \leq\varphi \Bigl( \int_M\Phi (\rho (x)) m(dx)\Bigr),$$
where the $\Phi$-entropy is bounded by
\begin{align}\label{PhiFisher}\int_M\Phi (\rho (x)) m(dx)-\Phi\Bigl( \int_M\rho (x) m(dx)\Bigr)\leq {\frac{1}{2\kappa_0}} \int_M \rho (x) \Vert \nabla (\Theta_1\circ \rho )(x)\Vert^2 m(dx),\end{align}
\noindent for all probability densities $\rho$ such that the Fisher $\Phi$- relative information is finite.\end{prop}
\begin{proof} We observe that $\Phi''(r)=r\Theta_1'(r)^2$, so $\Phi$ is strictly increasing and convex, and we can introduce a continuous and strictly increasing $\varphi$ by the formula
$r\Theta (r)=\varphi (\Phi (r))$ via the implicit function theorem.  Now $u\Theta_1'(u)$ is increasing, so $\varphi $ is concave; indeed, we have
$$\varphi''\circ \Phi (r)\Phi'(r)^2= \Theta_1'(r)-{\frac{\Theta_1(r)\Phi''(r)}{\Phi'(r)}},$$
where 
$$\Theta_1'(r)\Phi'(r)-\Theta_1(r)\Phi''(r)=\Theta_1'(r) \int_0^r u\Theta_1'(u)^2du-r\Theta_1(r)\Theta_1'(r)^2\leq 0$$
\noindent since  $\int_0^r u\Theta_1'(u)^2du/\int_0^r \Theta_1'(u)du\leq r\Theta_1'(r)$. Then by Jensen's inequality with $\rho\Theta (\rho )=\varphi\circ \Phi (\rho )$, the internal energy satisfies
\begin{align}\label{UPhi}{\mathcal{U}}(\rho)=\int_M \varphi \circ \Phi (\rho (x))m(dx)\leq\varphi \Bigl( \int_M\Phi (\rho (x)) m(dx)\Bigr).\end{align}
 Under geometrical hypotheses  \ref{Geometrical Hypotheses} discussed above, the Laplace operator on $M$ satisfies the Bakry-Emery curvature condition $DC(\kappa_0,\infty )$ where $\kappa_0>0$ is a lower bound on the Ricci curvature; see \cite{bib29}. Then by Theorem 2 of \cite{bib34}, we have an entropy-entropy production inequality for Riemannian measure $m$, so we have (\ref{PhiFisher}).
\end{proof} 
\indent In particular, we can choose $1\leq \gamma <3/2$ and can take $\Theta_1(r)=r^{\gamma -1}$ which gives 
$$\Phi (r)=\int_0^r (r-u) u\Theta_1'(u)^2du={\frac{r^{2\gamma -1}}{2(2\gamma -1)}}$$
\noindent which is admissible since $-1/\Phi''(r)$ is convex. This range of $\gamma$ includes $\gamma =7/5$, the ratio of specific heats  associated with a diatomic gas, and Lemma \ref{lemma2} applies.
The proof of the Bakry-Emery theorem for diffusions is discussed by Villani 9.2.2 \cite{bib28}, and Proposition \ref{proposition8} is an analogue for the internal energy functional.   
The main result of the folowing section is Theorem \ref{Theorem1}, and (iii) of that theorem is a type of converse to (\ref{PhiFisher}), for a specially chosen flow. 
\begin{prop}\label{proposition9} Suppose that $\rho_0$ satisfies a spectral gap inequality (\ref{Spec}). \par
\indent (i) Then for all $C^1$ vector fields $v:M\rightarrow TM$ there exists a decomposition $v=\nabla \phi+w$ where $\phi\in L^2(\rho_0)$ and $w:M\rightarrow TM$ is a $L^2$ vector field such that
\begin{align} \Bigl({\frac{d}{dt}}\Bigr)_{t=0}\int_M g(\exp_x(tw(x))) \rho_0(x) m(dx)=-\int_M \phi (x)\nabla\cdot (\rho_0(x)\nabla g(x))m(dx)\end{align}
\noindent for all $g\in C^\infty (M; {\mathbb{R}})$. \par
\indent (ii) In particular, (i) holds for the constant density $\rho_0=1/m(M)$ with $\lambda_1(M)\geq \iota_\infty(M)^2/4$ as in (\ref{Cheeger}).\end{prop}
\begin{proof}  (i) Define $L$ by $Lf=\rho_0^{-1}\nabla\cdot (\rho_0\nabla f)$. Then as in \cite{bib14}, we have 
$$\lambda_1\int_M\Vert\nabla \phi (x) \Vert^2\rho_0(x)m(dx)\leq \int_M (L\phi (x))^2\rho_0(x)m(dx),$$
\noindent so the nullspace of $L$ is the space of constants. 
Given a suitably smooth vector field $v$, we have $\int_M\rho_0^{-1} \nabla\cdot (\rho_0v) \rho_0dm=0$ by the divergence theorem, so there exists $\phi$ such that $L\phi=\rho_0^{-1}\nabla \cdot (\rho_0v)$; hence $v=\nabla \phi +w$ for some vector field $w:M\rightarrow TM$ such that $\nabla \cdot(\rho_0w)=0$. Then 
\begin{align}\int_M g(\exp_x(tv(x))\rho_0(x)m(dx)&=\int_M g(x)\rho_0(x)m(dx)+t\int_M\langle \nabla\phi (x) \nabla g(x)\rangle \rho_0(x)m(dx)\nonumber\\
&\quad+t\int_M \langle \nabla_M g(x), w(x)\rangle \rho_0(x)m(dx)+o(t)\end{align}
as $t\rightarrow 0$, where the final integral is zero; so we obtain the stated result by the divergence theorem.\par 
\indent (ii)  Then the smallest positive eigenvalue of $M$ satisfies $\lambda_1(M)\geq \iota_\infty(M)^2/4$ by \cite{bib7}, where $\iota_\infty (M)\geq\iota (M)/m(M)^{1/n}>0$ as in geometrical hypotheses (\ref{Geom}). Hence there is a spectral gap inequality for $L^2(m)$, and consequently a Helmholtz decomposition as above.
\end{proof} 
\section{\label{Minimizer}The minimizer of the energy}
\noindent The main result of this section gives specific information about the minimizer $\rho_1$ from Proposition \ref{proposition5}; in particular, we show  that $\log\rho_1(x)$ is bounded on $M$. This rules out the possibility that the density of gas in an atmosphere decreases to zero through the formation of a vacuum, or that the gas density becomes very large as the gas passes to a liquid state. As in section \ref{Specgap}, we show that  a spectral gap or Poincar\'e type inequalities holds in  $L^2(\rho_1 )$ which implies that the support of $\rho_1$ is connected; see \cite{bib28}. The following result extends a version that Cullen and Gangbo \cite{bib12} achieved for 
$\Theta (\rho )=\rho$, where $\rho$ is defined on a bounded region $\Omega$ of ${\mathbb{R}}^n$.\par 
\vskip.1in
\begin{thm}\label{Theorem1}Let $f$ be a probability density function on $M$ as in Proposition \ref{proposition8}, such that $\delta_1 \leq f (x)\leq 1/\delta_1$ for all $x\in M$ and some $\delta_1>0$. Let $\Theta$ satisfy (i)-(iii) of the convexity hypotheses  \ref{Convexity Hypotheses}. In particular, one can choose $\Theta (r)=r^{\gamma -1}$ for $\gamma >1$.\par
\indent (i) Then the minimization problem
\begin{align}\inf_g\Biggl\{ W_2^2(g, f )+h^2\int_M  g(x)\Theta (g (x)) m(dx): g \in L^1(m)\Biggr\}\end{align}
\noindent over probability density functions has a unique solution $\rho_h$ such that 
 $\delta \leq \rho_h(x)\leq 1/\delta$ for all $x\in M$ and all $0<\delta <\delta_1$. Also, $\rho_h$ satisfies a spectral gap inequality for some constant $\lambda_1>0$.\par
\indent (ii) With $\Theta_1(r)=r\Theta'(r)+\Theta (r)$, the solution $\rho_h$ is such that
\begin{align}\label{findingphi}-h^2\phi_h (y)=\inf_x \{d^2(x,y)/2+h^2\Theta_1(\rho_h(x))\}\end{align}
 is $d^2/2$ concave, and $T_h(x)=\exp_x(h^2\nabla \phi_h (x))$ induces $\rho_h$ from $f$, while\par
\noindent  $T_h^*(x)=\exp_x\bigl(h^2\nabla (\Theta_1\circ\rho_h)(x)\bigr)$ induces $f$ from $\rho_h$;\par
\indent (iii)  the internal energy satisfies
\begin{align}{\mathcal U}(f)\geq {\mathcal U}(\rho_h)+h^2\int_M \rho_h(x)\Vert \nabla ( \Theta_1\circ \rho_h)(x)\Vert^2m(dx).\end{align}
\end{thm}
\begin{proof}  (i) Existence of $\rho_h$ follows from Proposition \ref{proposition5}. We have $T_h\sharp f=\rho_h$ and $T_h^*\sharp \rho_h=f$. We introduce the set $E=\{x\in M: \rho_h(x)<\delta\}$, and we aim to show that $m(E)$=0. To this end, we introduce the  positive measures 
$$e_0={\bf I}_{E\cap T_h^*(E^c)}f ,\quad e_1={\bf I}_{E^c\cap T_h^*E)}\rho_0$$
where $E^c=M\setminus E$, so $T_h\sharp e_0=e_1$ and $T_h^*\sharp e_1=e_0$.  Suppose first that $m (E\cap T_h^*(E^c))=0,$ or equivalently that $e_0=0$. Then
\begin{align}\delta_1m(E)&\leq \int_Ef (x)m(dx)= \int_{E\cap T_h^*(E)}f (x)m(dx)\nonumber\\
&\leq\int_{T_h^*(E)}f (x)m(dx)=\int_E \rho_1(x)m(dx)\nonumber\\
&\leq \delta m(E)\end{align}
\noindent which implies that $m(E)=0$. \par
\indent Assume otherwise, that $m (E\cap T_h^*(E^c))>0,$ and for $0<\varepsilon <\delta$ consider the probability measure $g_\varepsilon =\rho_h+\varepsilon (e_0-e_1);$
clearly $\int g_\varepsilon m(dx)=1,$ and $g_\varepsilon\geq \rho_h-\varepsilon e_1\geq (1-\varepsilon )\rho_0$. Likewise $f -\varepsilon e_0\geq 0$, so we can introduce a transport plan
$$\gamma_\varepsilon (dxdy)=(id\times T_h)\sharp (f-\varepsilon e_0)+\varepsilon (id\times id)\sharp e_0$$
which has marginals $f$ and $g_\varepsilon$. Then from the definition of Wasserstein metric, and the choice of $T_h$, we have
\begin{align}W_2^2(g_\varepsilon, f )&\leq {\frac{1}{2}}\int_M d(x,y)^2 \gamma_\varepsilon(dxdy)\nonumber\\
&={\frac{1}{2}}\int_M d(x,T_h(x))^2 f (x) m(dx)-{\frac{\varepsilon}{2}}\int_Md(x,T_h(x))^2e_0(dx)\nonumber\\
&=W_2^2(\rho_h, f )-{\frac{\varepsilon}{2}}\int_{E\cap T^*(E^c)} d(x,T_h(x))^2 f (x)m(dx).\end{align} 
Let  $\Psi (r)=h^2r\Theta (r)$, with $\Theta$ as in the Theorem, so $\Psi :[0, \infty )\rightarrow [0, \infty )$ is a convex and increasing function such that $\Psi (2r)\leq C\Psi (r)$ for some $C>0$ and all $r>0$. Using the mean value theorem, we have
$$\Psi(g_\varepsilon)=\Psi (\rho_1)+(g_\varepsilon-\rho_1)\Psi'(\rho_1)+(1/2)(g_\varepsilon-\rho_1)^2\Psi''(\bar g_\varepsilon )$$
for some $\bar g_\varepsilon$ between $\rho_1$ and $g_\varepsilon$. We have $C\Psi (r)-\Psi (r)\geq \Psi (2r)-\Psi (r)\geq r\Psi'(r),$
so from the choice of $\rho_h$ and the definition of $g_\varepsilon$, it is easy to see that the first four terms in this equation are integrable, hence the final term involving $\Psi''$ is also integrable. Since $\rho_h$ was chosen as a minimizer, we have
\begin{align}W_2^2(&\rho_h, f)+\int_M \Psi(\rho_h(x)) m(dx)\nonumber\\
&\leq W_2^2(g_\varepsilon, f)+\int_M \Psi (g_\varepsilon (x))m(dx)\nonumber\\
&=W_2^2(\rho_h, f)+\int_M \Psi (\rho_h(x))m(dx))-{\frac{\varepsilon}{2}}\int_{E\cap T_h^*(E^c)} d(x,T_h(x))^2 f (x)m(dx)\nonumber\\
&\quad +\int_M(g_\varepsilon-\rho_0)\Psi'(\rho_h)m(dx)+\int_M (1/2)(g_\varepsilon-\rho_0)^2\Psi''(\bar g_\varepsilon )m(dx)\nonumber\\
&=W_2^2(\rho_h, f)+\int_M \Psi (\rho_h(x)) m(dx)-{\frac{\varepsilon}{2}}\int_{E\cap T_h^*(E^c)} d(x,T_h(x))^2 f (x)m(dx)\nonumber\\
&\quad+\varepsilon \int_{E\cap T^*(E^c)} \Psi'(\rho_h) e_0 (dx)-\varepsilon\int_{E^c\cap T_h(E)}\Psi'(\rho_h)e_1(dx)\nonumber\\
&\quad+{\frac{\varepsilon^2}{2}}\int_M(e_0-e_1)^2 \Psi''(\bar g_\varepsilon )m(dx).\end{align}
The difference in the integrals involving $\Psi'(\rho_0)$ is non positive, since $\Psi'$ is increasing and $\rho_0$ is smaller on $E$ than on $E^c$. By considering the signs of these terms, we deduce that 
$$\int_{E\cap T_h^*(E^c)} d(x,T_h(x))^2 f (x)m(dx)=0,$$
so $d(x,T_h(x))=0$ for all $x\in E\cap T_h^*(E^c)$; hence $T_h$ does not move any mass from $E\cap T_h^*(E^c)$ in the optimal transport, so 
$$\int_{E\cap T_h^*(E^c)}\rho_h(x)m(dx)\geq \int_{E\cap T_h^*(E^c)}f (x)m(dx)\geq \delta_1 m(E\cap T_h^*(E^c));$$
whereas $\rho_h(x)\leq \delta <\delta_1$ on $E\cap T_h^*(E^c)$. Hence $m (E\cap T_h^*(E^c))=0$, so $m(E)=0.$\par
\indent Likewise, by replacing $E$ by $\{ x\in M:\rho_h\leq 1/\delta\}$, one can show that  $m(\{ x: \rho_h(x)>1/\delta \})=0.$ Hence $\rho_h$ is bounded above and below, hence satisfies a logarithmic Sobolev inequality and spectral gap inequality for some $\lambda_1>0$, as in Proposition \ref{proposition8}(ii).\par
\indent (ii)  Let $T_h$ be the optimal transport map taking $f$ to $\rho_h$; we can take $T_h(x)=\exp_x(h^2\nabla \phi (x))$ for some $d^2/2$ concave function $-\phi :M\rightarrow {\mathbb{R}}$; then 
by McCann's Corollary 10 \cite{bib25}, there exists a tangent vector field $\zeta$ on $M$ such that $T^*(x)=\exp_x(h^2\zeta (x))$  induces $f$ from $\rho_h$, and $T_h^*$ is the inverse of $T_h$ in the sense that $T_h^*(T_h(x))=x$ almost everywhere on the support of $f$ and $T_h(T_h^*(x))=x$ almost everywhere on the support of $\rho_h$. The $d^2/2$ concave function $-\phi_h$ satisfies $(-\phi_h)^{cc}=-\phi_h ,$ so $-h^2\phi_h (x)=\inf_y\{ d(x,y)^2/2-(-h^2\phi_h)^c(y)\}$ and $\zeta (x)=-\nabla (-\phi)^c (x)$ at the points where $\phi^c$ is differentiable.

 Let $\rho_h$ be a minimizer of ${\mathcal E}(\rho; f)$, and let $T^*:M\rightarrow M$ be an optimal transport map such that $T_h^*\sharp \rho_h=f$; for $V:M\rightarrow TM$ a smooth vector  field, there exists $t_0>0$ such $\Phi_t(x)=\exp_x(tV(x))$ defines a diffeomorphism $\Phi_t:M\rightarrow M$ for $-t_0<t<t_0$. Then $\Phi_t$ gives an inner variation such that 
${\mathcal E}(\rho_h;f)\leq {\mathcal E}(\Phi_t\sharp \rho_h; f)$, or more explicitly
$$ W_2^2(\rho_h,f)+h^2{\mathcal U}(\rho_h)\leq W_2^2(\Phi_t\sharp\rho_h,f)+h^2{\mathcal U}(\Phi_t\sharp\rho_h)\qquad (-t_0<t<t_0);$$
since $T_h^*$ is the optimal transport map; hence 
\begin{align}{\frac{1}{2}}\int_M& d^2(T_h^*(x),x)\rho_1(x)m(dx)+h^2\int_M \Theta (\rho_h(x))\rho_h(x)m(dx)\nonumber\\
&\leq {\frac{1}{2}}\int_M d^2(T_h^*\circ \Phi_t^{-1}(x),x)\Phi_t\sharp \rho_h(x)m(dx)+h^2\int_M \Theta (\Phi_t\sharp\rho_h(x))\Phi_t\sharp\rho_h(x)m(dx)\nonumber\\
&={\frac{1}{2}}\int_M d^2(T_h^*(x),\Phi_t(x))\rho_h(x)m(dx)+h^2\int_M \Theta \Bigl({\frac{\rho_h(x)}{\Delta_t(x)}}\Bigl))\rho_h(x)m(dx)\end{align}
where $\Delta_t(x)=\det D\Phi_t(x)$, and there is equality at $t=0$. We deduce that the derivative of the right-hand side at $t=0$ vanishes, where  by Cabre's calculation  \cite{bib6} page 632
$(-1/2)\nabla_y d^2(z,y)=\exp_y^{-1}z$ and $(d/dt)_{t=0}\Phi_t(x)=V(x)$, so with $h^2\zeta (x)=\exp_x^{-1}T_h^*(x)$, the first variation of ${\mathcal{E}}(\rho_h;f)$ along $V$ is given by (\ref{varinnerproduct}) in Proposition \ref{proposition3} and (\ref{dotdet}) to be
\begin{align}\langle& \delta {\mathcal E}(\rho_h;f), V\rangle\nonumber\\
 &=\Bigl({\frac{d}{dt}}\Bigr)_{t=0}{\mathcal E}(\Phi_t\sharp \rho_h;f)\nonumber\\
& =h^2\int_M\bigl\langle -\zeta(x), V(x)\rangle \rho_h(x)m(dx)-h^2\int_M\Theta'(\rho_h(x))\rho_h(x)^2(d/dt)_{t=0}\Delta_t(x)m(dx)\nonumber\\
& =h^2\int_M\bigl\langle -\zeta(x), V(x)\bigr\rangle \rho_h(x)m(dx)-h^2\int_M\Theta'(\rho_h(x))\rho_h(x)^2\nabla\cdot V(x)m(dx)\nonumber\\
& =h^2\int_M\bigl\langle -\zeta(x)+(\rho_h(x)\Theta''(\rho_h(x))+2\Theta'(\rho_h(x))\nabla\rho_h(x), V(x)\bigr\rangle \rho_h(x)m(dx)\nonumber\\
& =h^2\int_M\bigl\langle -\zeta(x)+\nabla (\Theta_1\circ \rho_h)(x), V(x)\bigr\rangle \rho_h(x)m(dx)\end{align}
where we have used the divergence theorem and the identity $\Theta_1(\rho_h)=\rho_h\Theta'(\rho_h)+\Theta (\rho_h).$ Since $V$ was arbitrary,  and $\rho_h(x)>\delta >0$, we deduce that
\begin{align}\label{T^*}T_h^*(x)=\exp_x\bigl(h^2\nabla (\Theta_1\circ\rho_h)(x))\bigr),\end{align}
 so
$$-\nabla (-\phi_h)^c (x)=\zeta (x)=\nabla \bigl(\Theta_1\circ \rho_h(x))\bigr)$$
 and we deduce that $-(-h^2\phi_h )^c(x)=h^2\Theta_1(\rho_h(x))$.  For almost all $x$, we have
$$h^2\phi_h (y)+\inf_x\bigl\{ (1/2)d(x,y)^2+h^2\Theta_1(\rho_h(x))\bigr\}=0$$
where the infimum is attained at $y=\exp_x(h^2\nabla (\Theta_1\circ \rho_h)(x)))$.\par
By a version of Alexandrov's second differentiability theorem Theorem 14.1 of \cite{bib29}, we have a lower bound
\begin{align}\label{lowerboundonHessian}D_x^2\Bigl[ {\frac{1}{2}} d^2(x,y) +h^2\Theta_1\circ \rho_h (x)\Bigr]\geq 0.\end{align}
\noindent in the sense of distributions, where a nonnegative distribution is equivalent to a nonnegative Radon measure. \par
\indent (iii) The functions $T^*_{t}(x)=\exp_x(th\zeta (x))$ give the optimal transport maps from $\rho_h$ to the density $T_t^*\sharp \rho_h$, where $t\mapsto T_t^*(x)$ is a geodesic, so the function $t\mapsto {\mathcal U}(T_t^*\sharp \rho_h)$ is convex by Lemma \ref{lemma1}. Then
$(h-t){\mathcal U}(\rho_h)+t{\mathcal U}(f)\geq h{\mathcal U}(T_t^*\sharp \rho_h)$ gives
$${\mathcal U}(f)-{\mathcal U}(\rho_h)\geq h\Bigl({\frac{d}{dt}}\Bigr)_{t=0} {\mathcal U}(T_t^*\sharp \rho_h)$$
so by integrating by parts, we get
\begin{align}{\mathcal U}(f)-{\mathcal U}(\rho_h)&\geq h^2\int_M \nabla_M \bigl( \Theta'(\rho_h)\rho_h^2\bigr) \cdot \zeta (x) m(dx)\nonumber\\
&=h^2\int_M  \nabla_M \bigl( \Theta'(\rho_h)\rho_h^2\bigr) \cdot  \nabla_M(\Theta'(\rho_h )\rho_h+\Theta (\rho_h))m(dx)\nonumber\\
&=h^2\int_M \rho_h(x) \bigl\Vert  \nabla (\Theta_h\circ \rho_h)(x)\bigr\Vert^2 m(dx)=h^2{\mathcal{I}}_\Phi (\rho_h\mid m),\end{align}
\noindent where the right-hand side involves a generalized Fisher information for the pressure $p(\rho_h )=\Theta'(\rho_h )\rho_h^2$, which is positive, as in (\ref{PhiFisher}) and (\ref{specialFisher}).
\end{proof} 
\begin{cor}\label{Corollary1} For $M={\mathbb{S}}^2$, let $f$ and $\rho_h$ be as in Theorem \ref{Theorem1}.\par
\indent (i)  Then there exists a $d^2/2$-concave map $-\phi : {\mathbb{S}}^2\rightarrow {\mathbb{R}}$ such that $\nabla \phi$ is continuous and $T_h(x)=\exp_ x(h^2\nabla \phi (x))$ is continuous and induces $\rho_h$ from $f$; likewise $T_h^*:{\mathbb{S}}^2\rightarrow {\mathbb{S}}^2$ is continuous.\par
\indent (ii) If moreover $f,\rho_h\in C^1({\mathbb{S}}^2)$ have $\nabla f$ and $\nabla\rho_h$  Lipschitz continuous, then $\phi\in C^3({\mathbb{S}}^2)$.\par
\indent (iii) There exists a family of smooth probability density function $\psi_\varepsilon\ast \rho_h$  $(\varepsilon >0)$ such that 
${\mathcal E}(\psi_\varepsilon\ast \rho_h; f)\rightarrow {\mathcal E}(\rho_h; f)$ as $\varepsilon\rightarrow 0+$. \par
\end{cor}
\begin{proof}  (i) By H\"older's inequality, we have
\begin{align}\int_{B(x_0, r)} f(x)m(dx)&\leq \Bigl(\int_{B(x_0,r)} f(x)^\beta m(dx)\Bigr)^{1/\beta} m(B(x_0,r))^{1-1/\beta}\nonumber\\&\leq \Vert f\Vert_{L^\beta}(4\pi \sin^2(r/2))^{1-1/\beta},\end{align}
so taking $\beta>2$, we satisfy the hypothesis of Theorem 2.4(1) \cite{bib23}. The optimal transport map $T_h(x)=\exp_x(h^2\nabla\phi (x))$  is continuous, indeed H\"older continuous.\par
\indent (ii) The final statement follows from Theorem 2.4(3)\cite{bib23}.\par
\indent (iii)  Let $\psi_\varepsilon\ast\rho_h$ be a smooth approximation to $\rho_h$, so that $\psi_\varepsilon\ast\rho_h$ is also uniformly bounded and positive, and $\varepsilon\ast\rho_h\rightarrow \rho_h$ as $\varepsilon\rightarrow 0+$ almost surely and in $L^1(M)$ as $\varepsilon\rightarrow 0+$. By Lemma 6.1, if $\rho_h^{\gamma-1/2} \in H^1(M)$, then $\rho_h\in L^\nu (m)$ for all $\nu<n\gamma /(n-1)$.  
 Suppose  $\varphi_t(x)=\exp_x(tv(x))$, where $v(x)=\nabla_M(\Theta'(\rho_h )\rho_h+\Theta (\rho_h))$. Then $\Delta (0)=1$ and $\dot\Delta(0)=\nabla_M\cdot v,$ so ${\mathcal U}(t)={\mathcal U}(\varphi_t\sharp \rho_h)$ is convex by Lemma \ref{lemma1}. \par
\indent For $M={\mathbb{S}}^2$, we can approximate $\rho_h$ by a family of smooth probability density function $\psi_\varepsilon\ast \rho_h$  $(\varepsilon >0)$ such that 
${\mathcal E}(\psi_\varepsilon\ast \rho_h; f)\rightarrow {\mathcal E}(\rho_h; f)$ as $\varepsilon\rightarrow 0+$. Note that $SO(3)$ acts transitively on ${\mathbb{S}}^2$ via rotations, and $SO(3)$ has a bi invariant Haar probability measure $\mu_{SO(3)}$. So given a smooth approximate identity $(\psi_\varepsilon )_{\varepsilon >0}$ of probability densities in $L^1(\mu_{SO(3)})$, we can introduce $\psi_\varepsilon\ast \rho_h$, which is a probability density function which by Jensen's inequality applied to $\Psi (x)=x\Theta (x)$ satisfies  
\begin{align}\int_{SO(3)}\Psi \Bigl(\int \psi_\varepsilon (gh^{-1})&\rho_0(h)\mu_{SO(3)}(dh)\Bigr)\mu_{SO(3)}(dh)\nonumber\\
&\leq\int_{SO(3)}\psi_\varepsilon (g)\Psi (\rho_0(g^{-1}h))\mu_{SO(3)}(dh)\Bigr)\mu_{SO(3)}(dg)\nonumber\\
&=\int_{SO(3)}\Psi (\rho_0(g))\mu_{SO(3)}(dg),\end{align}
and 
$$\bigl\vert W_2^2 (\psi_\varepsilon\ast \rho_h, f)-W_2^2(\rho_h, f )\bigr\vert\leq W_2(\psi_\varepsilon\ast\rho_h, \rho_h)\bigl(2W_2(\rho_h, f )+  
W_2(\psi_\varepsilon\ast\rho_h, \rho_h)\bigr)$$
\noindent where $W_2(\psi_\varepsilon\ast\rho_h, \rho_h)\rightarrow 0$ as $\varepsilon\rightarrow 0+$, so
$$\lim\inf_{\varepsilon \rightarrow 0+} \Bigl(W_2^2(\psi_\varepsilon\ast\rho_h, f )+\int_M \Psi (\psi_\varepsilon\ast\rho_h (x)) m(dx)\Bigr) =
W_2^2(\rho_h, f )+\int_M \Psi (\rho_h (x)) m(dx).$$
\end{proof} 
 \section{\label{Energy} Energy estimates}
\noindent In the previous sections we started with a pair $(\rho_0, q_0)$ and carried out the first two stages of the algorithm once. In this section, we take a time step size $h$, and obtain energy estimates when we apply these steps repeatedly. In Proposition\ref{proposition10} we ensure that the hypotheses of Lemma \ref{lemma1} are satisfied, so we can proceed to obtain the energy estimates in Theorem \ref{Theorem2}. \par
\indent {\bf Stage 3.} As in \ref{Theorem1}, we have a predictor $f$ and corrector $\rho_h$ such that 
 $T_0^*\sharp \rho_0=f$ and $T_h^*\sharp \rho_h=f$ where
$$T_0^*(x)=\exp_x (h\nabla q_0(x)), \quad T_h^*(x)=\exp_x (h^2\nabla (\Theta_1\circ \rho_h)(x)).$$
so the maps based at $f$ are $T_0\sharp f=\rho_0$ and $T_h\sharp f=\rho_h$. We have
$$T_0(x)=\exp_x(h\nabla \phi_0(x)), \quad T_h(x)=\exp_x(h^2\nabla\phi_h(x)),$$
where $\phi_h$ was found in (\ref{findingphi}). The maps $T_0$ and $T_h$ are bijective, so $x\mapsto T_h\circ T_0^*(x)$ is also a bijection.
As in Proposition \ref{proposition4}, we consider
$$F_y(s,t)=\exp_y\bigl( t(1-s)\nabla\phi_0(y)+hst\nabla\phi_h(y)\bigr)$$
so that there is a geodesic $t\mapsto F_y(s,t)$ emanating from $F_y(s,0)=y$, and $F_y(s,h)$ joins $F_y(0,h)=T_0(y)$ to $F_y(1,h)=T_h(y)$.\par
\indent Also suppose that $(-hq_0)^{cc}=-hq_0$, and let $h\phi_0 =-(-hq_0)^c$, so $-\phi_0$ is $d^2/2$ concave, and 
$$-h\phi_0 (y)=\inf_x\bigl\{ d(x,y)^2/2+hq_0(x): x\in M\bigr\}$$
where at $y=\exp_x(h\nabla q_0(x))$, we have as a  consequence of Gauss's Lemma $\nabla (1/2) d(x,y)^2=-\nabla hq_0(x),$
by  \cite{bib6} p 632.
\vskip.05in
\indent (i) Then $F_x(s,t)=\exp_x((1-s)t\nabla\phi_0(x)\xi+hst\nabla\phi_h(x))$ gives a Jacobi field such that $t\mapsto F_x(s,t)$ is a geodesic emanating from $x$ at $t=0$;\par
\indent (ii) $T_0(x)=F_x(0,h)$ induces $\rho_0$ from $f_h$, and $T_h(x)=F_x(1,h)$ induces $\rho_h$ from $f_h$;\par
\indent (iii) $Y(t)={\frac{\partial}{\partial s}}F_x(s,t)\vert_{s=0}$ satisfies Jacobi's equation page 366  \cite{bib29} and has initial condition $Y(0)=0$ and $Y'(0)=h\nabla\phi_h(x)-\nabla\phi_0(x)$.\par
The curve $s\mapsto F_x(s,h)$ is not necessarily a geodesic; better to regard it as a geodesic variation. There exist matrix functions $A_x(s,h):T_xM\rightarrow T_{F_x(s,h)}M$, which are given by Jacobi's equation page 366 of \cite{bib29} 
$(d^2/ds^2)A+RA=0$, where $R$ is a $n\times n$ matrix given in terms of Riemann's curvature tensor.\par
\indent According to Lemma 3.2 of Cabre  \cite{bib6}, the Jacobian of $x\mapsto F_x(s,t)$ is given by
$$\det D_xF(s,t)=\det \bigl[D_x\exp_x(v)\bigr]\det \Bigl[{\frac{1}{2}} D_x^2d^2(x,y)+(1-s)tD_x^2\phi_0(x)+ hstD_x^2\phi_h(x)\Bigr],$$
\noindent where  $y=F_x(s,t)$, and the middle term is the Jacobian of $\exp_x: T_xM\rightarrow M$ evaluated at 
$v=(1-s)t\nabla\phi_0(x)+hst\nabla\phi_h(x)$. At $t=0,$ this reduces to $1$, by basic facts about the exponential map in normal coordinates; see \cite{bib25} for details. For $y=\exp_xv$, we have a differential $D_v\exp_x: T_xM\rightarrow T_yM$, which is an invertible linear map; also, one can compute $D^2_{x,y} d^2(x,y)$, and one finds $-D^2_{x,y}d^2(x,y)/2=(D_v\exp_x )^{-1}$. 

\indent For $M={\mathbb{S}}^2$, we can compute the Jacobian of $F_x(s,h)$ explicitly, and obtain conditions for log concavity with respect to $s$. Let $\xi=\xi(x) =h\nabla \phi_0 (x)$ and $\zeta =\zeta (x)=h^2\nabla\phi_h(x)-h\nabla\phi_0(x).$ Observe that
$$F_x(s,h)=\cos ( \Vert \xi +s\zeta\Vert) x+ \sin ( \Vert \xi +s\zeta\Vert) {\frac{\xi +s\zeta}{\Vert \xi +s\zeta\Vert}}.$$
 
\begin{prop}\label{proposition10}Let $\phi_0, \phi_h\in C^2({\mathbb{S}}^2; {\mathbb{R}})$ and suppose that $h>0$ is so small that $h\Vert\nabla \phi_0(x)\Vert+\gamma h^2\Vert\nabla\phi_h(x)\Vert<\pi/2$ for all $x\in {\mathbb{S}}^2$. Then
\begin{align}\label{concaveD}s\mapsto \log\det \Bigl[ {\frac{1}{2}}D_x^2d^2(x,y)\Bigr\vert_{y=F_x(s,h)}+(1-s)tD_x^2\phi_0(x)+ hstD_x^2\phi_h(x))\Bigr]\end{align}
\noindent is a concave function of  $s\in [0,1]$.\end{prop}
\vskip.05in
\begin{proof}  For $x,v\in {\mathbb{S}}^2$ such that $x\cdot v=0$, we have $\exp_x(tv)=(\cos t)x+(\sin t)v.$ We have $d(x,y)=\arccos (x\cdot y)$, so by Taylor's theorem, we have
\begin{align}{\frac{1}{2}}&d^2(\exp_x(tv),y)\nonumber\\
&={\frac{1}{2}}d^2(x,y)-{\frac{\arccos (x\cdot y)}{\sqrt{1-(x\cdot y)^2}}}\bigl( y\cdot ((\cos t)x + (\sin t)v-x)\bigr)\nonumber\\
&\quad+{\frac{1}{2}}\Bigl( {\frac{1}{1-(x\cdot y)^2}}-{\frac{(x\cdot y)\arccos (x\cdot y)}{(1-(x\cdot y)^2)^{3/2}}}\Bigr)\bigl( y\cdot (x\cos t+v\sin t -x)\bigr)^2+O(t^3)\end{align}
as $t\rightarrow 0$, so picking off the coefficients of $t^2$, we find the Hessian of $d^2(x,y)/2$ to be
\begin{align}{\frac{1}{2}}\bigl\langle D_x^2 d^2(x,y)v,v\bigr\rangle= &(y\cdot v)^2\Bigl( {\frac{1}{1-(x\cdot y)^2}}-{\frac{(x\cdot y)\arccos (x\cdot y)}{(1-(x\cdot y)^2)^{3/2}}}\Bigr)\nonumber\\
&\quad+{\frac{(v\cdot v)(y\cdot x)\arccos (x\cdot y)}{\sqrt{1-(x\cdot y)^2}}}.\end{align}
\noindent In particular, for $y=\exp_x(\tau\eta )$ where $\eta \cdot \eta =1$ and $\eta \cdot x=0$, we have 
$${\frac{1}{2}}\bigl\langle D_x^2 d^2(x,y)v,v\bigr\rangle_{y=\exp_x(\tau \eta )}=\Bigl( 1-{\frac{\tau\cos \tau }{\sin \tau }}\Bigr)(\eta \cdot v)^2 +{\frac{\tau \cos \tau }{\sin \tau }}v\cdot v.$$
 Then with $\tau =\Vert \xi+s\zeta \Vert$ and $\eta =(\xi +s\zeta )/\Vert\xi +s\zeta\Vert$ we have
$$ {\frac{1}{2}}D_x^2d^2(x,y)\bigr\vert_{y=\exp_x(\tau \eta)}=\Bigl(1-{\frac{\tau \cos \tau}{\sin\tau}}\Bigr) (\eta\otimes\eta )+{\frac{\tau \cos \tau}{\sin\tau}}I_2.$$
so
$$\det\Bigl[{\frac{1}{2}} D_x^2d^2(x,y)\Bigr\vert_{y=\exp_x(\tau \eta)}\Bigr]={\frac{\tau \cos \tau}{\sin\tau}}.$$
\indent The functions $\Delta (\tau)=\tau \cot \tau$ and $\log\Delta (\tau) =\log (\tau\cot \tau )$ are decreasing and concave functions of  $\tau\in (0, \pi/2)$, as one shows by elementary calculus.  Indeed,
$${\frac{d^2}{d\tau^2}}\log(\tau\cot \tau)={\frac{\tau^2-(1+\tau^2)\sin^2\tau+\sin^2\tau (\sin^2\tau-\tau^2)}{\tau^2\sin^2\tau (1-\sin^2\tau)}}<0\qquad (0<\tau<\pi/2).$$
Let $\tau =\Vert h\nabla q_0(x)-h^2s\nabla(\Theta_1\circ\rho_h)(x))\Vert$; then for $\Delta$ twice differentiable, we have
$${\frac{d^2\Delta}{ds^2}}={\frac{d^2\Delta}{d\tau^2}}\Bigl({\frac{d\tau}{ds}}\Bigr)^2+ {\frac{d\Delta}{d\tau}}{\frac{d^2\tau}{ds^2}}$$
where
$${\frac{d\tau}{ds}}={\frac{(\xi +s\zeta)\cdot\zeta }{\Vert\xi+s\zeta\Vert}}, \quad{\frac{d^2\tau}{ds^2}}={\frac{\Vert\xi +s\zeta \Vert^2\Vert\zeta\Vert^2-(\zeta\cdot(\xi+s\zeta ))^2}{\Vert\xi+s\zeta\Vert^3}},$$
\noindent where the lastest term is nonnegative by Cauchy-Schwarz, so  $\Delta (\tau (s))=\tau (s)\cot \tau (s)$ and $\log\Delta( \tau (s))=\log (\tau (s)\cot \tau (s))$ are decreasing and concave functions of  $s\in (-1,1)$ with $\Delta (s)\rightarrow 1$ as $h\rightarrow 0$. This proves that
\begin{align}s\mapsto \log\det \Bigl[ {\frac{1}{2}}D_x^2d^2(x,y)\Bigr\vert_{y=F_x(s,h)}\Bigr], \quad s\mapsto \det \Bigl[ {\frac{1}{2}}D_x^2d^2(x,y)\Bigr\vert_{y=F_x(s,h)}\Bigr]\end{align}
\noindent are concave functions of $s\in [0,1]$.\par
\indent By \cite{bib18} page 467, the function $A\mapsto \log\det A$ is concave on the positive cone of positive definite matrices $A$, so $s\mapsto \log\det A(s)$ is concave.  
This argument is not decisive when $y=F_x(s,h)$ depends on $s$ as in (\ref{concaveD}),  so we need a further calculation.
Consider a unit vector $\nu=x\times \eta$ in $T_x{\mathbb{S}}^2$ so that $\{ x, \eta,\nu \}$ is an orthonormal basis for ${\mathbb{R}}^3$. Then $\eta$ is a unit vector in $T_x{\mathbb{S}}^2$, hence  $d\eta/ds$ is perpendicular to $\eta$ and to $x$, 
$${\frac{d}{ds}}\begin{bmatrix}x\cr \eta\cr \nu\end{bmatrix} =\begin{bmatrix}0&0&0\cr 0&0&-\alpha\cr 0&\alpha &0\end{bmatrix} \begin{bmatrix}x\cr \eta\cr \nu\end{bmatrix}$$ 
where $\vert\alpha \vert\Vert\eta\Vert =\Vert d\eta/ds\Vert$.  
Then we consider
$$D={\frac{1}{2}}D_x^2d^2(x,y)\bigr\vert_{y=\exp_x(\tau \eta)}=(\tau\cot \tau )\, \nu\otimes\nu +\eta\otimes \eta,$$
$$Q_0=hD_x^2\phi_0(s),\quad  P= h^2D_x^2\phi_h(x)-hD_x^2\phi_0(x).$$
Then $D+Q_0+sP$ is  positive definite, provided that 
$${\frac{\tau\cos\tau}{\sin\tau}}> h\Vert D_x^2\phi_0\Vert +h^2\Vert D_x^2\phi_h(x)\Vert\qquad (x\in M),$$
\noindent which holds provided $h>0$ is sufficiently small. 
We have 
$${\frac{d}{ds}}\log \det \bigl[D+Q_0+sP\bigr]={\hbox{trace}}\Bigl[ (D+Q_0+sP)^{-1}\bigl({\frac{dD}{ds}}+P\bigr)\Bigr],$$
\begin{align}\label{secondderivative}{\frac{d^2}{ds^2}}&\log \det \bigl[D+Q_0+sP\bigr]\nonumber\\
&={\hbox{trace}}\Bigl[ (D+Q_0+sP)^{-1}{\frac{d^2D}{ds^2}}\Bigr]\nonumber\\
&\quad -{\hbox{trace}}\, \Bigl[ (D+Q_0+sP)^{-1}\bigl({\frac{dD}{ds}}+P\bigr) (D+Q_0+sP)^{-1}\bigl({\frac{dD}{ds}}+P\bigr)\Bigr]\end{align}
\noindent where
$$(D+Q_0+sP)^{-1/2}\bigl({\frac{dD}{ds}}+P\bigr) (D+Q_0+sP)^{-1/2}$$ 
is real symmetric, so the final term counts negative. To show that (\ref{secondderivative}) is negative, it therefore suffices to show that $d^2D(s)/ds^2\leq 0$, or equivalently that $\langle D(s)v,v\rangle$ is concave for all $v$. We compute
$${\frac{dD}{ds}}={\frac{d\Delta}{ds}}\, \nu\otimes\nu +(\Delta -1)\alpha (\eta\otimes\nu +\nu\otimes \eta ),$$
\begin{align}\label{d2D}{\frac{d^2D}{ds^2}}&={\frac{d^2\Delta}{ds^2}}\nu\otimes\nu +\Bigl( 2\alpha {\frac{d\Delta}{ds}}+(\Delta -1){\frac{d\alpha}{ds}}\Bigr) \bigl( \eta\otimes\nu +\nu\otimes\eta \bigr)\nonumber\\
&\quad+2\alpha^2(\Delta-1)\bigl(\eta\otimes\eta-\nu\otimes\nu \bigr),\end{align}
\noindent where $\tau=O(h)$ and $\alpha =O(h)$ as $h\rightarrow 0$, so 
$$0>\Delta-1 =\tau\cot\tau -1=-\tau^2/3+O(\tau^4)=O(h^2)\qquad (h\rightarrow 0)$$
and $d\Delta/ds=-(2/3)(\tau +O(\tau^3))d\tau/ds=O(h^2)$. We wish to have
$$\begin{bmatrix} {\frac{d^2\Delta}{ds^2}}-2\alpha^2(\Delta-1)&\alpha {\frac{d\Delta}{ds}}+{\frac{1}{2}}(\Delta -1){\frac{d\alpha}{ds}}\cr
\alpha {\frac{d\Delta}{ds}}+{\frac{1}{2}}(\Delta -1){\frac{d\alpha}{ds}}&2\alpha^2(\Delta-1)\end{bmatrix}\leq 0,$$
which will ensure that (\ref{d2D}) is negative. \par
\indent Loeper \cite{bib23} shows that the cross-sectional curvature on ${\mathbb{S}}^{n}$ for $n\geq 2$ is uniformly positive, so there exists $K_0>0$ such that 
 $${\frac{-3}{2}}\Bigl({\frac{d^2}{dt^2}}\Bigr)_{t=0}\Bigl({\frac{d^2}{ds^2}}\Bigr)_{s=0}{\frac{1}{2}}d^2(\exp_x (tv),\exp_x(\xi +s\zeta ))\geq K_0\bigl( \Vert v\Vert^2\Vert\zeta\Vert^2-\vert\langle v, \zeta\rangle\vert\Vert v\Vert\Vert\zeta\Vert\bigr)$$
\noindent for all $v,\zeta\in T_x{\mathbb{S}}^n$ this condition is known as $(As)$ or $(A3)$, and was introduced by Ma, Trudinger and Wang.  Note that
$$\Vert v\Vert^2\Vert\zeta\Vert^2-\vert\langle v, \zeta\rangle\vert\Vert v\Vert\Vert\zeta\Vert=\Bigl( {\frac{\Vert v\Vert \Vert \zeta\Vert}{\Vert v\Vert \Vert \zeta\Vert+\vert\langle v,\zeta\rangle\vert}}\Bigr)\bigl\Vert v\bigr\Vert^2\Bigl\Vert\zeta -{\frac{\langle \zeta,v\rangle v}{\Vert v\Vert^2}}\Bigr\Vert^2,$$
\noindent where the quotient in parentheses lies between $1/2$ and $1$, so one can reduce to the case of $\langle v, \zeta\rangle=0$. 
This shows that 
$$s\mapsto \bigl\langle D_x^2d^2(x, y)\vert_{y=\exp_x(\xi +s\zeta )} v,v\bigr\rangle$$
is concave, so (\ref{secondderivative}) is negative, as required.\par
\indent In our discussion, we have ignored the issue of cut-locus of $d^2(x,y)$, namely the points $y$ such that $x\mapsto d(x,y)^2/2$ is not differentiable. For ${\mathbb{S}}^n$, the cut locus consists of the antipodal point $-x$, and Loeper \cite{bib23} shows by a detailed analysis that the cut locus does not affect the validity of the results.
\end{proof} 
\begin{thm}\label{Theorem2}Suppose that $-\phi_0, -\phi_h$ are $d^2/2$ concave functions such that for
\begin{align}T_s(x)=F_x(s,h)=\exp_x\bigl( (1-s)h\nabla\phi_0(x)+sh^2\nabla\phi_h(x)\bigr)\end{align} 
\indent (i) $T_0(x)$ is the optimal transport map that induces $\rho_0$ from $f$;\par
\indent (ii) $T_1(x)$ is the optimal transport map that induces $\rho_h$ from $f$;\par
\indent (iii) $T_s(x)$ induces $\rho_{sh}$ from $f$.\par
\noindent Then $s\mapsto \rho_{sh}$ for $s\in [0,1]$ is a path in ${\mathcal W}({\mathbb{S}}^2)$ connecting $\rho_0$ to $\rho_h$ such that\par
\indent (1)  $s\mapsto {\mathcal E}(\rho_{sh};f)$ has a minimum at $s=1$, and
\begin{align}{\mathcal E}(\rho_0;f)\geq {\mathcal E}(\rho_h;f)+{\frac{2}{\pi^2}}W_2^2(\rho_0, \rho_h);\end{align}
\indent (2) the internal energy satisfies
\begin{align}{\mathcal U}(\rho_0)+\int_{{\mathbb{S}}^2}f(x)\nabla (\Theta_1\circ \rho_h )(x)\cdot \nabla (-h\phi_0(x)+
h^2\phi_h(x)) m(dx)\geq {\mathcal U}(\rho_h).\end{align}\end{thm} 
\begin{proof}  (1) The function $s\mapsto {\mathcal U}(\rho_{sh})$ is convex by Proposition \ref{proposition10} and Lemma \ref{lemma1}, and $s\mapsto W_2^2(\rho_{sh},f)$ satisfies Proposition \ref{proposition4}. Hence  
$$(1-s){\mathcal E}(\rho_0; f)+s{\mathcal E}(\rho_h;f)\geq {\mathcal E}(\rho_{sh}; f)+{\frac{2s(1-s)}{\pi^2}}W_2^2(\rho_0, \rho_h).$$
Since  ${\mathcal E}(\rho_{sh};f)\geq {\mathcal E}(\rho_h;f)$, we can divide by $1-s$, let $s\rightarrow 1-$ and deduce that 
$${\mathcal E}(\rho_0;f)\geq {\mathcal E}(\rho_h;f)+{\frac{2}{\pi^2}}W_2^2(\rho_0, \rho_h),$$
\noindent which gives
\begin{align}{\frac{1}{2}}\int_{{\mathbb{S}}^2}& \Vert\nabla q_0(x)\Vert^2\rho_0(x)m(dx)+\int_{{\mathbb{S}}^2}\rho_0(x)\Theta (\rho_0(x)) m(dx)\nonumber\\
&\geq 
{\frac{1}{2}}\int_{{\mathbb{S}}^2}\rho_h(x)  \big\Vert\nabla (\Theta_1\circ \rho_h)(x))\bigr\Vert^2m(dx)+\int_{{\mathbb{S}}^2}\rho_h(x)\Theta (\rho_h(x)) m(dx)\nonumber\\
&\quad+{\frac{2}{\pi^2}}W_2^2(\rho_0, \rho_h).\end{align}
\indent (2) An immediate consequence of Proposition \ref{proposition10} and Lemma \ref{lemma1} is the inequality
$${\mathcal U}(\rho_0)+\Bigl({\frac{d}{ds}}\Bigr)_{s=1}{\mathcal U}(\rho_{sh})\geq {\mathcal U}(\rho_h),$$
\noindent where the derivative term is
\begin{align}\Bigl({\frac{d}{ds}}\Bigr)_{s=1}\int_{{\mathbb{S}}^2} \Theta (\rho_{sh}(y))\rho_{sh}(y)m(dy)&=\int_{{\mathbb{S}}^2} \Theta_1 (\rho_h(y))\Bigl({\frac{\partial \rho_{sh}(y)}{\partial s}}\Bigr)_{s=1}m(dy)\nonumber\\
&=-\int_{{\mathbb{S}}^2} \Theta_1 (\rho_h(y))\nabla\cdot (\vec v(y;1))\rho_h(y))m(dy)\nonumber\\
&=\int_{{\mathbb{S}}^2}\rho_h(y)\vec v(y;1)\cdot \nabla  (\Theta_1\circ \rho_h)(y)\, m(dy),\end{align}
where the velocity field $\vec v(y;1)$ is given as follows. Let $\lambda (x\times (\xi +\zeta ))=\zeta-(\xi+\zeta )\cdot \zeta (\xi +\zeta)/\Vert \xi+\zeta \Vert^2$ be the component of velocity in $T_x{\mathbb{S}}^2$ that is perpendicular to $\xi+\zeta $; then
$$\begin{bmatrix} y\cr v\cr y\times v\end{bmatrix}=\begin{bmatrix} \cos (\Vert \xi +\zeta\Vert) & {\frac{\sin (\Vert \xi+\zeta \Vert)}{\Vert \xi+\zeta \Vert}} &0\cr 
{\frac{-\sin (\Vert \xi+\zeta \Vert )(\xi +\zeta )\cdot \zeta }{\Vert \xi +\zeta\Vert}}&{\frac{\cos (\Vert \xi +\zeta \Vert )(\xi +\zeta )\cdot \zeta}{\Vert \xi +\zeta\Vert^2}}&\lambda {\frac{\sin (\Vert \xi+\zeta\Vert)}{\Vert \xi +\zeta\Vert}}\cr
\lambda\sin^2 (\Vert \xi+\zeta \Vert) & -{\frac{\lambda \sin (\Vert \xi +\zeta \Vert) \cos (\Vert \xi +\zeta\Vert)}{\Vert \xi +\zeta\Vert}}& {\frac{(\xi +\zeta )\cdot \zeta}{\Vert \xi +\zeta\Vert^2}}\end{bmatrix}\begin{bmatrix} x\cr \xi +\zeta\cr x\times (\xi +\zeta)\end{bmatrix}$$
\noindent where $\zeta =h^2\nabla\phi_h(x)-h\nabla\phi_0(x)$ and $\xi +\zeta =h^2\nabla\phi_h(x)$. Hence the velocity is
\begin{align}v&={\frac{-\sin (\Vert h^2\nabla\phi_h(x) \Vert )(\nabla\phi_h(x)  )\cdot (h^2\nabla\phi_h(x)-h\nabla\phi_0(x)) }{\Vert \nabla\phi_h(x) \Vert}}x\nonumber\\
&\quad +{\frac{\cos (\Vert h^2\nabla\phi_h(x) \Vert )(\nabla\phi_h(x) )\cdot (h^2\nabla\phi_h(x)-h\nabla\phi_0(x))}{\Vert \nabla\phi_h(x) \Vert^2}}\nabla\phi_h(x)\nonumber\\
&\quad +{\frac{\sin (\Vert h^2\nabla\phi_h(x) \Vert)}{\Vert h^2\nabla\phi_h(x) \Vert}}\Bigl(h^2\nabla\phi_h(x)-h\nabla\phi_0(x)\nonumber\\
&\qquad-{\frac{(h^2\nabla\phi_h(x)-h\nabla\phi_0(x))\cdot (\nabla\phi_h(x))(\nabla\phi_h(x))}{\Vert \nabla\phi_h(x) \Vert^2}}\Bigr).\end{align}

\indent Recall that $T_h^*(x)=\exp_x (h^2\nabla \Theta_1(\rho_h(x)))$ and $T_h(x)=\exp_x(h^2\nabla\phi_h(x))$, so with $\sigma =h^2\Vert \nabla\phi_h(x)\Vert$, we have
$$\begin{bmatrix}y\cr {\frac{\nabla(\Theta_1\circ \rho_h)(y)}{\Vert \nabla (\Theta_1\circ \rho_h )(y)\Vert}}\end{bmatrix}=\begin{bmatrix}\cos\sigma &\sin \sigma\cr -\sin\sigma &\cos\sigma\end{bmatrix} \begin{bmatrix}x\cr {\frac{\nabla\phi_h(x)}{\Vert\nabla\phi_h(x)\Vert}}\end{bmatrix}.$$

\noindent When taking the scalar product with $\nabla \Theta_1(\rho_h(x))$, which is parallel to $\nabla\phi_h(x)$, we get
$$\Bigl({\frac{d}{ds}}\Bigr)_{s=1}\int_{{\mathbb{S}}^2} \Theta (\rho_{sh}(y))\rho_{sh}(y)m(dy)=\int_{{\mathbb{S}}^2} f(x)(h\nabla\phi_h(x)-\nabla\phi_0(x))\cdot \nabla(\Theta_1\circ\rho_h)(x))m(dx)$$
hence
$${\mathcal U}(\rho_0)+\int_{{\mathbb{S}}^2} f(x)\nabla( \Theta_1\circ \rho_h )(x)\cdot \nabla (-h\phi_0(x)+
h^2\phi_h(x)) m(dx)\geq {\mathcal U}(\rho_h).$$ 
\vskip.05in
\noindent At $y=F_x(1,h)=\exp_x(h^2\nabla\phi_h(x))=T_h(x)$, where $x=T_h^*(y)=\exp_y(h^2\nabla \Theta_1\circ\rho_h(y))$ the corresponding velocity is 
$$v(y)=-\nabla\phi_0(x)+h\nabla\phi_h(x)=-(\nabla \phi_0)(T_h^*(y))+h(\nabla\phi_h)(T_h^*(y))$$
so we use the Helmholz decomposition in $L^2(\rho_h)$ to write $v(y)=\nabla q_h(y)+w$, where $\nabla\cdot (\rho_hw)=0$, and $q_h$ is defined to 
be the new velocity potential. Thus we update $(\rho_0, q_0)$ to $(\rho_h, q_h)$. 
\end{proof}   
\vskip.05in
\begin{cor}\label{Corollary2} (i) The velocity potential $q_h$ is exponentially integrable.\par
(ii) The amount of energy that is dissipated during one step of the algorithm is
\begin{align}{\mathcal H}(\rho_0, q_0)-{\mathcal H}(\rho_h,q_h)\geq {\frac{h^2}{2}}\int_{{\mathbb{S}}^2}\Vert\nabla (\Theta_1\circ \rho_h )(x,t)\Vert^2\rho_h(x)m(dx).\end{align}
\end{cor}
\vskip.05in
\begin{proof}  (i) By Onofri's inequality \cite{bib31}, $q_h$ is exponentially integrable, since 
$$\log \int_{{\mathbb{S}}^2} e^{q(x)} {\frac{m(dx)}{4\pi}}\leq \int_{{\mathbb{S}}^2} q(x){\frac{m(dx)}{4\pi}}+{\frac{1}{4}}\int_{{\mathbb{S}}^2} \Vert \nabla q(x)\Vert^2 {\frac{m(dx)}{4\pi}}.$$
(ii) Combining Theorem \ref{Theorem2} with Proposition \ref{proposition9}, we have with $y=T_h(x)$, so 
\begin{align}{\mathcal H}(\rho_h,q_h)&={\frac{1}{2}}\int_{{\mathbb{S}}^2}\Vert\nabla q_h (y)\Vert^2 \rho_h(y)m(dy)+{\mathcal U}(\rho_h)\nonumber\\
&\leq {\frac{1}{2}}\int_{{\mathbb{S}}^2}\Vert v(y)\Vert^2 \rho_h(y)m(dy)+{\mathcal U}(\rho_h)\nonumber\\
&={\frac{1}{2}}\int_{{\mathbb{S}}^2}\Vert \nabla \phi_0(x)-h\nabla\phi_h(x)\Vert^2f(x)m(dx) +\int_{{\mathbb{S}}^2}\Theta (\rho_h(x))\rho_h(x)m(dx) \nonumber\\
&\leq {\frac{1}{2}}\int_{{\mathbb{S}}^2}\Vert \nabla \phi_0(x)\Vert^2f(x)m(dx)+\int_{{\mathbb{S}}^2}\Theta (\rho_0(x))\rho_0(x)m(dx)\nonumber\\
&\quad +{\frac{h^2}{2}}\int_{{\mathbb{S}}^2}\Vert\nabla\phi_h(x)\Vert^2f(x)m(dx) -h\int_{{\mathbb{S}}^2}\nabla \phi_0(x)\cdot \nabla\phi_h(x)f(x)m(dx)\nonumber\\
&\quad +\int_{{\mathbb{S}}^2}f(x)\nabla \Theta_1(\rho_h(x))\cdot (-h\nabla\phi_0(x)+h^2\nabla\phi_h(x))m(dx)\nonumber\\
&\leq {\frac{1}{2}}\int_{{\mathbb{S}}^2}\Vert \nabla q_0(x)\Vert^2\rho_0(x)m(dx) +\int_{{\mathbb{S}}^2}\Theta (\rho_0(x))\rho_0(x)m(dx)\nonumber\\
&\quad -{\frac{h^2}{2}}\int_{{\mathbb{S}}^2}\rho_h(x)\Vert\nabla (\Theta_1\circ \rho_h)(x))\Vert^2m(dx),\end{align}   
hence the result. Thus we have dissipation of energy, as required by (\ref{dissipation}).
\end{proof} 
\section{\label{weakeuler}Weak solutions of the Euler equations}
\noindent In this section we use the discrete time algorithm with $h\rightarrow 0$ to obtain weak solutions to the Euler equations. Both the continuity equation and the acceleration equation involve $\rho$, so we use the approximation procedure of the preceding section to create a discrete-time approximation, which we then convert into a $2$-absolutely continuous path $[0, \tau ]\rightarrow {\mathcal W}^2({\mathbb{S}}^2):$ $t\mapsto p_t$; then we solve the associated ODE to generate a flow in $T{\mathbb{S}}^2$, and this flow induces a $2$-absolutely continuous path $[0, \tau ]\rightarrow {\mathcal W}^1({\mathbb{S}}^2):$ $t\mapsto \rho (\cdot ,t )$. Due to lack of Lipschitz continuity, there is an extra approximation step, where we smooth the densities in the space variable. We are able to establish existence of a weak solution of the Euler continuity equation in this way.

\subsection{Weak solution of the continuity equation}

\begin{prop}\label{proposition11} Let $h,\varepsilon, \delta>0$ and suppose that $q_0\in H^1({\mathbb{S}}^2)$ and that $\rho_0\in L^1({\mathbb{S}}^2)$ satisfies $\delta \leq \rho_0(x)\leq 1/\delta $ and 
$\int_{{\mathbb{S}}^2}\rho_0(x)\Vert\nabla (\Theta_1\circ \rho_0)(x)\Vert^2m(dx)<\infty$.\par
\indent (i) Then there exists a $2$-absolutely continuous path $[0, \tau ]\rightarrow {\mathcal W}({\mathbb{S}}^2):$ $t\mapsto \rho_t^{(h, \varepsilon )}$  such that (\ref{ODE}) has a solution, and $x\mapsto X^{(h, \varepsilon )}(x,t)$ is bijective ${\mathbb{S}}^2\rightarrow {\mathbb{S}}^2$ with inverse $x\mapsto X^{(h, \varepsilon ), *}(x,t)$ for $t\in [0, \tau]$. \par
\indent (ii) There exists a sequence $(h_\nu, \varepsilon_\nu)\rightarrow (0,0)$, such that $X^{(h_\nu, \varepsilon_\nu )}(x,t)\rightarrow X(x,t)$ uniformly on ${\mathbb{S}}^2\times [0, \tau ]$. Let $\rho (x,t)$ be the probability density function that is induced by $x\mapsto X(x,t)$ from $\rho_0$. Then $\rho(x,t)$  gives a weak solution of the Eulerian continuity equation (\ref{continuity}) with velocity $\vec v(y,t)={\frac{\partial X}{\partial t}}\circ X^*(y,t)$.\end{prop}   
\begin{proof}  The Bochner--Lebesgue space $L^1([0, \tau ]; C({\mathbb{S}}^2))=L^1[0, \tau ]\hat\otimes C({\mathbb{S}}^2)$ has dual\par
\noindent  $L^\infty ([0, \tau ]; C({\mathbb{S}})')=L^\infty [0, \tau ]\check \otimes C({\mathbb{S}}^2)'$ where 
$C({\mathbb{S}}^2)'$ is the space of bounded Radon measures on ${\mathbb{S}}^2$. We consider $\rho (x,t)$ as a linear functional on  $L^1([0, \tau ]; C({\mathbb{S}}^2))$ with the pairing by integration. More specifically, for $\tau, \delta, K>0$, we introduce
\begin{align}{\mathcal M}(\delta, K,\tau )&=\Biggl\{\rho\in L^\infty ({\mathbb{S}}^2\times [0,\tau ];{\mathbb{R}}): \int_M\rho (x,t)m(dx)=1;\quad\delta \leq \rho (x,t)\leq 1/\delta;\nonumber\\
&\qquad  \int_0^\tau \int_{{\mathbb{S}}^2} \rho (x,t)\Vert\nabla (\Theta_1\circ\rho )(x,t)\Vert^2m(dx)dt\leq K\Biggr\}.\end{align}
 Let $U(r)=r\Theta (r).$ Then by Proposition \ref{proposition8}, ${\mathcal M}(\delta ,\tau,  K)$ is a bounded and hence relatively weakly compact subset of the reflexive Orlicz space $L_U ({\mathbb{S}}^2\times [0,\tau ];{\mathbb{R}})$, hence a weakly compact subset of
$L^1({\mathbb{S}}^2\times [0,\tau ];{\mathbb{R}})$. For each $t$, the set ${\mathcal M}_t(\delta ,\tau,  K)=\{ \rho (\cdot , t): \rho\in  {\mathcal M}(\delta, K,\tau )\}$ is a bounded, uniformly integrable and hence weakly compact subset of $L^1({\mathbb{S}}^2;m)$, hence a relatively weakly compact subset of ${\mathcal W}^2({\mathbb{S}}^2).$\par 
\indent This discussion can be simplified in the case $\Theta (r)=r^{1/2}$, since we have a quadratic expression $\rho \Vert\nabla (\Theta_1\circ\rho )\Vert^2=(3/4)^2 \Vert \nabla\rho\Vert^2$. The corresponding Dirichlet form is weakly lower semicontinuous or equivalently closeable in $L^2({\mathbb{S}}^2)$, and satisfies a spectral gap inequality (\ref{Spec}).\par 
\indent We create an approximate solution to the  (\ref{continuity}) in ${\mathcal M}(\delta, K,\tau )$ by the previous results. As in Theorem \ref{Theorem2} we update the frame $[x_0;v_0;x_0\times v_0]$ to $[x_h;v_h;x_h\times v_h]$ and $(\rho_0,q_0)$ to $(\rho_h, q_h)$, where $q_h\in H^1({\mathbb{S}}^2)$ and $\rho_h\in {\mathcal M}_h(\delta , K,\tau )$. The map $x_0\mapsto X_h$ gives a bijection ${\mathbb{S}}^2\mapsto  {\mathbb{S}}^2.$
By repeating the stages 1-3 of the approximation process, we can build frames $[X_{jh}; V_{jh}; X_{jh}\times V_{jh}]$ and corresponding $(\rho_{jh}, q_{jh})\in {\mathcal M}(\delta, K)\times H^1({\mathbb{S}}^2)$, so that $x_0\mapsto x_{jh}$ induces $\rho_{jh}$ from $\rho_0$ for $j=1, 2,\dots , \lfloor \tau/h\rfloor$. We join these points by polygonal paths  to give a continuous function 
$$T{\mathbb{S}}^2\times [0, \tau ]\rightarrow T{\mathbb{S}}^2\times {\mathcal M}(\delta ,K, \tau ): ((x,v,t)\mapsto (X(x,v,t),V(x,v,t),p_t(X(x,v,t))$$
such that $X(x,v,0)=x$, $V(x,v,0)=v$ and  $(x,t)\rightarrow \nabla_x (\Theta_1\circ\rho )(x,t)$ is continuous $ {\mathbb{S}}^2\times [0, \tau]\rightarrow T{\mathbb{S}}^2$; the final point is supported by Corollary \ref{Corollary1} and (\ref{Wcontinuity}). By Theorem \ref{Theorem1} and  (\ref{lowerboundonHessian}), the function $\nabla (\Theta_1\circ \rho_h)$ is of bounded variation, in the sense that $D^2(\Theta_1\circ \rho_h)$ is a positive matrix of measures. 
\indent By Corollary \ref{Corollary1}, $T_h^*(x)=\exp_xh\nabla (\Theta_1\circ\rho_h)(x)$ is continuous.  We have $X(x, jh)$ and $\rho (x,jh)$ such that $X(\cdot , jh)\sharp\rho_0(x)=\rho (x,jh)$, such that $x\mapsto \nabla (\Theta_1\circ\rho (\cdot ,jh))(x)$ is continuous. By construction $x\mapsto X(x,jh)$ is bijective.\par
\indent We introduce a $2$-absolutely continuous path $[0,\tau ]\rightarrow {\mathcal W}({\mathbb{S}}^2):$ $u\mapsto\rho_u^{(h)}$ such that $p_{jh}^{(h)}(x)=\rho (x,jh)$ and  $x\mapsto \nabla (\Theta_1\circ p_u^{(h)}) (x))$ is continuous. By Corollary \ref{Corollary2}, we have
\begin{align}\label{Wcontinuity}W_2^2(\rho_h,\rho_0)&\leq 2W_2^2(\rho_0,f)+2W_2^2(\rho_h,f)\nonumber\\
&\leq 2h^2\int_{{\mathbb{S}}^2} \Vert\nabla q_0\Vert^2\rho_0 m(dx)+2h^2\int_{{\mathbb{S}}^2} \Vert\nabla (\Theta_1\circ\rho_h)\Vert^2\rho_h m(dx).\end{align}
\noindent Then, from (\ref{Wcontinuity}), we have a $1/2$-H\"older continuity estimate, where $N=\tau/h -1$, 
\begin{align}\label{Wsum}W_2^2(\rho_{0}, \rho_{\tau})&\leq {\frac{\tau}{h}}\sum_{j=0}^{N}W_2^2(\rho_{(j+1)h},\rho_{jh})\nonumber\\
&\leq  {\frac{\tau}{h}}\sum_{j=0}^{N}2h^2\Bigl( \int_{{\mathbb{S}}^2} \Vert\nabla q_{jh}\Vert^2\rho_{jh} m(dx)+\int_{{\mathbb{S}}^2} \Vert\nabla (\Theta_1\circ\rho_{jh})\Vert^2\rho_{jh} m(dx)\Bigr),\end{align}
which is
$$\leq 2\tau\int_0^\tau  \int_{{\mathbb{S}}^2} \Vert\nabla q_{t}(x)\Vert^2p_{t}^{(h)}(x) m(dx)dt
+2\tau\int_0^\tau  \int_{{\mathbb{S}}^2} \Vert\nabla(\Theta_1\circ p_t^{(h)})(x)\Vert^2p_{t}^{(h)}(x) m(dx)dt,$$
so the process $(X(x,jh))_{j=0}^N$ is of finite quadratic variation, and there exists a $2$-absolutely continuous function $[0, \tau ]\rightarrow {\mathcal W}^2:$ $t\mapsto p_t^{(h)}$. The properties of such curves are established in Theorem 8.3.1 of \cite{bib2}, and the following argument uses the proof from there. \par
\indent By weak compactness, there exists a sequence $h_\nu\rightarrow 0$ and $p\in L^1({\mathbb{S}}^2\times [0, \tau ])$ such that 
\begin{align}\label{plimit}\int_0^\tau  \int_{{\mathbb{S}}^2}\psi (x,t) p_t^{(h_\nu )}(x)m(dx)dt\rightarrow \int_0^\tau  \int_{{\mathbb{S}}^2}\psi (x,t) p(x,t)m(dx)dt\end{align}
\noindent for all $\psi\in  C({\mathbb{S}}^2\times [0, \tau ];{\mathbb{R}})$. 
\indent Let $\psi\in C^1({\mathbb{S}}^2\times [0, \tau ];{\mathbb{R}})$ and let
$$H_s(x,y)=\begin{cases} \Vert\nabla\psi (x,s)\Vert, \qquad x=y;\\ {\frac{\vert \psi (x,s)-\psi (y,s)\vert}{d(x,y)}}, \qquad x\neq y;\end{cases}$$
also let $\pi_{s,t}$ be an optimal transport plan for taking $p^{(h)}_sdm$ to $p^{(h)}_tdm$. We have
\begin{align}{\frac{1}{\vert \varepsilon \vert}}\Bigl\vert \int_{{\mathbb{S}}^2}&\psi (x,s)(p^{(h)}_{s+\varepsilon }(x)-p^{(h)}_{s}(x))m(dx)\Bigr\vert\nonumber\\
&\leq {\frac{1}{\vert \varepsilon \vert}}\Bigl\vert\int\!\!\! \int_{{\mathbb{S}}^2\times {\mathbb{S}}^2}d(x,y)H_s(x,y)\pi_{s,s+\varepsilon }(dxdy)\nonumber\\
&\leq{\frac{1}{\vert \varepsilon \vert}} \Bigl(\int\!\!\! \int_{{\mathbb{S}}^2\times {\mathbb{S}}^2}d(x,y)^2\pi_{s,s+\varepsilon }(dxdy) \Bigr)^{1/2} \Bigl(\int\!\!\!  \int_{{\mathbb{S}}^2\times {\mathbb{S}}^2}H_s(x,y)^2\pi_{s,s+\varepsilon }(dxdy)\Bigr)^{1/2}\nonumber\\
&={\frac{W_2(p^{(h)}_{s+\varepsilon },p^{(h)}_s)}{\vert \varepsilon \vert}}\Bigl( \int_{{\mathbb{S}}^2\times {\mathbb{S}}^2}H_s(x,y)^2\pi_{s,s+\varepsilon }(dxdy)\Bigr)^{1/2};\end{align}
hence
\begin{align}\label{lipdual}\lim\sup_{\varepsilon \rightarrow 0}   {\frac{1}{\vert \varepsilon \vert}}\Bigl\vert &\int_{{\mathbb{S}}^2}\psi (x,s)(p^{(h)}_{s+\varepsilon }(x)-p^{(h)}_{s}(x))m(dx)\Bigr\vert\nonumber\\
&\leq\Bigl(\int_{{\mathbb{S}}^2} \Vert \vec v(x,s)\Vert^2p^{(h)}_s (x)\, m(dx)\Bigr)^{1/2} \Bigl(\int_{{\mathbb{S}}^2}\Vert\nabla\psi (x,s)\Vert^2p^{(h)}_s(x)m(dx)\Bigr)^{1/2}.\end{align}
This gives bounds of the form\par
\begin{align}\label{quadvar}\Bigl\vert\int_0^\tau \int_{{\mathbb{S}}^2}&{\frac{\partial \psi}{\partial t}} (x,t)p_t^{(h)}(x)m(dx)dt\Bigr\vert\nonumber\\
&\leq \Bigl( \int_0^\tau \int_{{\mathbb{S}}^2}\Vert \vec v(x,t)\Vert^2p^{(h)}_t (x) m(dx)dt\Bigr)^{1/2}\Bigl(\int_0^\tau \int_{{\mathbb{S}}^2} \bigl\Vert\nabla\psi (x,t)\bigr\Vert^2m(dx)dt\Bigr)^{1/2}.\end{align}
Here $W_1(p^{(h)}_u, p^{(h)}_v)\leq W_2(p^{(h)}_u,p^{(h)}_v)$, where $W_2(p_u,p_v)\rightarrow 0$ as $u\rightarrow v$, so\par
\noindent $t\mapsto 
\int_{{\mathbb{S}}^2}\psi (y)p^{(h)}_t (y)m(dy)$ is continuous.
Also,
$$h_\nu\sum_{j=1}^{\lfloor 1/h_\nu\rfloor} \int_{{\mathbb{S}}^2} \psi (X_{h_\nu j}(x),h_\nu j)\rho_0(x)m(dx)=h_\nu\sum_{j=1}^{\lfloor 1/h_\nu\rfloor} \int_{{\mathbb{S}}^2}\psi (y, jh_\nu ) \rho_{jh_\nu }(x)m(dx)$$
 converges to the same limit as (\ref{plimit}).\par
\indent As in (iii) of Corollary \ref{Corollary1}, we introduce an approximating family  $( \psi_\varepsilon \ast p^{(h)}_t)_{\varepsilon >0}$ for $p^{(h)}_t$, and define $p_t^{(h, \varepsilon)}=
 \psi_\varepsilon \ast p^{(h)}_t$. We show that there exists $\tau>0$ such that for $v\in T_x{\mathbb{S}}^2$, the initial value problem
\begin{align}\label{modifiedODE}{\frac{d}{dt}}\begin{bmatrix}X\\ V\end{bmatrix}=\begin{bmatrix}0&1\\ -1&0\end{bmatrix}\begin{bmatrix}X\\ V\end{bmatrix}+\begin{bmatrix}0\\ -\nabla(\Theta_1\circ p_t^{(h, \varepsilon)})(X) \end{bmatrix}, \quad  \begin{bmatrix}X(0;x,v)\\ V(0;x,v)\end{bmatrix}=\begin{bmatrix}x\\ v\end{bmatrix}  \end{align}
has a unique solution on $[0,\tau]$.  
 The right-hand side of (\ref{modifiedODE}) is a Lipschitz continuous function of $[X;V]$; indeed we have 
$$\nabla \bigl(\Theta_1\circ p_t^{(h, \varepsilon)}\bigr)(x)=(\Theta'_1 \circ p_t^{(h, \varepsilon)})(x)(\nabla \psi_\varepsilon\ast p^{(h)}_t)(x)$$
and the Lipschitz norm of this is bounded by $\Vert D^2 \bigl(\Theta_1\circ p_t^{(h, \varepsilon)}\bigr)(x)\Vert_{L^\infty_{xt}}$, where
\begin{align}D^2 \bigl(\Theta_1\circ (\psi_\varepsilon \ast p^{(h)}_t)\bigr)(x)&=\Theta'_1 \circ (\psi_\varepsilon \ast p^{(h)}_t)(x)(D^2\psi_\varepsilon\ast p^{(h)}_t)(x)\nonumber\\
&\quad +\Theta''_1 \circ (\psi_\varepsilon \ast p^{(h)}_t)(x)(\nabla \psi_\varepsilon\ast p^{(h)}_t(x)\otimes (\nabla \psi_\varepsilon\ast p^{(h)}_t)(x)\end{align}
\noindent  is bounded for $(x,t)\in {\mathbb{S}}^2\times [0, \tau ]$. By Cauchy-Lipschitz theory, there exists a unique solution to (\ref{modifiedODE}) on $[0, \tau ]$ for all $h,\varepsilon>0$ for all for each $x\in {\mathbb{S}}^2$ and $v\in T_x{\mathbb{S}}^2$, in the guise of the integral equation
\begin{align}\label{integralequation}{}&\begin{bmatrix} X^{(h,\varepsilon )}(t;x,v)\\ V^{(h,\varepsilon )}(t;x,v)\end{bmatrix} =\begin{bmatrix} \cos t&\sin t\\ -\sin t&\cos t\end{bmatrix}\begin{bmatrix} x\\ v\end{bmatrix}\nonumber\\
&+\int_0^t \begin{bmatrix} \cos(t-u)&\sin (t-u)\\ -\sin (t-u)&\cos (t-u)\end{bmatrix}\begin{bmatrix}0\\ -\nabla ( \Theta_1\circ p_u^{(h,\varepsilon )}) (X^{(h,\varepsilon )}(u;x,v))\end{bmatrix}du.\end{align}
\indent There is a natural map $T{\mathbb{S}}^2\rightarrow T{\mathbb{S}}^2$ given by $[x;v]\mapsto [X^{(h, \varepsilon)}(t;x,v); V^{(h, \varepsilon)}(t;x,v)]$, which is bijective for all $0\leq t\leq \tau$ and $\tau>0$ sufficiently small, and the inverse map is obtained by running the differential equation backwards in time. Indeed, from the integral equation (\ref{integralequation}), we have
$$\Bigl\Vert D\begin{bmatrix} X^{(h, \varepsilon)}(t)\\ V^{(h, \varepsilon)}(t)\end{bmatrix}\Bigr\Vert\leq 1+\int_0^t\bigl\Vert D^2(\Theta_1\circ p_u^{(h, \varepsilon)} ) \bigr\Vert_{L^\infty_x} \Bigl\Vert D\begin{bmatrix} X^{(h, \varepsilon)}(u)\\ V^{(h, \varepsilon)}(u)\end{bmatrix}\Bigr\Vert du,$$
so by Gronwall's inequality Theorem 12.3.3 of \cite{bib40}, we have
$$\Bigl\Vert D\begin{bmatrix} X^{(h, \varepsilon)}(t)\\ V^{(h, \varepsilon)}(t)\end{bmatrix}- \begin{bmatrix} \cos t&\sin t\\ -\sin t&\cos t\end{bmatrix} \otimes I_2 \Bigr\Vert\leq \exp\Bigl( \int_0^t \bigl\Vert D^2(\Theta_1\circ p_u^{(h, \varepsilon)} ) \bigr\Vert_{L^\infty_x} du\Bigr)-1.$$
\noindent Hence there exists $\tau>0$ such that  map $[x;v]\mapsto [X^{(h, \varepsilon)}(x,v;t);V^{(h, \varepsilon)}(x,v;t)]$ gives a bijection $T{\mathbb{S}}^2\rightarrow T{\mathbb{S}}^2$ for all $0\leq t\leq \tau$, such that $x\mapsto X^{(h, \varepsilon)}(x, \nabla q_0(x); t)$ is a bijection $ {\mathbb{S}}^2\rightarrow {\mathbb{S}}^2$.\par
\indent Consider the density $\rho^{(h, \varepsilon)} (x,t)$ that is  induced from $\rho_0$ by $x\mapsto X^{(h, \varepsilon )}(t; x, v)$ where $v=\nabla q_0(x)$. (Note the distinction between $\rho^{(h, \varepsilon )} (x,t)$ and $p^{(h, \varepsilon )}_t (x)$.) As in the area formula of page 138 of \cite{bib8}, we have 
$$\rho^{(h, \varepsilon)} (X^{(h, \varepsilon )}(x,t),t)\vert\det DX^{(h, \varepsilon )}(x,t)\vert=\rho_0(x)\sharp \{ z: X^{(h, \varepsilon )}(x,t)=X^{(h, \varepsilon )}(z,t)\},$$
 so $\rho^{(h, \varepsilon)} (y,t)=\sum_{x:X^{(h, \varepsilon )}(x,t)=y}\rho_0(x)/\vert\det DX^{(h, \varepsilon )}(x,t)\vert$
and we introduce the push forward of ${\frac{dX}{dt}}\rho_0(x)$ by 
$$\vec v^{(h, \varepsilon )}(y,t) \rho^{(h, \varepsilon)} (y,t)= \sum_{x: X^{(h, \varepsilon )}(x,t)=y} {\frac{\rho_0(x)}{\vert \det DX^{(h, \varepsilon)}(x,t)\vert}}{\frac{dX^{(h, \varepsilon )}}{dt}}(x,t)$$ 
so $\vec v^{(h, \varepsilon )}(y,t)$ is the average of the Lagrangian velocities of trajectories that pass through a specific point $y$ at the same time $t$; the sum is finite since $\rho^{(h, \varepsilon)}( y,t)$ is bounded. By construction $x\mapsto X(x,jh)$ is bijective, and by the preceding analysis of the ODE,  $x\mapsto X^{(h, \varepsilon )}(x,t)$ is bijective and $X^{(h, \varepsilon ), *}(X^{(h, \varepsilon )}(x,t),t)=x$ and $X^{(h, \varepsilon )}(X^{(h, \varepsilon ),*}(x,t),t)=x$, then  $\vec v(t,y)={\frac{\partial X^{(h, \varepsilon )}}{\partial t}}\circ X^{(h, \varepsilon ),*}(y,t)$ satisfies
\begin{align}\label{intweakcont}-\int_{{\mathbb{S}}^2} \psi (x,0)\rho_0(x)m(dx)&=\int_0^\tau \int_{{\mathbb{S}}^2} {\frac{\partial\psi (y,t)}{\partial t}}\rho^{(h, \varepsilon)}(y, t)m(dy)dt
\nonumber\\
&\quad +\int_0^\tau \int_{{\mathbb{S}}^2}\vec v^{(h, \varepsilon )}(y,t)\cdot \nabla\psi (y,t)\rho^{(h, \varepsilon)} (y,t)m(dy)dt;\end{align}
so at time $t$ we interpret $\vec v^{(h, \varepsilon)} (t,y)$ as the Eulerian velocity vector field at $y$, while $V^{(h, \varepsilon)} (x,t)$ is the Lagrangian velocity with label $x$. 
Hence we have the continuity equation
$${\frac{\partial \rho^{(h, \varepsilon)} (x,t)}{\partial t}}+\nabla\cdot \bigl(\vec v^{(h, \varepsilon )}(t,x)\rho^{(h, \varepsilon)} (x,t)\bigr)=0$$ 
in the weak sense, and we proceed to obtain a metric version. We have
$$\int_{{\mathbb{S}}^2}\Vert \vec v^{(h, \varepsilon )}(t,y)\Vert^2\rho^{(h, \varepsilon)} (y,t)m(dy)=\int_{{\mathbb{S}}^2}\Vert \vec v^{(h, \varepsilon )}(t,X^{(h, \varepsilon )}(x,t))\Vert^2\rho_0 (x)m(dx).$$
\indent Next we compare $\rho_{jh}$ from the approximating path with the density $\rho^{(h, \varepsilon)} (x,t)$ that is  induced by from $\rho_0$ by $x\mapsto X^{(h, \varepsilon )}(t; x, v)$ where $v=\nabla q_0(x)$; we write $ 
X^{(h, \varepsilon )}(t; x)=X^{(h, \varepsilon )}(t; x, v)$ for $v=\nabla q_0(x)$. Suppose that $h=\tau /N$ for some integer $N$, and let  $t_0=0$ and $t_{j+1}=t_j+h$. 
Also for all $\psi\in C({\mathbb{S}}^2; {\mathbb R})$ and $t_j=jh$, we take 
\begin{align}\bigl\vert \int_{{\mathbb{S}}^2}&\psi (y)\rho^{(h, \varepsilon)} (y,t_j)m(dy)- \int_{{\mathbb{S}}^2}\psi (y)\rho_{jh} (y)m(dy)\Bigr\vert\nonumber\\
 &=\Bigl\vert \int_{{\mathbb{S}}^2}\bigl( \psi (X^{(h, \varepsilon )}(x,t_j))-\psi (X_{jh}(x))\bigr)\rho_{0}(x)m(dx)\Bigr\vert\nonumber\\
&\leq \Vert\nabla\psi\Vert_{L^\infty} \int_{{\mathbb{S}}^2}\Vert X^{(h, \varepsilon )}(x,t_j)-X_{jh}(x)\Vert\rho_{0}(x)m(dx), \end{align}
so
$$W_1(\rho^{(h, \varepsilon)} (\cdot , t_j), \rho_{jh})\leq  \int_{{\mathbb{S}}^2}\Vert X^{(h, \varepsilon )}(x,t_j)-X_{jh}(x)\Vert\rho_{0}(x)m(dx).$$
The measures $\rho^{(h, \varepsilon )}(x,t_j)$ satisfy a weaker variant of (\ref{Wsum}), so $t\mapsto \rho^{(h, \varepsilon )}(\cdot ,t)$ is $2$-absolutely continuous $[0, \tau ]\rightarrow {\mathcal W}^1( {\mathbb{S}}^2).$ From the integral equation (\ref{integralequation}), we have 
\begin{align}\sum_{j=1}^N {\frac{\Bigl\Vert \begin{bmatrix} X^{(h, \varepsilon )}(x, t_{j+1})-X^{(h, \varepsilon )}(x,t_j)\\ V^{(h, \varepsilon )}(x,t_{j+1})-V^{(h, \varepsilon )}(x,t_j)\end{bmatrix}\Bigr\Vert^2}{t_{j+1}-t_j}}&\leq 2\sum_{j=1}^N (t_{j+1}-t_j) \Bigl\Vert \begin{bmatrix} X^{(h, \varepsilon )}(x,t_j)\\ V^{(h, \varepsilon )}(x,t_j)\end{bmatrix}\Bigr\Vert^2\nonumber\\
&+2\int_0^\tau \Vert\nabla (\Theta_1\circ p^{(h, \varepsilon )}_u )(X^{(h, \varepsilon )}(x,u))\Vert^2 du\end{align}
 so  integrating this against $\rho_0(x)m(dx)$, we obtain
\begin{align}\sum_{j=1}^N {\frac{W_1(\rho^{(h,\varepsilon )}(\cdot , t_{j+1}), \rho^{(h,\varepsilon )}(\cdot , t_{j}))^2}{t_{j+1}-t_j}}&\leq 2\sum_{j=1}^N (t_{j+1}-t_j) \int_{{\mathbb{S}}^2}\Bigl\Vert \begin{bmatrix} X^{(h, \varepsilon )}(x,t_j)\\ V^{(h, \varepsilon )}(x,t_j)\end{bmatrix}\Bigr\Vert^2\rho_0(x)m(dx)\nonumber\\
&+2\int_0^\tau\int_{{\mathbb{S}}^2}\Vert\nabla(\Theta_1\circ p^{(h, \varepsilon )}_u )(x)\Vert^2\rho^{(h, \varepsilon )}(x,u)m(dx)du\end{align}
\indent The final step is to let $h,\varepsilon\rightarrow 0+$. The family of functions $[X^{(h, \varepsilon )}; V^{(h, \varepsilon )}]$ is uniformly equicontinuous, so by Arzela--Ascoli's theorem, there exists a sequence $h_\nu\rightarrow 0$ and $\varepsilon_\nu\rightarrow 0$ such that   $[X^{(h_\nu, \varepsilon_\nu )};V^{(h_\nu, \varepsilon_\nu )}]\rightarrow [X;V]$ uniformly on ${\mathbb {S}}^2\times [0, \tau ]$. Let $\rho (x,t)$ be the probability density that is induced from $\rho_0$ by $x\mapsto X(x,t)$. 
We have 
\begin{align} \Bigl\vert\int_{{\mathbb{S}}^2}&\psi (X(x,t))\rho_0(x)m(dx)-\int_{{\mathbb{S}}^2}\psi (X^{(h_\nu, \varepsilon_\nu )}(x,t)\rho_0(x)m(dx)\Bigr\vert\nonumber\\
&\leq \Vert\nabla\psi\Vert_{L^\infty} \int_{{\mathbb{S}}^2}\Vert X^{(h_\nu, \varepsilon_\nu)}(x,t)-X(x,t)\Vert\rho_{0}(x)m(dx), \end{align}
so
$$W_1(\rho (\cdot ,t),\rho^{(h_\nu, \varepsilon_\nu)}(\cdot , t))\leq \int_{{\mathbb{S}}^2}\Vert X^{(h_\nu, \varepsilon_\nu)}(x,t)-X(x,t)\Vert\rho_{0}(x)m(dx)$$
where the right-hand side goes to $0$ as $(h_\nu, \varepsilon_\nu)\rightarrow (0,0).$\par
We can take the limit of (\ref{intweakcont}), and deduce that the weak continuity equation holds with velocity $\vec v(x,t)$ and density $\rho (x,t)$. We have
\begin{align}\int_0^\tau  \int_{{\mathbb{S}}^2}\psi (x,t) \rho^{(h_\nu ,\varepsilon_\nu )}(x,t)m(dx)dt&=\int_0^\tau  \int_{{\mathbb{S}}^2}\psi (X^{(h_\nu ,\varepsilon_\nu )}(x,t),t)\rho_0(x)m(dx)dt\nonumber\\
&\rightarrow \int_0^\tau  \int_{{\mathbb{S}}^2}\psi (x,t) \rho (x,t)m(dx)dt\end{align}
\noindent for all $\psi\in  C({\mathbb{S}}^2\times [0, \tau ];{\mathbb{R}})$, so $p(x,t)=\rho (x,t)$ almost everywhere, at least for the sequence $h_\nu\rightarrow 0$. Finally, one can easily check that $\delta\leq \rho (x,t)\leq 1/\delta$ for almost all $(x,t)\in {\mathbb{S}}^2\times [0, \tau ]$.
\end{proof} 
\subsection{The acceleration equation}
The function $T{\mathbb{S}}^2\rightarrow {\mathbb{R}}^4:$ $[X;V]\mapsto [V; -X-\nabla (\Theta_1\circ \rho_u)]$ on the right-hand side of (\ref{ODE}) has derivative
$$\begin{bmatrix} 0& -I-D^2(\Theta_1\circ \rho_u )\\ I&0\end{bmatrix},$$
where $D^2(\Theta_1\circ \rho )$ is a matrix of measures by (\ref{lowerboundonHessian}), and the diagonal terms are evidently zero, so the function is of bounded variation. In the following result, we hypothesize that the family of measures $D^2(\Theta_1\circ\rho^{(h, \varepsilon )}_t)$ is $L^2$. As in Ambrosio's stability Theorem 6.3 \cite{bib37}  for Lagrangian flows, we impose additional assumptions on the vector fields in order to control the flow on the measures that the differential equation generates. \par 

\begin{prop}\label{proposition12} Suppose that there exist $\varepsilon_0,h_0>0$ such that the solution of (\ref{modifiedODE}) gives a differentiabe function $\vec v^{(h, \varepsilon )}: {\mathbb{S}}\times [0, \tau ]\rightarrow {\mathbb R}^2$ such that 
$$\int_0^\tau \int_{{\mathbb{S}}^2}\bigl\Vert D_y\bigl(\vec v^{(h,\varepsilon)}(y,t)\bigr)\bigr\Vert^2\rho^{(h, \varepsilon)}(y,t)m(dy)dt$$
are uniformly bounded for $h_0>h>0$ and $\varepsilon_0>\varepsilon>0$. \par
\indent (i) Then there exists a weak solution to the Euler equations.\par
\indent (ii) Suppose moreover that $\Vert D^2(\Theta_1\circ p_u^{(h,\varepsilon)})(x)\Vert\leq K$ for all $x\in {\mathbb{S}}^2$ and $h_0>h>0$ and $\varepsilon_0>\varepsilon>0$. Then as $(h_k, \varepsilon_k)\rightarrow (0,0)$, the approximate solution $p_t^{(h_k, \varepsilon_k)}$ converges to the weak solution $\rho (x,t)$ in ${\mathcal{W}}^2( {\mathbb{S}}^2).$\end{prop}
\begin{proof} The acceleration equation leads to 
\begin{align}\label{weakacceleration}\int_0^\tau\int_{{\mathbb{S}}^2}&\phi (X^{(h,\varepsilon)}(x,t),t)\cdot {\frac{\partial V^{(h,\varepsilon)}(x,t)}{\partial t}}\rho_0(x)m(dx)dt\nonumber\\
&= \int_0^\tau\int_{{\mathbb{S}}^2}\Bigl(-\phi (X^{(h,\varepsilon)}(x,t),t)\cdot
X^{(h,\varepsilon )}(x,t)\nonumber\\
&\quad -\phi (X^{(h,\varepsilon)}(x,t),t) \cdot\nabla (\Theta_1\circ p_t^{(
h, \varepsilon)})(X^{(h,\varepsilon)}(x,t)\Bigr)\rho_0(x)m(dx)dt;\end{align} 
for all  $\phi\in  C^1({\mathbb{S}}^2\times [0, \tau ];{\mathbb{R}})$, which gives 
\begin{align}\label{weakacceleration}\int_0^\tau\int_{{\mathbb{S}}^2}&\phi (y,t)\cdot {\frac{\partial \vec v^{(h,\varepsilon)}(y,t)}{\partial t}}\rho ^{(h, \varepsilon)}(y,t)m(dy)dt\nonumber\\
&= \int_0^\tau\int_{{\mathbb{S}}^2}\Bigl(-\phi (y,t)\cdot y-\phi (y,t) \cdot\nabla (\Theta_1\circ p_t^{(h, \varepsilon)})(y)\Bigr)\rho^{(h, \varepsilon)}(y,t)m(dy)dt;\end{align} 
\noindent which is bounded above by
\begin{align} {}&\int_0^\tau\int_{{\mathbb{S}}^2}\Vert\phi (y,t)\Vert\rho ^{(h, \varepsilon)}(y,t)m(dy)dt+\Bigl( \int_0^\tau\int_{{\mathbb{S}}^2}\Vert\phi (y,t)\Vert^2\rho^{(h, \varepsilon)}(y,t)m(dy)dt\Bigr)^{1/2}\nonumber\\
&\times \Bigl( \int_0^\tau\int_{{\mathbb{S}}^2}\Vert\nabla (\Theta_1\circ p_t^{(h, \varepsilon)}) (y)\Vert^2\rho ^{(h,\varepsilon)}(y,t)m(dy)dt\Bigr)^{1/2}.\end{align}

\indent  Now the left-hand side of (\ref{weakacceleration}) leads to
\begin{align}\label{split}\int_0^\tau \int_{{\mathbb{S}}^2}&\phi (y,t)\cdot {\frac{\partial \vec v^{(h,\varepsilon)}(y,t)}{\partial t}}\rho^{(h, \varepsilon)}(y,t)m(dy)dt\nonumber\\
&=-\int_0^\tau \int_{{\mathbb{S}}^2}
{\frac{\partial \phi (y,t)}{\partial t}}\cdot  \vec v^{(h,\varepsilon)}(y,t)\rho^{(h, \varepsilon)}(y,t)m(dy)dt\nonumber\\
&\quad-\int_0^\tau \int_{{\mathbb{S}}^2}\phi (y,t)\cdot  \vec v^{(h,\varepsilon)}(y,t){\frac{\partial\rho^{(h, \varepsilon)}(y,t)}{\partial t}}m(dy)dt,\end{align}
in which the first integral in (\ref{split}) is bounded in modulus by 
\begin{align}\leq \Bigl(\int_0^\tau &\int_{{\mathbb{S}}^2}\Bigl\Vert {\frac{\partial \phi (y,t)}{\partial t}}\Bigr\Vert^2\rho^{(h, \varepsilon)}(y,t)m(dy)dt\Bigr)^{1/2}\Bigl(\int_0^\tau \int_{{\mathbb{S}}^2}\Vert \vec v^{(h,\varepsilon )}(y,t)\Vert^2\rho^{(h, \varepsilon)}(y,t)m(dy)dt\Bigr)^{1/2};\end{align}
while we use the continuity equation to replace the second integral in (\ref{split}) by
\begin{align}\int_0^\tau \int_{{\mathbb{S}}^2}&\phi (y,t)\cdot  \vec v^{(h,\varepsilon)}(y,t)\nabla \cdot\bigl( \vec v^{(h,\varepsilon)}(y,t)\rho^{(h, \varepsilon)}(y,t)\bigr)m(dy)dt\nonumber\\
&=-\int_0^\tau \int_{{\mathbb{S}}^2}\nabla\bigl(\phi (y,t)\cdot  \vec v^{(h,\varepsilon)}(y,t)\bigr)\cdot\bigl( \vec v^{(h,\varepsilon)}(y,t)\rho^{(h, \varepsilon)}(y,t)\bigr)m(dy)dt\nonumber\\
&\leq\Bigl( \int_0^\tau \int_{{\mathbb{S}}^2}\bigl\Vert\nabla\bigl(\phi (y,t)\cdot  \vec v^{(h,\varepsilon)}(y,t)\bigr)\bigr\Vert^2\rho^{(h, \varepsilon)}(y,t)m(dy)dt\Bigr)^{1/2}\nonumber\\
&\quad\times \Bigl( \int_0^\tau \int_{{\mathbb{S}}^2}\bigl\Vert \vec v^{(h,\varepsilon)}(y,t)\bigr\Vert^2\rho^{(h, \varepsilon)}(y,t)m(dy)dt\Bigr)^{1/2},\end{align} 
\noindent where the last step is in accord with (\ref{quadvar}). Hence both sides of (\ref{weakacceleration}) define bounded linear functionals on $C^1({\mathbb{S}}^2\times [0, \tau ];{\mathbb{R}})$, and we obtain a weak solution in the limit as $(h_\nu , \varepsilon_\nu )\rightarrow (0,0)$ through some subsequence.\par
\indent (ii) We introduce
$$\chi (x,t)=\Bigl\Vert \begin{bmatrix} X^{(h_1, \varepsilon_1)}(x,v,t)-X^{(h_2, \varepsilon_2)}(x,v,t)\\  V^{(h_1, \varepsilon_1)}(x,v,t) -V^{(h_2, \varepsilon_2)}(x,v,t)\end{bmatrix}\Bigr\Vert$$
and 
\begin{align}\label{eta}\eta (x,t)=\int_0^t \bigl\Vert \nabla (\Theta_1\circ p_u^{(h_1, \varepsilon_1 )})(X^{(h_1,\varepsilon_1)}(x,v,u)- \nabla (\Theta_1\circ p_u^{(h_2, \varepsilon_2 )})(X^{(h_1,\varepsilon_1)}(x,v,u)\bigr\Vert du,\end{align}
so that by (\ref{integralequation}), we have
$$\chi (x,t)\leq \eta (x,t)+\int_0^t \bigl\Vert D^2(\Theta_1\circ p_u^{(h_1, \varepsilon_1)} )(y)\bigr\Vert_{L^\infty_y}\chi (x,u)du.$$
Then by Gronwall's inequality Theorem 12.3.3 of \cite{bib40}, we deduce
\begin{align}\label{Gron}\chi &(x,t)\leq \eta(x,t)\nonumber\\
&+\int_0^t  \bigl\Vert D^2(\Theta_1\circ p_s^{(h_1, \varepsilon_1)} )(y) \bigr\Vert_{L^\infty_y}\exp\Bigl( \int_s^t\bigl\Vert D^2(\Theta_1\circ p_u^{(h_1, \varepsilon_1)} )(y) \bigr\Vert_{L^\infty_y}du\Bigr) \eta (s,x)ds,\end{align}
where
\begin{align}\label{intchi}W_1\bigl(\rho^{(h_1, \varepsilon_1)}(\cdot , t), \rho^{(h_2, \varepsilon_2)}(\cdot , t)\bigr)\leq \int_{{\mathbb{S}}^2}\chi (x,t)\rho_0(x)m(dx)\end{align}
and 
\begin{align}\label{inteta}&\int_{{\mathbb{S}}^2}\eta (x,t)\rho_0(x)m(dx)\nonumber\\
&=\int_0^t \int_{{\mathbb{S}}^2}\bigl\Vert \nabla (\Theta_1\circ p_u^{(h_1, \varepsilon_1 )})(y)- \nabla (\Theta_1\circ p_u^{(h_2, \varepsilon_2 )})(y)\bigr\Vert\rho^{(h_1, \varepsilon_1)}(y,u)m(dy)du.\end{align} 
\noindent For $\varepsilon >0$, we deduce that $W_1((\rho^{(h_1, \varepsilon )}(\cdot , t), \rho^{(h_2, \varepsilon )}(\cdot , t))\rightarrow 0$ as $h_1,h_2\rightarrow 0$.\par
\indent We have $\nabla p_{t}^{(h, \varepsilon)}=
 \nabla \psi_\varepsilon \ast p^{(h)}_{t}\rightarrow \nabla \rho_t$ in $L^2$ where $t=hj$ as $h, \varepsilon\rightarrow 0$, so $\nabla ( \Theta_1\circ p_{t}^{(h_1, \varepsilon_1)}(y)-  \Theta_1\circ p_{t}^{(h_2, \varepsilon_2)}(y))\rightarrow 0$ in $L^2$ as $(h_1, \varepsilon_1), (h_2, \varepsilon_2)\rightarrow (0,0)$. We deduce that (\ref{inteta}) converges to $0$, hence by (\ref{Gron}) and (\ref{intchi}), $\rho^{(h_1, \varepsilon_1)}(\cdot , t)$ converges to $\rho (\cdot ,t)$.
\end{proof} 
\begin{rem}\label{remark4} (i) In their solution of the compressible semigeostrophic equations in dual space, Cullen, Gilbert and Pelloni \cite{bib36} (6.2) use a continuity equation in which the velocity vector field is divergence free, and they can therefore exploit directly the results of Ambrosio \cite{bib37} on vector fields of bounded variation. In Propositions \ref{proposition11} and \ref{proposition12}, we have a vector field on $T{\mathbb{S}}^2$ which is given by velocity and acceleration on ${\mathbb{S}}$, and have the additional complication that the density appears in both differential equations.\par
\indent (ii) The estimates of Corollary \ref{Corollary2} seem too weak to force convergence of the $\rho_{hj}$ in $L^p$ spaces; see page 1067 of \cite{bib17}.\end{rem}

\end{document}